\documentclass[10pt,a4paper]{article}
\usepackage{typearea}
\usepackage{amsfonts}
\usepackage{amssymb}
\usepackage{amsthm}
\usepackage{amscd}
\usepackage[centertags]{amsmath}
\usepackage{eucal}
\usepackage{eufrak}
\usepackage{epsfig}
\usepackage[matrix,arrow,curve]{xy}
\usepackage[small,nohug,heads=vee]{diagrams}
\diagramstyle[labelstyle=\scriptstyle]
\CompileMatrices
\typearea{12}
\author{Andreas Klein}
\title{Hamiltonian spectral invariants, symplectic spinors and Frobenius structures I}

\newcommand{\R}{\mathbb{R}}

\newcommand{\Z}{\mathbb{Z}}

\newcommand{\C}{\mathbb{C}}
\newcommand{\N}{\mathbb{N}}

\newtheorem{theorem}{Theorem}[section]
\newtheorem{Def}[theorem]{Definition}
\newtheorem{prop}[theorem]{Proposition}
\newtheorem{lemma}[theorem]{Lemma}
\newtheorem{folg}[theorem]{Corollary}
\newtheorem{propdef}[theorem]{Proposition/Definition}

\newtheorem{conj}[theorem]{Conjecture}

\begin{document}
\maketitle

\begin{abstract} This is the first of two articles aiming to introduce symplectic spinors into the field of symplectic topology and the subject of Frobenius structures. After exhibiting a (tentative) axiomating setting for Frobenius structures resp. 'Higgs pairs' in the context of symplectic spinors, we present immediate observations concerning a local Schroedinger equation, the first structure connection and the existence of 'spectrum', its topological interpretation and its connection to 'formality' which are valid for the case of standard Frobenius structures. We give a classification of the irreducibles and the indecomposables of the latter in terms of certain $U(n)$-reductions of the $G$-extension of the metaplectic frame bundle and a certain connection on it, where $G$ is the semi-direct product of the metaplectic group and the Heisenberg group, while the indecomposable case involves in addition the combinatorial structure of the eigenstates of the $n$-dimensional harmonic oscillator. In the second part, we associate an irreducible Frobenius structure to any Hamiltonian diffeomorphism $\Phi$ on a cotangent bundle $T^*M$ by letting elements of $T(T^*M)$ act on a line bundle $E$ on $T^*M$ spanned by 'coherent states'. The spectral Lagrangian in $T^*(T^*M)$ associated to this Frobenius structure intersects the zero-section $T^*M$ exactly at the fixed points of $\Phi$. We give lower bounds for the number of fixed points of $\Phi$ by defining a $C^*$-valued function on $T^*\tilde M$ defined by matrix coeficients of the Heisenberg group acting on spinors, where $\tilde M$ is a certain 'complexification' of $M$, whose critical points are in bijection to the fixed points of $\Phi$ resp. to the intersection of the spectral Lagrangian with the zero section $T^*\tilde M$. We discuss how to define spectral invariants in the sense of Viterbo and Oh by lifting the above function to a real-valued function on an appropriate cyclic covering of $T^*\tilde M$ and using minimax-methods for 'half-infinite' chains. 
\end{abstract}

\section{Introduction}\label{intro}

This is the first of a series of articles (\cite{kleinlag}, \cite{kleinham}) which aim to introduce the concept of symplectic spinors (Kostant \cite{kost}) into symplectic topology on one hand and the field of 'Frobenius structures' as introduced by Dubrovin (\cite{dubrovin}) on the other hand. Note that neither the former nor the latter relation is completely new in the mathematical literature, as can be read off for instance from the occurence of symplectic spinors in the literature concerning the Maslov index, semiclassical approximation and geometric quantization (cf. Guillemin, Leray, Crumeyrolle \cite{cru1}, \cite{guillemin}, \cite{leray}) on one hand and the introduction of the 'Geometric Weil representation' by Deligne (letter to Kazhdan, 1982 \cite{deligne}) on the other hand. The latter was reinforced in contemporary discourse in the realms of the Langlands program (cf. V. Lafforgue and Lysenko \cite{lafforgue}) resp. the 'mirror-symmetry'-conjecture first introduced by Kontsevich into mathematics. However, as far as the author knows, there has been no systematic treatment yet to explore the possible role of the notion of symplectic spinors and the Weil representation in 'modern symplectic topology', which can be traced back to pseudoholomorphic curves introduced by Gromov and the advent of infinite dimensional variational methods as introduced by Floer. In between both, one can consider the finite dimensional variational methods of Viterbo (\cite{viterbo}) and their relation to symplectic capacities as introduced by Hofer (\cite{hofer}) and exactly this will be the starting point of this series of papers. The main observation linking symplectic spinors to symplectic topology on one hand and 'Frobenius structures' on the other hand is the existence of a construction which links Lagrangian submanifolds of the cotangent bundle $T^*M$ of a compact Riemanian manifold $M$, intersecting each cotangent fibre transversally, at least outside of their 'caustic' to sums of complex lines, viewed as subbundles in the symplectic spinor bundle, that is we have a correspondence:
\[
{\rm (unramified)\ Lagrangian\ submanifolds\ of}\ T^*M \quad \leftrightarrow\quad {\rm direct\ sums}\ \bigoplus_i (\mathcal{L}_i\rightarrow M)
\]
where $\mathcal{L}_i, i=1,\dots, k$ are a certain set of complex line-subbundles of the symplectic spinor bundle $i^*\mathcal{Q}$ on $T^*M$, pulled back to $M$ where $i:M\hookrightarrow T^*M$ is the inclusion of the zero section. Recall that the symplectic spinor bundle $\mathcal{Q}$ over the symplectic manifold $(T^*M^{n},\omega)$ is the bundle associated to a certain connected $2$-fold covering of the principal bundle of symplectic frames, called a metaplectic structure, by the Shale-Weil-representation of the connected $2$-fold cover of the symplectic group acting as intertwining operators for the Schroedinger representation $\rho$ of the Heisenberg group $H_n$ on $L^2(\mathbb{R}^n)$. Metaplectic structures exist under relatively mild conditions on $M$, that is if $c_1(T^*M)=0\ {\rm mod}\ 2$. Note that each branch of the Lagrangian submanifold $\pi:L\subset T^*M\rightarrow M$ covering $M$ gives over any $x \in M$ rise to an element $\psi_{i,x} \in i^*\mathcal{Q}_x\simeq L^2(\mathbb{R}^n)$ by setting
\[
\psi_{i,\lambda,x}(u)=\rho((0,p_i),\lambda)f(u), \quad ((0,p_i),\lambda) \in H_n, \ u\in \mathbb{R}^n.
\]
Here, $p_i \in \mathbb{R}^n$ locally parametrizes the $i$-th branch of $L$, $\lambda \in \mathbb{R}$ (arbitrary at this point) and $f \in L^2(\mathbb{R}^n)$ is the Gaussian, we identify $H_n=\mathbb{R}^{2n} \times \mathbb{R}$. The set $\psi_{i,\lambda, x}, \ x \in M$ defines a smooth complex line bundle $\mathcal{L}_i$ (outside of ramification points) over $M$ since $i^*\mathcal{Q}_x$ allows a reduction to the structure group $O(n)$ (or its two-fold covering) and $\rho$ acts equivariantly w.r.t. to the Shale-Weil-representation. By construction, $k$ equals the local number of branches of $L$. Note that physically, the vectors $\psi_{i,\lambda,x}$ correspond exactly to 'coherent states' of the quantum mechanical Harmonic oscillator. The above correspondence will be called a {\it symplectic Fourier Mukai transformation}. In this and the second paper in this series, we will mostly assume that $\pi$ is of constant non-zero degree (hence surjective) and the set of caustic points ${\rm ker}\ d\pi\cap TL\neq\{0\}$ is empty (note however the second example below Corollary \ref{curvature} where the case ${\rm dim}({\rm ker}\ d\pi\cap TL)=1$ is studied). Under this hypothesis, each branch of the above non-ramified Lagrangian furthermore corresponds to a summand of a certain $\mathbb{C}$-valued function on $M$, namely we pair the above $\psi_{i,\lambda_i, x} \in E= \bigoplus_i^k \mathcal{L}_i\rightarrow M$ over each point $x \in M$ with certain 'elementary vectors' of $i^*\mathcal{Q}$ (cf. \cite{mumford}). Let us assume each fibre $E_x$ carries a lattice $\Gamma_x$ being compatible with $L\cap E_x$ in the sense that $L=p^{-1}(\tilde L)$ for a Lagrangian $\tilde L$ in the torus bundle $p:E\rightarrow E/\Gamma$. Then, by duality, the structure group of $i^*\mathcal{Q}$ is reducible to $O(n)\cap Sp(2n, \mathbb{Z})$. In this situation, the canonical pairing in $i^*\mathcal{Q}$ of the $\psi_{i,\lambda_i, x}$ with another (the globally defined) distinguished vector $e_\mathbb{Z} \in i^*\mathcal{Q}_x$, which can be considered as a sum of delta distributions centered on the integer points of $\mathbb{R}^n$, defines over each point of $M$ a sum of matrix elements which extends to a mapping
\[
\Theta: E\rightarrow \mathbb{C},\quad (x,c) \mapsto \sum_i^k<\psi_{i,\lambda_i, x}, e_\mathbb{Z}>(c)
\]
where we extend over each fibre $\mathcal{L}_{i,x}$ by multiplying the argument of $\rho$ acting on $f$ as well as the argument of $e_\mathbb{Z}$ by an affine-linear polynomial in $c$ ($c=(c_i)_{i=1}^k$ is the complex coordinate on the fibres of $E$, for details see \cite{kleinlag}). In case of {\it exact} $L$, that is, the canonical one-form $\alpha$ on $T^*M$ is exact on $L$, we will fix the above $\lambda_i$ by being the integral of the Poincare-Cartan-form $\alpha_H=\alpha-H_tdt$ along rays emanating from $x$ to the $i$-th branch of $L$, where $H_t$ is defined so that its Hamiltonian flow generates these rays. Choosing an appropriate basis for $i^*T(T^*M)$, each summand of this function, evaluated at $x \in M$, considering $M$ as the zero-section of $E$, can be interpreted as a value of a certain (sum of) theta functions, that is of functions of the form
\[
\theta(z, \Omega)=\sum_{k \in \mathbb{Z}^n}e^{\pi i (k, \Omega k)+2 \pi i(k, z)+i\lambda},
\]
where $\Omega$ is an element of the Siegel upper half space (a symmetric complex $n\times n$-matrix $\Omega$ whose imaginary part is positive definite) and $(\cdot,\cdot)$ denotes the standard sesquilinar form on $\mathbb{C}^n$. Note that in the case the above torus-bundle structure is absent, we will use different distinguished vectors of $i^*\mathcal{Q}_x$ to define $\Theta$, one choice is to replace $e_\mathbb{Z}$ by the Gaussian $f$. The above choice $e_\mathbb{Z}$ in the presence of a transversal Lagrangian $L$ and a compatible lattice $\Gamma$ will be considered as the most fundamental for reasons that will hopefully become clearer in the course of this article and its followers. To summarize the above philosophically, we want to stress that using these constructions, there is a local correspondence between Lagrangian submanifolds and (special values of) theta functions on one hand and complex line bundles over $M$ on the other hand, as long as the latter are spanned by 'coherent states'. For this terminology, see Perelmov (\cite{perelmov}). If $L$ is furthermore exact, then choosing the data as above, $\Theta$, outside of an eventual zero set $S$ (to be interpreted as some sort of theta divisor) defines a generating function $\Theta:E\setminus S\rightarrow \mathbb{C}^*$ for $L$ (generalizing Viterbo's construction) that reproduces $L$ by taking the 'logarithmic derivative' and, lifted to a suitable cyclic covering $\tilde E$ (associated for instance to $\Theta_*:\pi_1(E\setminus S)\rightarrow \pi_1(S^1)$), allows to define spectral invariants in a very similar way, using the Morse theory for Novikov one forms developed by Novikov, Farber, Ranicki and others. The critical points of $\Theta$ then correspond to the intersection points of $L$ with the zero section. Note that $\tilde E$ is a vector bundle over a (non-compact) cyclic covering $\tilde M$, of $M$.\\ 
Finally, since the vectors $\psi_{i,\lambda_i}$ define a non-vanishing section of $E=\bigoplus_i^k  \mathcal{L}_i$ on $M$, symplectic Clifford multipliction on $T^*M$ allows us to define a Frobenius multiplication $\star$ in the sense of Dubrovin \cite{dubrovin} for tangent vectors on $M$, that is for $v \in TM$ we set
\[
\star \in H^0(T^*M\otimes End(E)),\quad v\star \psi_i:=(v-iJv)\cdot\psi_i,
\]
where $\cdot$ denotes symplectic Clifford multiplication over $T^*M$ and $J$ denotes a compatible nearly complex structure on $T(T^*M)$. As it turns out, the $\psi_i$ diagonalize $\star$ and its eigenvalues ($\star$ is semisimple, which is a consequence of our assumption of $L$ being non-ramified), considered as elements of $\Gamma(\Lambda^1(T^*M))$, are precisely the branches of the above Lagrangian submanifold $L$, that is, we recover $L$ as the spectral Lagrangian of $\star$. As a set, this Lagrangian thus identifies with 
\[
L\simeq {\rm Spec}(\frac{{\rm Sym}(TM)}{\mathcal{I}_s}),
\]
where ${\rm Sym}(TM)$ denotes the sheaf of symmetric tensor algebras in the fibres of $TM$ and $\mathcal{I}_s$ is the ideal spanned by the characteristic polynomial $s$ of $\star$, acting on $E$. Note that in appropriate coordinates, $\star$ is pointwise nothing else than the 'creation' operator of the quantum mechanical harmonic oscillator and the 'diagonalizing' vectors are 'coherent states'.\\

In this first article, we will mainly present an axiomatic setting and certain classification results for irreducible resp. indecomposable Frobenius structures arising in the context of symplectic spinors (cf. Definition \ref{frobenius}, Theorem \ref{genclass}, Theorem \ref{genclassN} and Proposition \ref{higgs}). The regular semisimple case describes the situation where the Frobenius multiplication is diagonalizable and the eigenvalues of $\star$ are distinct, in this situation, one can restrict to an examination of irreducible, hence one dimensional, semisimple Frobenius structures $E$ (and their sums). The indecomposable, non-irreducible case typically appears in a situation where on certain subsets of $M$ certain 'directions' in $M$ are distinguished as it is the case of a stratification of $M$ resp. $L$ by smooth (closed) submanifolds, this case will be discussed in Theorem \ref{genclassN}, a typical example is the Frobenius structure associated to a Lagrangian embedding in $T^*M$ with one-dimensional smooth caustic and trivial normal bundle of its Thom-Boardman-strata, cf. the discussion in the second example below Corollary \ref{curvature}.\\
The emphasis of the second part of this article, \cite{kleinham}, will be applications to Hamiltonian systems and their spectral invariants, while we will postpone a closer examination of the above Lagrangian case and its Frobenius structure, i.e. its connection to 'higher Maslov classes' and miniversal deformations of holomorphic functions with isolated singularities to the third article in the series (\cite{kleinlag}). It will turn out that a given Hamiltonian function $H:M\times[0,1]\rightarrow \mathbb{R}$ on a symplectic manifold which is a contangent bundle $(M=T^*N, \omega)$ (we will always assume that the time one map of the corresp. Hamiltonian flow has only non-degenerate fixed points and is of the form $|p|^2$ outside of some compact subset in $T^*N$ containing $N$) also defines a Frobenius structure $\star:TU\rightarrow End(E)$ in analogy to the above, where $E$ is a complex line bundle on a neighbourhood $U$ of the diagonal $\Delta$ in $(M\times M, \omega\oplus\omega)$ so that the corresponding spectral Lagrangian lies in the complexification $(T_{\mathbb{C}}^*U, \omega_{\mathbb{C}})$ and $\pi:L\subset T_{\mathbb{C}}^*U \rightarrow U$ has degree one as well as a $S^1$-valued 'generating function' on $U\subset M\times M$ in the above sense. This function $\Theta$ can be considered to live on $U\subset M \times M$ since $L$ is a section of $\pi:T^*_{\C}U\rightarrow U$, then the critical points of $\Theta$ on $U$ correspond exactly to the fixed points of the time-one map of the Hamiltonian flow on $M\times M$, where one extends the Hamiltonian flow of $H$ to $M\times M$ by taking $\tilde H(x,y)=1/2(H(x)+H(y))$ on $U$ (we will assume that $|d\Theta|\rightarrow \infty$ near the boundary of $U$). Since the critical points of the generating function $\Theta$ on $U$ also correspond to the zeros of the spectral Lagrangian, we have the theorem:
\begin{theorem}\label{theorem1}
A Hamiltonian function $H:M\times[0,1]\rightarrow \mathbb{R}$ on a cotangent bundle $M=T^*N$ as above defines a Frobenius structure $\star:TU \rightarrow End(E)$ over a neighbourhood $U$ of the diagonal of $(M\times M, \omega \oplus \omega)$, $E$ being a complex line bundle over $U$, so that the following discrete subsets in $U$ coincide:
\begin{itemize}
\item the intersection of the spectral Lagrangian $L$ in $T_{\mathbb{C}}^*U$ with the zero section in $T_{\mathbb{C}}^*U$.
\item the fixed points of the time one flow of $\tilde H$ on $U$. 
\item the critical points of the corresponding generating function $\Theta: U\rightarrow \mathbb{C}^*$.
\end{itemize}
These points are in turn in bijective correspondence to the fixed points of the time one flow of $H$ on $M$.
\end{theorem}
Note that the latter correspondence follows by choosing $U$ sufficiently small and altering $\tilde H$ outside $\Delta\subset U$ so that its only fixed points lie on $\Delta$. Note further that we have to pass from $M$ to a neighbourhood of the diagonal $U\subset M\times M$ to identify the critical points of an $S^1$-valued function $\Theta$ with the fixed points of the time-one flow of $H$ for reasons which will become clear in \cite{kleinham} (it is closely connected to the question of finding invariant Lagrangian subspaces for the differential of the time one flow of $H$). A Frobenius structure $E$ and a spectral Lagrangian living in the complex bundle $T_{\mathbb{C}}^*M$ is always associated to $H$ on $M$ alone, but the zeros of the corresponding spectral Lagrangian do not necessarily correspond to the critical points of a function on $M$ given by matrix elements associated to $E$ over $M$ (as opposed to the case of the Frobenius structure associated to a 'real' Lagrangian of degree one as above), while these zeros still coincide with the fixed points of the time one flow of $H$. Alternatively, one can consider a certain 'dual' $E'$ of a given $E$ (cf. Definition \ref{frobenius}) to define a function by matrix elements associated to $E'$ in the sense that its logarithmic derivative gives the spectral Lagrangian of $E$. Note also, that for general $M=T^*N$, we have to embed $N$ into a higher dimensional affine space $A$ using the embedding theorem of Nash and Moser (a certain almost complex structure on $TM$ determining the embedding) and then proceed by pulling back the symplectic spinor bundle over $T^*A\times T^*A$ to $U\subset M\times M$ (cf. \cite{kleinham}). We will give in the second part of this article \cite{kleinham} first a discussion for $N=T^n$, where $T^n$ denotes the flat torus, which requires no such embedding, then $\Theta$ is again determined by special theta values. Note finally that the spectral Lagrangian $L$ in $T^*_{\mathbb{C}}U$ is not connected to the image of the zero section in $T^*N$ under the time one flow of $H$ in an obvious way.\\
To estimate the number of fixed points of the time one flow of $\tilde H$ on $U$, note that the class $\xi=\Theta^*(\frac{dz}{z}) \in H^1(U, \mathbb{Z})$ associated to $\Theta:U\rightarrow \mathbb{C}^*$ defines a local system $\mathcal{L}_{\xi}$ over $U$ by the ring homomorphism
\[
\phi_{\xi}:\mathbb{Z}[\pi]\rightarrow {\bf Nov}(\pi), \quad \phi_{\xi}(g)=t^{<\xi,g>}
\]
where $\pi=\pi_1(U)=\pi_1(M)$ is the fundamental group, $\mathbb{Z}[\pi]$ its group ring, ${\bf Nov}(\pi)$ is the Novikov ring in the indeterminate variable $t$ and $<\xi,g> \in \mathbb{R}$ denotes the evaluation of $\xi$ on the homology class represented by $g \in H_1(U, \mathbb{Z})$. $\mathcal{L}_{\xi}$ is then a left ${\bf Nov}$-module over $U$. Recall that the Novikov ring denotes formal sums 
\[ 
\sum_{i=1}^{\infty}n_it^{\gamma_i},
\]
where $\gamma_i \in \mathbb{R}, \gamma_i\rightarrow -\infty$ and $n_i \in \mathbb{Z}$ are unequal to zero for only a finite number of $i$ obeying $\gamma_i>c$ for any given $c\in \mathbb{R}$. Let $b_i(\xi)$ denote the rank of $H_i(U; \mathcal{L}_{\xi})$ as a module over ${\bf Nov}(\pi)$ and $q_i(\xi)$ the minimal number of generators of its torsion part. Then by the Novikov inequalities resp. their generalizations to manifolds with boundary (cf. Bravermann \cite{bravermann}), Theorem \ref{theorem1} allows to estimate the number of geometrically distinct critical points of $\Theta$ and thus the number of fixed points of $H$ on $M$ by
\begin{folg}
Let $\phi_H$ be the time-one flow of a time-dependent Hamiltonian $H$ on $M$, $n={\rm dim}{M}$ and $\#{\rm Fix}(\phi_H)$ be the number of its fixed points. Then we have the following estimate:
\[
\#{\rm Fix}(\phi_H)\geq \sum_{i=0}^{2n} b_i(\xi)+2 \sum_{i=1}^{2n} q_i(\xi)+ q_0(\xi).
\]
\end{folg}
We assume here that $\Theta$ is modified along a tubular neighbourhood of the boundary $\partial U$ to match the conditions in \cite{bravermann} (which can always be achieved without introducing new critical points). Note that the Novikov numbers $b_i(\xi), q_i(\xi)$ equivalently appear as Betti- resp. torsion numbers of the $\mathbb{Z}[\pi_1(U)]$-module $H_i(\tilde U_{\xi}, \mathbb{Z})$ on the covering $\tilde U_{\xi}$ of $U$ associated to the kernel of the monodromy homomorphism $Per_\xi:\pi_1(U)\rightarrow \mathbb{R}, \ [\gamma]\mapsto <\gamma, \xi>$. Here, $\pi_1(U)$ act as the group deck-transformations on $\tilde U_{\xi}$. We expect to extract further information on the critical points of $\Theta$ by examining the structure of the underlying Morse-Novikov-complex on the chain level more closely. In especially, in the absence of 'homoclinic orbits' estimates involving Lusternik-Schnirelman-like categories of the type introduced in Farber (\cite{farber}) give estimates like the following.
\begin{folg}
Let $\phi_h$ be the time-one flow of a time-dependent Hamiltonian $H$ on $M$ as above and let ${\rm cat}(U, \xi)$ be the category of $U$ with respect to $\xi$ as in introduced in Farber \cite{farber}. Assume that the homology class $[\xi] \in H^1(U, \mathbb{R})$ admits a gradient-like vector field with no homoclinic cycles. Then
\[
\#{\rm Fix}(\phi_H)\geq {\rm cat}(U, \xi).
\]
\end{folg}
Now following the concept of Viterbo \cite{viterbo} and Oh \cite{oh}, we are tempted to define spectral invariants associated to $\Theta$ on $U$ as follows. Denote by $C_*(\tilde U_{\xi})$ the simplicial or cellular chain complex on $\tilde U_{\xi}$, then the Novikov complex $C_*$, generated by the critical points of $\xi$ on $U$ over ${\bf Nov}(\pi)$ is represented as $C_*={\bf Nov}(\pi)\otimes_{\mathbb{Z}[\pi]}C_*(\tilde U_{\xi})$. Let $\Theta_{\xi}:\tilde U_{\xi} \rightarrow \mathbb{R}$ be a primitive of $\xi$ on $\tilde U_{\xi}$. For $\alpha \in C_*$, represent $\alpha =\sum_{i=1}^{\infty}n_{[p,g]} t^{<\xi,g>}$, where $p$ is a critical point of $\Theta$, $g \in \pi$ and $<\xi,g> \in \mathbb{R}$ is the period mapping. We define the level $\lambda_{\xi}(\alpha)$ of $\alpha \in C_*$ as
\[
\lambda_{\xi}(\alpha)=\max_{[p,g]}\{\Theta_{\xi}([p,g]):\ n_{[p,g]}\neq 0\}
\]
Note that $\Theta_{\xi}([p,g])=\Theta(p)+<\xi,g>$ by the definition of the covering $\tilde U_{\xi}$. $\lambda_{\xi}$ defines a filtration on $C_*$ by considering $C_*^\lambda$ as the span of all chains $\alpha$ so that $\lambda_{\xi}(\alpha)\leq\lambda$. There is a natural inclusion $i_\lambda:C_*^\lambda\rightarrow C_*$ and an associated map on $H_*(\tilde U_{\xi}, \mathbb{Z})$. Then we define for any $a \in H_*(\tilde U, \mathbb{Z})$:
\[
\rho(H, a)=\inf_{\alpha; (i_\lambda)[\alpha]=a} \lambda_{\xi}(\alpha).
\]
Note that for $\rho(H, a)$ be finite, necessarily $a \neq 0$, so unless we guarantee the existence of some non-zero homology class $a$ in $H_*(\tilde U_{\xi}, \mathbb{Z})$, we cannot prove the finiteness of $\rho(H, a)$. However, we will prove in the second article of this series the following finiteness, spectrality and $C_0$-continuity-property, further investigations and applications of this spectral invariant are postponed to subsequent publications.
\begin{theorem}
Assume there is a non-zero, non-torsion element in $H_*(\tilde U_{\xi}, \mathbb{Z})$ being a module over ${\bf Nov}(\pi)$. Then $\rho(H, a)$ is finite and a critical value of $\Theta_{\xi}$ for any $0\neq a \in H_*(\tilde U_{\xi}, \mathbb{Z})$. Furthermore, if $H$ and $F$ are two (time-dependent) Hamiltonian functions, then
\[
|\rho(H, a)-\rho(F, a)|\leq ||H-F||,
\]
where $||\cdot||$ is Hofer's pseudo-norm on $C_0(T^*N\times [0,1])$. I.e., $\rho_a$ mapping $H\mapsto \rho(H, a)$ is $C_0$-continuous.
\end{theorem} 
Note that the construction of such a spectral invariant for a Hamiltonian system on a general cotangent bundle $T^*N$ here goes (potentially) beyond the reach of Viterbo's finite dimensional methods in \cite{viterbo}, which are in the Hamiltonian case only applicable for $T^*N=\mathbb{R}^{2n}$. The proof of the above finiteness and $C_0$-continuity property leans very closely to the existing proofs of Viterbo and Oh in their respective contexts. This is possible since our 'generating function' $\Theta$ can be interpreted as a 'crude version' of Chaperon's method of broken geodesics resp. Conley and Zehnder's proof of the Arnol'd conjecture for flat tori. However, we want to stress that the main objective of this paper was not to give sharper lower bounds for the existence of Hamiltonian fixed points on cotangent bundles, but to show that the notion of Frobenius stuctures and fundamental questions of symplectic topology are very intimately connected. Interpreting $\Theta$ at least for the case of the torus $M=T^*T^n$ as assuming 'special values' of a certain automorphic function following Mumford's remarks \cite{mumford}, the connection given in Theorem \ref{theorem1} between the spectral cover of a Frobenius structure associated to the vector bundle $E$ and the critical points of $\Theta$ should have an interpretation in the realms of the Langlands program as giving some sort of 'characteristic zero' analogy for the correspondence between Galois representations and automorphic representations. In especially, the relation between the two complex line bundles $E$ and $\mathcal{L}_{\xi}$ above deserves a closer examination. To both sides, the 'Galois representation side' (the action of the Hamiltonian flow) and the 'automorphic side' (the gradient like-flow of $\Theta$) one can associate a dynamical zeta-function (cf. Hutchings \cite{hutchings}), both should be in a sense 'dual' to another (see also \cite{deligneflicker}).\\ 
We finally formulate a conjecture which connects the above spectral invariants (if nontrivial) with the 'eigenvalues' of the covariant derivative of the Euler vector field $X_E$ associated to the Frobenius structure $\star:TU \rightarrow End(E)$ over $U$ for a non-degenerate Hamiltonian $H$ on $M$. Note that 'eigenvalues' we call here (compare Proposition \ref{spectrum}) the evaluation of the closed part (via Hodge decomposition) of the one form with values in $End(E)$ associated to $\nabla X_E$ on a set of generators of $H_1(U, \mathbb{Z})$, this definition is expected to coincide with the usual definition in the case of Frobenius structures associated to the miniversal deformation of an isolated singularity (\cite{kleinlag}). The non-triviality of such a closed part follows once one assumes $\xi \in H^1(U, \mathbb{R})$ is non-trivial and $H^*(M, \mathbb{C})$ is {\it formal}, that is all higher order cohomology operations vanish. Note further that, at least in the 'flat' case (cf. Definition \ref{frobenius}) our construction of $\star$ should associate a 'variation of Hodge structure' to a given Hamiltonian $H$ on a cotangent bundle by the common scheme (cf. \cite{fernan}) of interpreting Frobenius manifolds in terms of 'variations of Hodge structure' and vice versa. On the other hand, our generating function $\Theta$ should be linked to a 'Gromov-Witten'-type theory and its variation of Hodge structures by selecting topologically 'relevant' coherent subbundles of $i^*\mathcal{Q}$ over $M$ by a Thom-isomorphism and thus defining a Frobenius structure on $H^*(M)$ (cf. a subsequent publication). In any case, we conjecture here, complementing Theorem \ref{theorem1}:
\begin{conj}
The 'eigenvalues' (in the above sense) of $\nabla X_E$ over $U$, that is the spectrum of the Frobenius structure $\star:TU \rightarrow End(E)$ (that is the spectral numbers of the variation of Hodge structures associated to $H$) coincide generically (after eventual affine scaling) with the above spectral numbers $\rho(H, a)$ of $H$, where $a$ ranges over all elements $a \in H_*(\tilde U_{\xi}, \mathbb{Z})$.
\end{conj}
Note that together with Theorem \ref{theorem1} and interpreting our function $\Theta$ as the kernel of an appropriate integral operator and invoking a related trace formula, this conjecture should be interpreted as an analogon of the (conjectural) Hecke eigenvalue/Frobenius eigenvalue correspondence in the (geometric) Langlands program, an analogous result will be examined in  (\cite{kleinlag}).\\
We want to thank the IHES at Bures sur Yvette, where parts of this research was done, for support and kind hospitality. Furthermore, we are in gratitude to Svatopluk Krysl for helpful remarks on an early draft of this paper.

\tableofcontents
 
\section{Symplectic Clifford algebra, Lagrangian relations and Gaussians} \label{lagcoh}

In this section, we will essentially review certain results on Lagrangian relations, the symplectic Clifford algebra and Gaussians \cite{hoermander}, \cite{nerethin}, \cite{habermann} which will suffice to describe the 'semi-simple' Frobenius structures appearing in this article. That semi-simple Frobenius stuctures are in a specific sense characterized by Gaussians or 'coherent states' will be discusssed in \cite{kleinlag}. We will reformulate all results in the language of certain (sub-Lie algebras of) the symplectic Clifford algebra, to be defined now.
\subsection{Symplectic Clifford algebra} 
Let $V$ be a real vector space, $\mathcal{T}(V)$ its tensor algebra and $\omega$ an antisymmetric, non-degenerated bilinear form on $V$. Let $\mathcal{I}(\omega)$ the two sided ideal spanned by 
\begin{equation}\label{comm1}
\{ x\otimes y  -  y \otimes x  -  \omega(x,y) : x,y \in V \} \subset  \mathcal{T}(V)
\end{equation}
Then $\textbf{sCl}(V,\omega)=\mathcal{T}(V)/ \mathcal{I}(\omega)$ is an associative
algebra with over $\mathbb{R}$ mit identity, the symplectic Clifford algebra of $(V,\omega)$. Let $j: \mathcal{T}(V) \rightarrow 
\textbf{sCl}(V,\omega)$ the canonical projection and $i:V \hookrightarrow \mathcal{T}(V)$ the natural embedding of $V$ into its tensor algebra, then the linear mapping $\kappa = j\circ i$ satisfies
\begin{equation}\label{symrel}
\kappa(x)\cdot \kappa(y)  -  \kappa(y)\cdot \kappa(x)  =  \omega(x,y) \cdot 1
\end{equation}
for all $x,y \in V$. Since $\kappa$ is injective, we will regard $V$ as a linear subspace of $\textbf{sCl}(V,\omega)$ in the following and suppress $\kappa$.\\
Let $\textbf{sCl}(\mathbb{R}^{2n}):=\textbf{sCl} (\mathbb{R}^{2n},-\omega_{0})$, where $\omega_{0}$ is the symplectic standard strcuture on $\mathbb{R}^{2n}$. $\textbf{sCl}(\mathbb{R}^{2n})$ becomes, equipped with the commutator, an infinite dimensional  real Lie algebra. Let $a_{1},\ldots, a_{n}, b_{1}, \ldots, b_{n}$ be the elements of the standard basis in $\mathbb{R}^{2n}$, so that
\begin{equation}\label{standardb}
\omega_{0}(a_{i},b_{j}) =  \delta_{ij}, \quad \omega_{0}(a_{i},a_{j})=0, \quad
\omega_{0}(b_{i},b_{j})=0 \quad \textrm{for} \quad i,j = 1,\ldots,n.
\end{equation}
We will in the following look at two sub-algebras of $\textbf{sCl}(\mathbb{R}^{2n})$. The first is the sub-Lie algebra of polynomials in $a_{1},\ldots, a_{n}, b_{1}, \ldots, b_{n}$ of degree $\leq 1$ in $\textbf{sCl}(\mathbb{R}^{2n})$, which defines the Heisenberg-algebra $\mathfrak{h}=\mathbb{R}^{2n} \oplus \mathbb{R}$. For the second, observe that the symmetric homogeneous polynomials of degree $2$ define a sub-Lie algebra of $\textbf{sCl}(\mathbb{R}^{2n})$, which we will call $\mathfrak{a}$ henceforth. Note that $\mathfrak{a}\subset \textbf{sCl}(\mathbb{R}^{2n})$ acts linearly on $\mathbb{R}^{2n}$ by setting 
\[
ad(a)=[a,x],\ a \in \mathfrak{a},\ x \in \mathbb{R}^{2n},
\]
as one can directly verify using the relations (\ref{symrel}), further one has for $x \in
\mathfrak{a}$ und $y,z \in \mathbb{R}^{2n}$ 
\[
\omega_{0}([x,y],z)+\omega_{0}(y,[x,z]) = 0 \,
\]
thus we have a linear map $ad:\mathfrak{a}\rightarrow \mathfrak{sp}(2n,\mathbb{R})$, where $\mathfrak{sp}(2n,\mathbb{R})$ denotes the Lie algebra of the symplectic group $Sp(2n, \R)$, and this map is in fact a Lie algebra- isomorphism, that is we have the following. Set for $B_{jk}$ a $n \times n$-matrix being $1$ at the $jk$-th position ($j$-th line, $k$-th column) and else $0$. Then the matrices $X_{jk}$ with $1 \leq j,k \leq n$, $Y_{jk}$ and $Z_{jk}$ mit $1 \leq j \leq k \leq n$ furnish a basis of the Lie-Algebra $\mathfrak{sp}(2n,\mathbb{R})$: 
\begin{align}
X_{jk} \ =& \ 
\begin{pmatrix}
B_{jk}&0 \\ 0 &-B_{kj}
\end{pmatrix} \nonumber \\[6pt]
Y_{jk} \ =& \
\begin{pmatrix}
0&B_{jk}+B_{kj} \\ 0&0
\end{pmatrix} \nonumber \\[6pt]
Z_{jk} \ =& \ 
\begin{pmatrix}
0&0 \\ B_{jk} + B_{kj}&0
\end{pmatrix} \quad .\nonumber 
\end{align}
\begin{lemma}[\cite{habermann}]\label{liesymp}
The polynomials $a_{j} \cdot a_{k}$ mit $1 \leq j \leq k \leq n$,
$b_{j} \cdot b_{k}$ mit $1 \leq j \leq k \leq n $ und $a_{j} \cdot b_{k}+b_{k}
\cdot a_{j}$ mit $1 \leq j,k \leq n$ span a basis of the Lie algebra $\mathfrak{a}$. Furthermore, the linear map $ad:\mathfrak{a}\rightarrow \mathfrak{sp}(2n,\mathbb{R})$ is a Lie algebra isomorphism, and we have
\begin{align}
ad(a_{j}\cdot a_{k}) \ &= \ -Y_{jk} \\
ad(b_{j}\cdot b_{k}) \ &= \ Z_{jk} \\
ad(a_{j}\cdot b_{k} + b_{k}\cdot a_{j}) \ &= \ 2X_{jk} \ .
\end{align}
\end{lemma}\label{symm}
It is obvious that the defining relations of $\mathfrak{h}\subset\textbf{sCl}(\mathbb{R}^{2n})$ reproduce the quantum mechanical
'Heisenberg commutator relations', thus we have a representation of  $\mathfrak{h}=\mathbb{R}^{2n} \oplus \mathbb{R}$ over
the Schwartz-space $\mathcal{S}(\mathbb{R}^n)\subset L^{2}(\mathbb{R}^n)$ as
\begin{equation}\label{cliffbig}
\begin{split}
1 \in \mathbb{R} \ \quad &\mapsto \quad i \\
a_{j} \in \mathbb{R}^{2n} \quad &\mapsto \quad ix_{j} \\
b_{j} \in \mathbb{R}^{2n} \quad &\mapsto \quad \frac{\partial}{\partial x_{j}}
\qquad \textrm{for} \quad  j = 1, \dots ,n. 
\end{split}
\end{equation}
Here, $i$, $ix_{j}$ as well as $\frac{\partial}{\partial x_{j}}$ act as unbouded operators on the dense domain $\mathcal{S}(\mathbb{R}^n)$ in the Hilbert space $L^{2}(\mathbb{R}^{n})$. Denoting the restriction of the above map to 
$\mathbb{R}^{2n}$ by $\sigma$, we get 'symplectic Clifford multiplication':
\begin{Def}\label{defclif}
Symplectic Clifford multiplication is a map
\[
\begin{split}
\mu: \mathbb{R}^{2n} \times \mathcal{S}(\mathbb{R}^n) \ &\rightarrow \
\mathcal{S}(\mathbb{R}^n) \\
(v,f) \ &\mapsto \ v \cdot f := \mu (v,f) = \sigma (v)f . 
\end{split}
\]
\end{Def}
Indeed, by direct calculation one then concludes:
\begin{folg}\label{heisen} For $v,w \in \mathbb{R}^{2n}$ und $f \in
\mathcal{S}(\mathbb{R}^n)$ we have \begin{equation}\label{cliff}
v \cdot  w \cdot f - w \cdot v \cdot f = -i\omega_{0}(v,w)f. 
\end{equation}
\end{folg}

\subsection{Heisenberg group and metaplectic representation}
Via the exponential map, we can consider the simply connected Lie group associated to $\mathfrak{h}$ and denote it by $H_n$. Then the relations noted in (\ref{heisen}) imply that if writing $H_{n}=\mathbb{R}^{2n} \times \mathbb{R}$ we have 
\[
(v,t)\cdot(w,s)=(v+w,t+s+\frac{1}{2} \omega_{0}(v,w)), \ (v,t),(w,s) \in 
H_{n}=\mathbb{R}^{2n} \times \mathbb{R}.
\] 
We call $H_n$ the $2n+1$-dimensional Heisenberg-group. The theorem of von Stone-Neumann states that there exists up to unitary
equivalence a unique irreducible unitary representation $(\pi,L^{2}(\mathbb{R}^n))$ of $H_{n}$ satisfying
\begin{equation}\label{eig} 
\pi(0,t)=e^{it}id_{L^{2}(\mathbb{R}^n)}.
\end{equation}
Indeed (cf. \cite{lionvergne}) we have for $(v,t) =
((x,y),t) \in  \mathbb{R}^{2n} \times \mathbb{R}$ an explicit irreducible
unitary representation $(\pi,L^{2}(\mathbb{R}^n))$ of $H_n$ satisfying (\ref{eig}) which is given by
\begin{equation}\label{expl}
\left( \pi ((x,y),t)f\right)(z)=e^{i(t+\langle x,z-\frac{1}{2}y
\rangle)}f(z-y) \quad \textrm{for} \ f \in L^{2}(\mathbb{R}^{n}), \ z \in
\mathbb{R}^{n}. \end{equation}
Since it is very illustrative of the implicit presence of 'Lagrangian relations' in our context, we recall the construction of $(\pi,L^{2}(\mathbb{R}^n))$ in {\it loc. cit}. For this observe that for a Lagrangian subspace $L$ of $(\R^{2n}, \omega_{0})$, the group $\mathcal{L}=(L, \mathbb{R}\cdot 1)$ is an abelian subgroup of $H_n$ since $\R\cdot 1$ is the center of $H_n$.
Furthermore, 
\[
f(v,t)=e^{it},\ (v,t) \in \mathcal{L},
\]
is a character on $\mathcal{L}$. Now choosing a Lagrangian decomposition $L\oplus L'=\R^{2n}$ we get an invariant measure on $H_n/\mathcal{L}$ by identifying the latter with $L'$ and using the Euclidean measure on the latter. These ingredients finally define $(\pi,L^{2}(\mathbb{R}^n))$ by the well-known (cf. \cite{lionvergne}) construction of induced representations 
\[
\pi=\pi(L):= {\rm Ind}\uparrow_\mathcal{L}^{H_n}f
\]
and by identifying $L^2$-spaces on $H_n/\mathcal{L}$, $L'$ and $\mathbb{R}^n$, respectively. Recall that $\pi(L)$ consists of the completion of the continuous functions $g$ on $H_n$ satisfying $g(x+l)=f(l)^{-1}g(x), \ l \in \mathcal{L}, \ x \in H_n$ and being square integrable w.r.t. the above measure on $H_n/\mathcal{L}$. Thus we want to stress that, for a given choice of character $f$, the set of unitarily equivalent representations of the Heisenberg group are essentially parameterized by Lagrangian splittings of the form $L\oplus L'=\R^{2n}$, or special cases of {\it Lagrangian relations}. We mention that $\pi$ also reproduces our choice of representation of $\mathfrak{h}$ (restricted to $\R^{2n})$, namely $\sigma$:
\begin{equation}\label{cliffordheis}
d\pi(v)=\sigma(v), \ v \in \R^{2n},
\end{equation}
while of course $d\pi(1)=i$, as in (\ref{cliffbig}). We have an action of $Sp(2n, \R)$ on $H_n$:
\[
\begin{split}
Sp(2n,\mathbb{R}) \times H_{n} \ &\rightarrow \ H_{n} \\
(g,(v,t)) \ &\mapsto \ (gv,t)
\end{split}
\]
Note that $\pi^{g}(v,t)=\pi(gv,t)$ defines an irreducible unitary representation of $H_{n}$ s.t. $\pi^{g}(0,t) =e^{it}id_{L^{2}(\mathbb{R}^n)}$ (amounting to a change of $L'$ above under $g \in Sp(2n,\mathbb{R})$), thus by the above there exists a family of unitary operators $U(g):L^{2}(\mathbb{R}^n) \rightarrow L^{2}(\mathbb{R}^n)$ so that
\[
\pi^{g}=U(g)\circ \pi \circ U(g)^{-1},
\]
and $U(g)$ ist uniquely determined up to multiplication by a complex constant of modulus $1$. By Shale and Weil, $g \in Sp(2n,\mathbb{R}) \mapsto U(g) \in U(L^{2}(\mathbb{R}^{n}))$ defines a projective unitary representation of
$Sp(2n,\mathbb{R})$ lifting to a representation $L:Mp(2n,\mathbb{R}) \rightarrow U(L^{2}(\mathbb{R}^{n}))$ of the (up to isomorphism unique, since $\pi_1({Sp}(2n, \R))=\mathbb{Z}$) connected two-fold covering $\rho:Mp(2n,\R) \rightarrow Sp(2n,\R)$
\[
1 \quad \rightarrow \quad \mathbb{Z}_{2} \quad \rightarrow \quad
Mp(2n,\R) \quad \xrightarrow{\rho}  \quad Sp(2n,\R) \quad
\rightarrow \quad 1,
\]
sarisyfing
\begin{equation}\label{heiequiv}
\pi(\rho(g)h)=L(g)\pi(h)L(g)^{-1} \quad \textrm{for} \ h \in H_{n}, \ g \in Mp(2n,\mathbb{R}).
\end{equation}
The representation $L$ has the following explicit construction on the elements of three generating subgroups of $Mp(2n,\mathbb{R})$, as follows:
\begin{enumerate}
\item Let $g(A)=(det(A)^{\frac{1}{2}},\left(\begin{smallmatrix}A&0
\\ 0&(A^{t})^{-1}\end{smallmatrix} \right))$ where $A \in GL(n,\mathbb{R})$. To fix a root of $det(A)$ defines $g(A)$ as an element in $Mp(2n, \R)$ and we have
\begin{equation}\label{metalinear}
(L(g(A))f)(x)= det(A)^{\frac{1}{2}}f(A^{t}x), \ f \in L^2(\mathbb{R}^n).
\end{equation}
\item Let $B \in M(n,\mathbb{R})$ s.t. $B^{t} = B$, set $t(B)
= \left(\begin{smallmatrix}1&B \\0&1\end{smallmatrix} \right) \in
Sp(2n)$, then the set of these matrices is simply-connected. So $t(B)$ can be considered an element of
$Mp(2n)$, with $t(0)$ being the identity in $Mp(2n)$. Then one has
\begin{equation}\label{meta2}
(L(t(B))f)(x) = e^{-\frac{i}{2}\langle Bx,x\rangle}f(x).
\end{equation}
\item Fixing the root $i^{\frac{1}{2}}$, we can consider
$\tilde \sigma=(i^{\frac{1}{2}},\left(\begin{smallmatrix}0&-1
\\1&0\end{smallmatrix}\right))$ as an element of $Mp(2n)$. Then
\begin{equation}\label{fourier}
(L(\tilde \sigma)f)(x)=(\frac{i}{2\pi})^{\frac{n}{2}}\int_{\mathbb{R}^{n}}e^{i\langle
x,y\rangle}f(y)dy,
\end{equation}
so $L(\tilde \sigma)= i^{\frac{n}{2}}F^{-1}$, where $F$ is the usual Fourier transform.
\end{enumerate}
Inspecting these formulas it is obvious that the metaplectic group $Mp(2n,\R)$ acts bijectively and unitarily on the Schwartz space $\mathcal{S}(\mathbb{R}^{n})$, so its closure extends to $\mathcal{U}(L^2(\mathbb{R}^n))$. We fix the $2$-fold covering $\rho: Mp(2n,\mathbb{R}) \rightarrow Sp(2n,\mathbb{R})$ by demanding 
\[
\rho_{\ast}=ad:\mathfrak{mp}(2n,\mathbb{R}) \rightarrow
\mathfrak{sp}(2n,\mathbb{R})  
\]
to be exactly the algebra-isomorphism $ad$ of Lemma \ref{liesymp}. Since both groups in question are connected, $\rho$ is correctly defined. While by $\cite{wallach1}$ the mapping $L:Mp(2n,\mathbb{R}) \rightarrow U(L^{2}(\mathbb{R}^{n}))$ is not differentiable, we define the notion of a differential of $L$ as follows using the set of 'smooth vectors'. Let $f \in
\mathcal{S}(\mathbb{R}^{n})$. Then $\mathcal{L}^{f}:\mathfrak{mp}(2n,\mathbb{R}) \rightarrow
L^{2}(\mathbb{R}^{n})$ given by
\[
\mathcal{L}^{f}(X)=L(exp(X))f
\]
is (again \cite{wallach1}) a differentiable mapping with image $\mathcal{S}(\mathbb{R}^{n})$. Thus we set $L_{\ast}: \mathfrak{mp}(2n,\mathbb{R}) \mapsto \mathfrak{u}(\mathcal{S}(\mathbb{R}^{n}))$ as
\[
L_{\ast}(X)f=d\mathcal{L}^{f}(X)=\frac{d}{dt}L(exp(tX))f_{\vert t=0}.
\]
We finally have the following. 
\begin{prop}\label{diffmp}
Let $S \in Sp(2n,\mathbb{R})$ and $\hat S \in Mp(2n,\mathbb{R})$ so that $\rho(\hat S)= S$. Then for any $u,v \in \R^n$
\[
(\sigma(Su)+\sigma(Sv))L(\hat S)f=L(\hat S)(\sigma(u)+\sigma(v))f, \ f \in \mathcal{S}(\mathbb{R}^n).
\]
Let $f \in \mathcal{S}(\mathbb{R}^{n})$, then we have for $L_{\ast}: \mathfrak{mp}(2n,\mathbb{R}) \mapsto
\mathfrak{u}(\mathcal{S}(\mathbb{R}^{n}))$:
\begin{equation}
\begin{split}
L_{\ast}(a_{j}\cdot a_{k})(f) \ &= \ ix_{j}x_{k}f \ = \ -ia_{j}\cdot
a_{k}\cdot f \\ 
L_{\ast}(b_{j}\cdot b_{k})(f) \ &= \
-i\frac{\partial^{2}}{\partial x_{j}\partial x_{k}}f \ = \ -ib_{j}\cdot
b_{k}\cdot f \\
L_{\ast}(a_{j}\cdot b_{k} + b_{k}\cdot a_{j})(f) \ &= \
\left(x_{j}\frac{\partial}{\partial x_{k}}+\frac{\partial}{\partial
x_{k}}x_{j}\right)f \ = \  -i(a_{j}\cdot b_{k} + b_{k}\cdot a_{j})\cdot f.
\end{split} 
\end{equation}
\end{prop}
\begin{proof}
The first assertion is proven by differentiating (\ref{heiequiv}) and using the fact that $\omega_0|W=0$. The second assertion is a direct computation and can be found in \cite{habermann}.
\end{proof}

\subsection{Coherent states, positive Lagrangians and commutative algebras}\label{coherent}
Consider again a real symplectic vectorspace $(V,\omega)$ of dimension $2n$ and let $\omega_{\mathbb{C}}$ be the complex bilinear extension of $\omega$ to the complexification $V^{\mathbb{C}}$. Then it is well-known (cf. \cite{mumford}) that the following data are equivalent
\begin{enumerate}
\item a complex structure $J$ on $V$ being compatible with $\omega$, that is $\omega(Jx,Jy)=\omega(x,y)$ for all $x,y\in V$ and $\omega(x,Jx)>0$ for all $x \in V, x\neq 0$.
\item a complex structure $J$ and a positive definite Hermitian form $H$ on $V$ such that ${\rm Im}(H)=\omega$.
\item a totally complex subspace $L\subset V^{\mathbb{C}}$ of (complex) dimension $n$ so that $\omega_{\mathbb{C}}$ vanishes on $L$ and $i\omega_{\mathbb{C}}(x,\overline x)>0$ for all $x \in L$.
\end{enumerate}
Any of these data defines a point in the Siegel space $\mathfrak{h}_V$, i.e. choosing a symplectic basis $e_1,\dots, e_n, f_1, \dots, f_n$ for $\omega$ as above, we get from $L$ a $n\times n$ complex symmetric matrix $T$ so that ${\rm Im}(T)$ is positive definite by requiring $e_i- \sum_jT_{ij}f_j \in L$ (note that $T\in \mathfrak{h}_V$ implies that $T$ invertible). On the other hand, given $J$ as in (1.), $H$ is defined as
\[
H(x,y)=\omega(x,Jy)+i\omega(x,y), \ x,y \in V,
\]
and $L$ is given by the image of the map 
\[
\alpha_J: V \rightarrow V^{\mathbb{C}}, \ \alpha_J(x)=x-iJx.
\]
$Sp(V,\omega)$, the symplectic group, acts on the set of compatible complex structures $\mathcal{J}_\omega\simeq \mathfrak{h}_V$ by conjugation
\[
Sp(V,\omega) \times  \mathcal{J}_\omega \rightarrow \mathcal{J}_\omega,\quad (g,J)\mapsto gJg^{-1}, 
\]
so $\mathcal{J}_\omega\simeq Sp(V,\omega)/U(V, \omega)$, where $U(V, \omega)$ is the unitary group, while the corresponding action of $Sp(V,\omega)$ on $\mathfrak{h}_V$ is given by
\[
(g, T) \mapsto (DT-C)(-BT+A)^{-1}, \quad g=\begin{pmatrix}A & B \\ C & D\end{pmatrix}.
\]
Let now be again $(V, \omega)=(\mathbb{R}^{2n}, \omega_0)$. Fix one $T \in \mathfrak{h}_V$ and consider the function $f_T=e^{\pi i <x, Tx>} \in L^2(\R^n)$, where $<\cdot, \cdot>$ denotes the standard scalar product. Let $J=J_T \in \mathcal{J}_{\omega_0}$ be the element corresponding to $T$ relative to the symplectic standard basis $a_1,\dots, a_n, b_1, \dots, b_n$ in $(\ref{standardb})$ which we will fix henceforth. Then the Lagrangian $L_T\subset V^{\mathbb{C}}$ associated to $T$ is given by the span of $a_i- \sum_jT_{ij}b_j, i \in \{1,\dots,n\}$. We will frequently need the following result:
\begin{theorem}[\cite{mumford}]\label{mumford}
The subspace $\mathbb{C}\cdot f_{T^{-1}}$ is the subspace annihilated by $\sigma\circ \alpha_{J_T}$. Let $g=\left(\begin{smallmatrix}A&B\\C&D \end{smallmatrix}\right) \in Sp(2n, \R)$ and $\hat g \in Mp(2n, \R)$ so that $\rho(\hat g)=g$. Then 
\[
L(\hat g)f_T=c(g,T)f_{g(T)},
\]
where $c(g,T) \in \mathbb{C}^*$ is an appropriate branch of the holomorphic function $[det(-BT+A])^{1/2}]$ on $\mathfrak{h}_V$.
\end{theorem}
\begin{proof}
Note that since ${\rm Im}(T)$ is positive definite, we can solve $y=Tx$ for $x$. Then $L_T$, the locus of $\alpha_{J_T}(x), x \in V^{\C}$ is by the above given by the (complex) span of the
\[
a_i- \sum_jT_{ij}b_j= a_i- \sum_jT_{ij}J_0a_j, \ i \in \{1,\dots,n\},
\]
where $J_0:V\rightarrow V$ is the standard complex structure $J_0=\left(\begin{smallmatrix}0&1\\-1&0 \end{smallmatrix}\right)$. Since $L_T$ is given equivalently by the locus
\[
x - iJ_Tx,\quad x \in V,
\]
the annihilator of $f_{T^{-1}}$ under $\sigma$ is exactly $L_T$ by \cite{mumford}, Theorem 2.2. So $f_{T^{-1}}$ is annihilated by $\sigma\circ \alpha_{J_T}$ (note our convention for $a_i, b_i$ in (\ref{cliffbig})). The second assertion is Theorem 8.3 in {\it loc. cit}.
\end{proof}
The next statement is a simple consequence of the first part of the above theorem, still it lies at the heart of this paper.
\begin{lemma}\label{eigenvalues}
Let $T \in \mathfrak{h}_V$, $J_T \in \mathcal{J}_{\omega_0}$ the associated complex structure, $h=(h_1, h_2) \in \R^{2n}$. Let $f_{h, T}=\pi((h_1,h_2),0)f_{T^{-1}}$. Then
\[
(\sigma\circ \alpha_{J_T})(a_j) f_{h, T}= ((h_2)_j+\sum_{i}T_{ji}(h_1)_i) f_{h, T},\quad (\sigma\circ \alpha_{J_T})(b_j) f_{h, T}= i((h_2)_j+\sum_{i}T_{ji}(h_1)_i) f_{h, T} 
\]
for $j \in \{1,\dots, n\}$. In especially, for $T=iI$, so $J_T=J_0$, we conclude that the eigenvalues of $\sigma\circ \alpha_{J_0}$ acting on $f_{h, iI}$ constitute the set $\{((h_2)_j+i(h_1)_j),(i(h_2)_j-(h_1)_j)\}_{j=1}^n$.
\end{lemma}
\begin{proof}
First note that since the sets $\{u_i=a_i - iJ_Ta_i\}$ and $\{w_i=a_i- \sum_jT_{ij}J_0a_j\}$ with $i\in \{1,\dots,n\}$ both span $L_T$ (over $\C$) and since the real span of the $a_i$ is (real) Lagrangian, the expressions $iJ_Ta_i$ and $\sum_jT_{ij}J_0a_j$ actually coincide, since otherwise, we could produce real linear combinations of the $b_i$ from (complex) linear combinations of $u_i$ and $w_i$ which contradicts the fact that $L_T$ is totally complex. We consider the complexification of the Lie algebra $\mathfrak{h}_n$ of $H_{n}=\mathbb{R}^{2n} \times \mathbb{R}$ and the corresponding extension of $\pi_*: \mathfrak{h}_n\rightarrow {\rm End}(\mathcal{S}(\mathbb{R}^{n}))$. Then the claims follow from the following elementary calculation:
\[
\begin{split}
(\sigma\circ \alpha_{J_T})(a_j) f_{h, T}&=\frac{d}{dt}|_{t=1}\left(\pi(ta_j -t\sum_iT_{ji}b_i, 0)f_{h, T}\right)\\
&=\frac{d}{dt}|_{t=1}\left(\pi(ta_j -t\sum_iT_{ji}b_i, 0)\pi((h_1,h_2),0)f_{T^{-1}}\right)\\
&=\frac{d}{dt}|_{t=1}\left(\pi((h_1,h_2),0)\pi(ta_j -t\sum_iT_{ji}b_i, \omega_0(ta_j -t\sum_iT_{ji}b_i,(h_1,h_2))f_{T^{-1}}\right)\\
&= \pi((h_1,h_2),0)\left(\frac{d}{dt}|_{t=1}e^{i\omega_0(ta_j -t\sum_iT_{ji}b_i,(h_1,h_2))}f_{T^{-1}}\right)\\ &+(\pi((h_1,h_2),0)\pi\left(a_j -\sum_iT_{ji}b_i, \omega_0(a_j -\sum_iT_{ji}b_i, (h_1,h_2))\right)\sigma(a_j-\sum_iT_{ji}b_i)f_{T^{-1}}.
\end{split}
\]
by Theorem \ref{mumford}, the latter summand is zero, thus
\[
\begin{split}
(\sigma\circ \alpha_{J_T})(a_j) f_{h, T}&=\left(\frac{d}{dt}|_{t=1}e^{i\omega_0(ta_j -t\sum_iT_{ji}b_i,(h_1,h_2))}\pi((h_1,h_2),0) f_{T^{-1}}\right)\\
&=((h_2)_j+\sum_{i}T_{ji}(h_1)_i)f_{h, T}.
\end{split}
\]
The case $(\sigma\circ \alpha_{J_T})(b_j)$ acting on $f_{h, T}$ is derived in complete analogy.
\end{proof}
Before beginning to state the above in terms of representations of commutative algebras, we give an immediate corollary of the lemma which illustrates a certain reciprocity of information contained in the vectors $f_{h, T}$ resp. the (commuting set of) operators acting on them. For this, note that a pair consisting of a vector $h=(h_1, h_2) \in \R^{2n}$ so that $(h_1)_j>0, j\in \{1,\dots,n\}$ defines an element $T_h\in \mathfrak{h}_V$  by setting
\begin{equation}\label{recip}
T_h={\rm diag}((h_2)_1,\dots, (h_2)_n) + i\cdot {\rm diag}((h_1)_1,\dots, (h_1)_n)
\end{equation}
where ${\rm diag}(\dots)$ denotes the $n\times n$-matrix with the given entries on the diagonal and $0$ otherwise. By positivity of the entries of $h_1$, $T_h \in \mathfrak{h}_V$. Then we have:
\begin{folg}
For $h=(h_1,h_2) \in R^{2n}$ with $(h_1)_j>0, j\in \{1,\dots,n\}$ let $T_h \in \mathfrak{h}_V$ as in (\ref{recip}). Set $(\tilde h)=(\tilde h_1, \tilde h_2)$ where $\tilde h_1=(1,\dots,1)\in \R^n$ and $\tilde h_2=(0,\dots,0)\in \R^n$. Then we have
\[
(\sigma\circ \alpha_{J_{T_{h}}})(a_j) f_{\tilde h, T_{h}}= ((h_2)_j+i(h_1)_j f_{\tilde h,T_{h}},\quad (\sigma\circ \alpha_{J_{T_{h}}})(b_j) f_{\tilde h, T_{h}}= i((h_2)_j+i(h_1)_j f_{\tilde h,T_{h}}
\]
Note the eigenvalues of $\sigma\circ \alpha_{J_{T_{h}}}$ acting on $f_{\tilde h, T_{h}}$ thus coincide with the eigenvalues of $\sigma\circ \alpha_{J_{iI}}$ acting on $f_{h, iI}$ in Lemma \ref{eigenvalues}.
\end{folg}
\begin{proof}
The proof is immediate by plugging in the definitions and using the fact that with $T_h$ as in (\ref{recip}) we have
\[
\sum_{i}(T_h)_{ji}(\tilde h_1)_i=(h_2)_j+i(h_1)_j.
\]
\end{proof}
The rationale of this is, that at least for positive vectors $h_1$ in the tuple $(h_1,h_2)$, the information contained in such a tuple can always be 'shifted' to the parameter space given by positive Lagrangians resp. the Siegel space. We now interpret the above Lemma in terms of representations for certain commutative algebras. \\

Let $(V=\R^{2n}, \omega_0)$ be as above, $J_0$ the standard complex structure, $T \in \mathfrak{h}_V$, $J_T$ be the associated complex structure. Let $a_1,\dots, a_n, b_1, \dots, b_n$ be the symplectic standard basis, $V=L_0\oplus L_1$ the associated Lagrangian direct sum decomposition, that is, $L_0={\rm span}\{a_1,\dots, a_n\}$, $L_1={\rm span}\{b_1,\dots, b_n\}$. Denote by $\mathcal{A}_1(V)$ the associative subalgebra of $\textbf{sCl}(V,\omega_0)$ generated (as a subalgebra over $\R$) by the elements of $L_0$. Since $\omega_0|L_0=0$ we have with the two-sided ideal $\mathcal{I}_1(L_0)=\{ x\otimes y  -  y \otimes x : x,y \in L_0 \} \subset  \mathcal{T}(L_0)$ that
\[ 
\mathcal{A}_1(V)\simeq \mathcal{T}(L_0)/ \mathcal{I}_1= {\rm Sym}^*(L_0).
\]
On the other hand, consider ${\rm Sym}^*(V)$ as an algebra over $\R$ and consider the two-sided ideal in ${\rm Sym}^*(V)$ defined by 
\[
\mathcal{I}_2=\{ x\otimes y  +  J_0y \otimes J_0x : x,y \in L_0\}+\mathcal{I}_1(V),
\]
with $\mathcal{I}_1(V)$ the ideal generated by the commutators in $\mathcal{T}(V)$. Then $\mathcal{A}_2(V)={\rm Sym}^*(V)/\mathcal{I}_2$ is again a commutative, associative, but non-free $\R$-algebra. We have the identifications $\mathcal{A}_1(V)\simeq \R[x_1,\dots,x_n]$ and $\mathcal{A}_2(V)\simeq \R[x_1,\dots, x_n, ix_1,\dots,ix_n]$. In the latter, $x_j$ and $ix_j$ are interpreted as independent variables while we have the relation $x_jx_k=-(ix_j)(ix_k), j,k \in \{1,\dots,n\}$. Both algebras can be represented in 'rotated' form again as subalgebras in the (complexification of) the symplectic Clifford algebra and these representations will actually give rise to the irreducible one-dimensional representations we need to define 'Frobenius structures'.\\

Let $T \in \mathfrak{h}_V$, $J_T$ be the associated complex structure. Let $\textbf{sCl}_\mathbb{C}(V,\omega_0)=\textbf{sCl}(V,\omega_0)\otimes_{\R}\C$ be the complexification of $\textbf{sCl}(V,\omega_0)$. Let $\mathcal{A}_1(V, J_T)$ be generated as an $\R$-subalgebra of $\textbf{sCl}_{\C}(V,\omega_0)$ by the set
\[
\mathcal{L}_T=\{a_1 - iJ_Ta_1,\dots, a_n - iJ_Ta_n\} \subset \textbf{sCl}_{\C}(V,\omega_0),
\]
thus $\mathcal{A}_1(V, J_T)$ is the smallest subalgebra of $\textbf{sCl}(V,\omega_0)$ containing all real linear combinations and tensor products of elements of $\mathcal{L}_T$ (the latter is just $L_T$, considered as subspace in $V^{\C}$). Note that since $\omega_{\mathbb{C}}|L_T=0$ we have that the ideal generated by the relation $\mathcal{I}(\omega)$ in (\ref{comm1}), restricted to  $L_T$, is just $\mathcal{I}_1(L_T)$. Thus $\mathcal{A}_1(V, J_T)$ is a commutative $\R$-sub-algebra of $\textbf{sCl}_{\C}(V,\omega_0)$. Analogously, define $\mathcal{A}_2(V, J_T)$ as the $\R$-subalgebra of $\textbf{sCl}_\mathbb{C}(V,\omega_0)$ generated in $\textbf{sCl}_\mathbb{C}(V,\omega_0)$ over $\R$ by the set
\[
\mathcal{W}_T=\{a_1 - iJ_Ta_1,\dots, a_n - iJ_Ta_n, b_1 - iJ_Tb_1,\dots, b_n - iJ_Tb_n\} \subset \textbf{sCl}_{\C}(V,\omega_0).
\]
Note that for $\mathcal{A}_2(V, J_T)$ its commutativity again follows since $\alpha_{J_T}(V)=L_T$ and $L_T$ is Lagrangian w.r.t. $\omega_{\C}$. Thus we have the following proposition:
\begin{prop}\label{algebras}
For any $T \in \mathfrak{h}_V$, the algebras $\mathcal{A}_1(V)= \R[x_1,\dots,x_n]$ and $\mathcal{A}_1(V, J_T)$ are isomorphic as $\R$-algebras. Furthermore, for any $T \in \mathfrak{h}_V$, the $\R$-algebras $\mathcal{A}_2(V)= \R[x_1,\dots, x_n, ix_1,\dots,ix_n]$ and $\mathcal{A}_2(V, J_T)$ are isomorphic. Put another way, $\mathcal{A}_1(V, J_T),\ T \in \mathfrak{h}_V$ resp. $\mathcal{A}_2(V, J_T),\ T \in \mathfrak{h}_V$ can be considered as a set of mutually equivalent repesentations of $\R[x_1,\dots,x_n]$ resp. $\R[x_1,\dots, x_n, ix_1,\dots,ix_n]$ on sub-algebras of $\textbf{sCl}_\mathbb{C}(V,\omega_0)$.
\end{prop}
\begin{proof}
The homomorphism $\phi_T:\mathcal{A}_1(V) \rightarrow \mathcal{A}_1(V, J_T)$ is on the generating elements $a_i$ just given by the $\R$-linear map $\alpha_{J_T}(a_i)$, the same homomorphism, extended to the $b_i$, gives $\phi_T:\mathcal{A}_2(V) \rightarrow \mathcal{A}_2(V, J_T)$ and these homomorphisms are clearly bijective. The composition $\phi_T\circ \phi^{-1}_{T'}$ then intertwines the corresponding representations for two given $T,T'\in \mathfrak{h}_V$.
\end{proof}
From Lemma \ref{eigenvalues} and Proposition \ref{algebras} it is now clear that the pairs $(\C\cdot f_{h, T}, \mathcal{A}_1(V, J_T))$ resp. $(\C \cdot f_{h, T}, \mathcal{A}_2(V, J_T))$ for $T \in \mathfrak{h}, h \in \R^{2n}$ together with symplectic Clifford multiplication considered as a map
\[
\sigma:\mathcal{A}_{1,2}(V, J_T) \rightarrow {\rm End}(\C \cdot f_{h, T}),
\]
define irreducible (necessarily one-dimensional) representations denoted by $\kappa^{1,2}_{h,T}$ respectively, of the algebras $\mathcal{A}_1(V)$ resp. $\mathcal{A}_2(V)$. It remains to identify which of these are equivalent. For this, consider the following semi-direct product $G=H_n\times_\rho Mp(2n, \R)$, that is for $(h_i, t_i) \in H_n, g_i \in Mp(2n, \R), i=1,2$ we have the composition (note that this differs from the usual definition since we will consider $G$ acting on the right on diverse objects in what follows)
\begin{equation}\label{semid}
(h_1, t_1, g_1)\cdot (h_2,t_2, g_2)=(h_2+\rho(g_2)^{-1}(h_1), t_1+t_2+\frac{1}{2}\omega_0(\rho(g_2)^{-1}(h_1), h_2)), g_1g_2).
\end{equation}
Consider the subgroups $G_U=H_n\times_\rho \hat U(n)\subset G$, $G_0=\{(0,0),\R\}\times_\rho Mp(2n,\R)$ where $\hat U(n)=\rho^{-1}(U(n))$ and $U(n)=Sp(2n)\cap O(2n)$. Consider now the sets $\mathcal{A}_{1,2}=\{(\C\cdot f_{h, T}, \mathcal{A}_{1,2}(V, J_T)), \ T\in \mathfrak{h}, h \in \R^{2n}\}$ of complex lines and commutative algebras. We define maps
\begin{equation}\label{action1}
\begin{split}
\mu_{1,2}:G \times \mathcal{A}_{1,2}\rightarrow \mathcal{A}_{1,2},\ &(h,t,g)\cdot (\C\cdot f_{h_0, T}, \mathcal{A}_{1,2}(V, J_T))= (\C\cdot f_{h+\rho(g^{-1})h_0, T.g}, \mathcal{A}_{1,2}(V, J_{T.g})),\\
 &h,h_0\in \R^{2n},\ t \in \R,\ g \in Mp(2n),
\end{split}
\end{equation}
where $T.g$ indicates $g^{-1}.T^{-1}$. We have induced actions of $G$ on the set $\mathcal{K}_{1,2}=\{\kappa^{1,2}_{h,T},\ T\in \mathfrak{h}, h \in \R^{2n}\}$. Note that $\mu_{1,2}$ are smooth (i.e. continuous) actions of $G$ on the set of complex lines and algebras $\mathcal{A}_{1,2}$ in the sense that for any pair $(\C\cdot f_{h_0, T}, \mathcal{A}_{1,2}(V, J_T)) \in \mathcal{A}_{1,2}$, the map $(h,t,g) \mapsto \mu_{1,2}((h,t,g),(\C\cdot f_{h_0, T}, \mathcal{A}_{1,2}(V, J_T)))$ is smooth (continuous) as a map from $G$ to $\mathcal{A}_{1,2}(V, J_T)$. For the following, note that $T$ appears in $f_{h, T}=\pi((h_1,h_2),0)f_{T^{-1}}$ with negative power which is why we have to resort to right actions to define the action of $G$ on the set $\{\C\cdot f_{h_0,T}\}, T\in \mathfrak{h}, h \in \R^{2n}$.
\begin{prop}\label{class}
$\mu_{1,2}$ define transitive $G$-actions on the {\it sets} $\mathcal{A}_{1,2}$ whose action on the first coordinate of $\mathcal{A}_{1,2}$ equals the right action
\[
\tilde \mu: G \times {\rm pr}_1\circ \mathcal{A}_{1,2}\rightarrow {\rm pr}_1\circ \mathcal{A}_{1,2}, \ ((h,t,g),\C\cdot f_{h_0,T}) \mapsto \C \cdot \pi((h,t))L(g^{-1})f_{h_0,T}.
\]
The isotropy group of this action at a given point of $\mathcal{A}_{1,2}$ is isomorphic (conjugated) to $G_0\cap G_U$. On the other hand, the irreducible {\it representations} $\kappa^{1,2}_{h_1,T_1}$ and $\kappa^{1,2}_{h_0,T_0}$ are equivalent (as pairs of algebras and representation) if and only if there exists $\hat g \in \hat g_0G_0\hat g_0^{-1},\ \hat g_0 \in G$ so that $\hat g \cdot \kappa^{1,2}_{h_1,T_1}= \kappa^{1,2}_{h_0,T_0}$.
\end{prop}
\begin{proof}
We first prove that if $(h_1, g_1), (h_2, g_2)\in G$ (we suppress the real number $t$ in the following since it has no effect when dealing with the action of $G$ on representations) then if $T=iI$ and $h_0=0$
\begin{equation}\label{compos}
f_{0, iT}.\tilde \mu(h_1, g_1).\tilde \mu(h_2, g_2)=\pi(h_2+\rho(g_2^{-1})h_1)L((g_1g_2))^{-1}f_{0, iI}=f_{h_2+\rho(g_2^{-1})h_1, iI.(g_1g_2)}.
\end{equation}
where the action on the left hand side is $\tilde \mu$. For the second equality we used the definition of $f_{h_0,T}, T\in \mathfrak{h}, h \in \R^{2n}$ and Theorem \ref{mumford}. So the first equality is to be shown. We have
\begin{equation}\label{semirep}
\begin{split}
f_{0, iI}.\tilde \mu(h_1, g_1).\tilde \mu(h_2, g_2)&=\pi(h_2, 0)L(g_2^{-1})\pi(h_1, 0)L(g_1^{-1}) f_{0, iI}\\
&=\pi(h_2+\rho(g_2^{-1})(h_1), 0)L(g_2^{-1})L(g_1^{-1})f_{0, iI}\\
&=\pi(h_2+\rho(g_2^{-1})(h_1), 0)L((g_1g_2))^{-1}f_{0, iI}.
\end{split}
\end{equation}
Thus $\tilde \mu$ gives the action of (\ref{action1}) on $\mathcal{A}_{1,2}$, restricted to the first coordinate. We leave transitivity to the reader. From the explicit formula for $\tilde \mu$, we see that the isotropy group of $(f_{h_0,iI}, \mathcal{A}_{1,2}(V, J_{iI}))$ is $G_0\cap G_U$. It remains to show that if two elements $\kappa^{1,2}_{h,T}, \  \kappa^{1,2}_{h_0,T_0}\in \mathcal{K}_{1,2}$ are equivalent, then they differ by an appropriate element of $\hat g \in \hat g_0G_0 \hat g_0^{-1}$ for some $\hat g_0 \in G$, that is $\hat g \cdot \kappa^{1,2}_{h_1,T_1}= \kappa^{1,2}_{h_0,T_0}$. For this note that elements of the form $(0, g) \in G_U$ act by $\tilde \mu$ as invertible intertwining operators on the set of pairs $\mathcal{A}_{1,2}$ resp. the set of representations $\mathcal{K}_{1,2}$. This follows directly from the definition of $\tilde \mu$ resp. (\ref{compos}). Furthermore one checks by direct calculation that if $\kappa^{1,2}_{h,T_1}$ and $\kappa^{1,2}_{\tilde h,T_2}$ are equivalent as pairs of algebras and representations then $\mu(\hat h, t,g).\kappa^{1,2}_{h,T_1}$ is equivalent to $\mu(\hat h,t, g).\kappa^{1,2}_{\tilde h,T_2}$ for any $(\hat h,t, g) \in G$. Thus $G$ acts transitively on the set $\mathcal{K}_{1,2}/\sim$ where $\sim$ denotes the equivalence relation induced by identifying equivalent pairs of algebras and representations in $\mathcal{K}_{1,2}$. We claim that the isotropy group of this action at $(0,iI)$ is $G_0$. Thus it suffices to check that if $\kappa^{1,2}_{h,T}$ for $T \in \mathfrak{h}, h \in \R^{2n}$ is equivalent to $\kappa^{1,2}_{0,iI}\in \mathcal{K}_{1,2}$, then $h=0$. We check the case $\mathcal{K}_{1}$. By (\ref{compos}) we have to show that if 
\[
\sigma\circ \alpha_{J_T}(a_i)f_{h, T}=0
\]
for all $a_i$, then $h=(h_1,h_2)=0$. But
\[
\sigma\circ \alpha_{J_T}(a_i)f_{h, T}=((h_2)_j+\sum_{i}T_{ji}(h_1)_i) f_{h, T}.
\]
by the invertibility of $Im(T)$, we then infer $h_1=0$. But then it also follows that $h_2=0$. The case $\mathcal{K}_{2}$ is proven analogously.
\end{proof}
Considering for a fixed $T \in \mathfrak{h}$, $\sigma_T=\sigma\circ \alpha_{J_T}$ as giving an algebra homomorphism $\sigma_T:\mathcal{A}_{1,2}(V)\rightarrow {\rm End}(\mathcal{S}(\R^n))$, thus a representation of $\mathcal{A}_{1,2}(V)$ on the set of smooth vectors of $L$, we arrive at the following
\begin{folg}\label{equivalence}
The set of equivalence classes of irreducible subrepresentations of ${\rm Im}(\sigma_T)$ on $\mathcal{S}(\R^n)$, where $T \in \mathfrak{h}$ is fixed, is isomorphic to $G/G_0$, to be more precise it is explicitly given by the $G/G_0$-orbit of $\mu$ through $(\C \cdot f_{0, T}, \mathcal{A}_{1,2}(V, J_T))$ in $\mathcal{K}_{1,2}$. On the other hand the set of all $\sigma_T, \ T\in \mathfrak{h}$ and their corresponding set of irreducible representations on $\mathcal{S}(\R^n)$ is isomorphic to $G/G_0\cap G_U$ by the same identifications.
\begin{proof}
By Proposition \ref{class}, for fixed $T$, the $G/G_0$-orbit of $\mu$ through $(\C \cdot f_{0, T}, \mathcal{A}_{1,2}(V, J_T))$ in $\mathcal{K}_{1,2}$ is contained in the set of irreducible representations of $\sigma_T(\mathcal{A}_{1,2}(V))$ on $\mathcal{S}(\R^n)$. Now let $\C\cdot f, f \in \mathcal{S}(\R^n)$ define an irreducible representation of the subalgebra $\sigma_T(\mathcal{A}_{1}(V))\subset {\rm End}(\mathcal{S}(\R^n))$, that is
\[
\sigma_T(a_i)f= \lambda_i f,
\]
for some set $\lambda_i \in \C, i=1,\dots, n$. Then by using induction on $n$ and the Cauchy-Kovalevskaya Theorem, we see that $f$ is uniquely determined, hence the assertion. The case $\mathcal{A}_2(V)$ is similar.
\end{proof}
\end{folg}
We finally note that the generating elements of $\sigma_T(\mathcal{A}_{1,2}(V))\subset {\rm End}(\mathcal{S}(\R^n)), T \in \mathfrak{h}$ can be recovered as a subset of a natural representation of the Lie algebra $\mathfrak{g}$ of $G$. Recall (\cite{sternbergwolf}) that $\mathfrak{g}$ can be desribed as the sum $\mathfrak{g}=\mathfrak{sp}(2n,\mathbb{R})+ \mathfrak{h}_n$, where $\mathfrak{sp}(2n,\mathbb{R})$ and its Lie bracket are decribed in Section \ref{lagcoh} and $\mathfrak{h}_n$ is just the vectorspace $V+\R$ with the Lie bracket
\[
[(v,s),(w,t)]=(0, \omega_0(v,w)),\ v,w \in V, \ s,t \in \R,
\]
while on $\mathfrak{g}$, we have
\[
[(a,v,s),(b,w,t)]=([a,b], aw-bv, \omega_0(v,w)), a,b \in \mathfrak{sp}(2n,\mathbb{R}), \ v,w \in V, \ s,t \in \R,
\]
where $a \in \mathfrak{sp}(2n,\mathbb{R})$ here acts on $v \in V$ by $av=ad(a)(v)$ as in Section \ref{lagcoh}. We then claim that the following assignment $\kappa_T: \mathfrak{g}\rightarrow  {\rm End}(\mathcal{S}(\R^n))$ gives a Lie algebra representation of $\mathfrak{g}$ on $\mathcal{S}(\R^n)$ (compare \cite{sternbergwolf}):
\[
\begin{split}
(a_i,0)\mapsto  \hat \sigma_T(a_i):=\sigma\circ (a_i - iJ_Ta_i),& \quad (0,b_i)\mapsto  \hat \sigma_T(b_i):=\sigma\circ (a_i + iJ_Ta_i), \\
\ u\cdot v+v\cdot u\in \mathfrak{a}\simeq \mathfrak{sp}(2n,\mathbb{R})&\ \mapsto\ \hat \sigma_T(u)\cdot\hat \sigma_T(v)+\hat \sigma_T(v)\cdot\hat \sigma_T(u),
\end{split}
\]
where we identified $\mathfrak{sp}(2n,\mathbb{R})$ with the algebra $\mathfrak{a}$ of symmetric homogeneous polynomials of order two in $\textbf{sCl}(\mathbb{R}^{2n})$ as in Lemma \ref{symm} and we defined $\hat \sigma_T(v), v \in \R^n$ by extending linearly. Notice that the embeddings $\alpha^\pm_J=Id\pm iJ:V\rightarrow V^{\C}$, considered as $\R$-isomorphisms onto its image, define isomorphisms 
\[
\Phi^\pm_T: {\rm Im}(\alpha^\pm_J)\rightarrow {\rm Im}(\alpha^\pm_{J_0}), \ \Phi_T=\alpha^\pm_{J_0}\circ (\alpha^\pm_{J})^{-1}|{\rm Im}(\alpha^\pm_J)
\] 
defining an endomorphism $\Phi^\pm_T:V^{\C}\rightarrow V^{\C}$ which maps to a Lie algebra isomorphism $\Phi_T:\hat \sigma_{T_0}(\mathfrak{a})\rightarrow \hat \sigma_{T}(\mathfrak{a})$ via $\sigma$, which we denote also by $\Phi_T$. We then claim:
\begin{lemma}\label{fock}
The assignment $\kappa_T: \mathfrak{g}\rightarrow  {\rm End}(\mathcal{S}(\R^n))$ defines a Lie algebra representation of $\mathfrak{g}$ on $\mathcal{S}(\R^n)$ so that we have the equality $\kappa_T(a_i)= \sigma_T(a_i), i=1, \dots, n$. Furthermore there is an invertible Lie algebra endomorphism (the one defined above) $\Phi_T:\mathfrak{sp}(2n,\mathbb{R})\rightarrow \mathfrak{sp}(2n,\mathbb{R})$ so that
\[
\kappa_T|\mathfrak{sp}(2n,\mathbb{R})=\Phi_T\circ  L_{\ast}
\]
where $L_{\ast}: \mathfrak{mp}(2n,\mathbb{R}) \mapsto \mathfrak{u}(\mathcal{S}(\mathbb{R}^{n}))$ is as given by Proposition \ref{diffmp}.
\end{lemma}
\begin{proof}
The result is a direct calculation based on the formulas in Proposition \ref{diffmp} (see also the analogous calculation in \cite{sternbergwolf}, Lemma 4.8).
\end{proof}
In the following we aim to generalize the above slightly and classify the equivalence classes of (finite-dimensional) indecomposable, non-irreducible sub-representations of $\mathcal{A}_{1,2}(V, J_T)$, acting by $\sigma$ on the set of smooth vectors of $L$, $\mathcal{S}(\R^n)$ for any $T\in \mathfrak{h}$. For this we have to consider the 'higher eigenmodes' of the harmonic oscillator. Consider the one-parameter-subgroup $t \mapsto exp(ta)$ of $Mp(2n,\mathbb{R})$, determined by the element
\[
a\ = \ \frac{1}{2}\sum_{i=1}^{n}(a_{i} \cdot a_{i} + b_{i} \cdot b_{i}) \in \mathfrak{mp}(2n,\mathbb{R})
\]
Then its $L^{2}(\mathbb{R}^{n})$-representation $F_{t}=L(exp(ta))$ factorizes as a torus-representation into a countable number of irreducible representations:
\begin{lemma}[\cite{habermann})]
The representation $F_{(\cdot)}:\mathbb{R} \rightarrow U(L^{2}(\mathbb{R}^{n}))$ factorizes to a torus representation on
$L^{2}(\mathbb{R}^{n})$ that induces a decomposition $L^{2}(\mathbb{R}^{n})=\oplus_{k=0}^\infty Q_{k},\ k \in \mathbb{N}_{0}$ into irreducible subrepresentations of $F$. $Q_k$ is given by the eigenspaces of the $n$-dimensional harmonic oscillator $H$ and we have for $h_{k} \in Q_{k}$ 
\[ 
F_{t}h_{k}\ = \ e^{\frac{it\mu_{k}}{2}}h_{k}, \quad t \in \mathbb{R}, 
\] 
where $\mu_{k}$ is the eigenvalue of $H$ corresponding to $h_{k}$.
\end{lemma}
\begin{proof}
We reproduce the proof here since we will need the notation in the following. For $u \in \mathcal{S}(\mathbb{R}^{n})$ we have
\[
\begin{split}
\frac{d}{dt}F_{t}u_{\vert t=0}\ &= \ \frac{d}{dt}L(exp(ta))u_{\vert t=0}\
=\ L_{*}(a)u = \ L_{*}(\frac{1}{2}\sum_{j=1}^{n}(a_{j} \cdot a_{j} + b_{j} \cdot b_{j}))u
\\
&= \ \frac{1}{2}\sum_{j=1}^{n}(ix_{j}^{2}-i\frac{\partial^{2}}{\partial
x_{j}^{2}})u =\ \frac{1}{2} iHu.
\end{split}
\]
Where $H$ is for $u \in\mathbb{R}^{n}$ given by
\[
Hu = \sum_{j=1}^{n}(x_{j}^{2}u-\frac{\partial^{2}u}{\partial x_{j}^{2}})=\Delta
u+\vert x\vert^{2}u.
\] 
Its eigenfunctions are of the form
\[
h_{k_{1}\dots k_{n}}(x) =  \prod_{l=1}^{n}h_{k_{l}}(x_{l})
\]
for $k_{1},\dots,k_{n}\in \N_0$, $h_{k_{l}}$ being the eigenfunctions of the $1$-dimensional hamonic oscillator to the eigenvalue  $\mu_{k_{l}}=2k_{l}+1$ sind. Thus the eigenvalue corresponding to $h_{k_{1}\dots k_{n}}$ is
\[
\mu_{k}=\sum_{j=1}^{n}\mu_{k_{j}}=2k+n \qquad \textrm{where} \
k=\sum_{j=1}^{n}k_{j}.  
\]
The dimension $M_{n}^{k}$ of the eigenspace associated to $\mu_{k}$ thus equals the number of ordered $n$-tuples of non-negative integers $k_{j}$ so that $k=\sum_{j=1}^{n}k_{j}$. We have (siehe \cite{habermann}) 
\[
M_{n}^{k}=\left(\begin{matrix}n+k-1 \\k \end{matrix}\right).
\]
We thus can define $Q_{k}\subset L^{2}(\mathbb{R}^{n})$ by
\[
Q_{k}\ = \ \{u \in L^{2}(\mathbb{R}^{n}): F_{t}u =
e^{\frac{it\mu_{k}}{2}}u,\ \textrm{for\ all}\ t \in \mathbb{R} \}
\]
thus coinciding with the eigenspaces of $H$.
\end{proof}
Note that this of course gives in essence a proof of the splitting theorem reproduced below in Proposition \ref{decomp}.
Consider now for $g \in Mp(2n,\mathbb{R})$ the element $a^g=Ad(g^{-1})(a) \in \mathfrak{mp}(2n,\mathbb{R})$. Then it is obvious that $h_{k_{1}\dots k_{n}}^g=L(g^{-1})h_{k_{1}\dots k_{n}},\ h_{k_{1}\dots k_{n}} \in Q_k, \ (k_{1}\dots k_{n}) \in \N_0^n$ gives the eigenfunctions to the eigenvaluae $\mu_k$ of $L_*(a^g)=\frac{d}{dt}L(exp(ta))u_{\vert t=0} \in \mathfrak{u}(\mathcal{S}(\mathbb{R}^{n}))$ acting on $\mathcal{S}(\R^n)$, their set denoted as $Q_k^g$. Since $\rho_*(a)=-J$ we see that $ad(b)(a)=0$ for any $b \in \mathfrak{u}(n)$, where $\mathfrak{u}(n)$ is the Lie Algebra of $\hat U(n)$, thus also $Ad(g)(a)=a$ for all $g \in \hat U(n)$ and $h_{k_{1}\dots k_{n}}^g \in Q_{k}, \ (k_{1}\dots k_{n}) \in \N_0^n$ for all $g \in \hat U(n), \ (k_{1}\dots k_{n}) \in \N_0^n$ and thus $Q_k=Q_k^g,\ g \in \hat U(n)$. Note that the decomposition $h_{k_{1}\dots k_{n}}(x)=\prod_{l=1}^{n}h_{k_{l}}(x_{l})\in Q_k$ is not in general preserved by the action of $g \in \hat U(n)$ by $L(g^{-1})$ on $Q_k$. With these precautions we define in the following for $(h,t, g) \in G, \ (k_{1}\dots k_{n}) \in \N_0^n$ and $T=T_0.g\in \mathfrak{h}$, $T_0=iI$, $f_{h, g, {k_{1}\dots k_{n}}}=\pi((h_1,h_2),0)h^g_{k_{1}\dots k_{n}}$, where we write shortly $k={k_{1}\dots k_{n}}$ where this causes no confusion, thus $f_{h, g, k}=f_{h, g, {k_{1}\dots k_{n}}}$, while we will keep the previous notation for $\mathcal{A}_{1,2}(V, J_T)$ introduced above Proposition \ref{algebras}, note that then $f_{h,T}=f_{h, g, k}$. We then define
\begin{equation}\label{indecomp}
\sigma:\mathcal{A}_{1,2}(V, J_T) \rightarrow {\rm End}(\oplus_{k=0}^N\C \cdot f_{h, g, k}),\quad \ T=T_0.g \in \mathfrak{h},\ h \in \R^{2n},\ g \in Mp(2n,\mathbb{R}),
\end{equation}
where the summation is again over multi-indices $k=(k_{1}\dots k_{n}), k_i \in \N_0$, so that $\sum_{j=1}^{n}k_{j}\leq N\in \N_0$, this mapping is well-defined giving representations $\kappa^{1,2}_{h,g,N}$ and we get an analogue of Proposition \ref{class}. Note that we treat $\mathcal{F}_{h, g, N}:=\oplus_{k=0}^N \C \cdot f_{h, g, k},\ N\in \N_0$ here as a direct sum of $M^{N}_{n}$-dimensional vectorspaces, not as a subspace of $L^{2}(\mathbb{R}^{n})$. In further analogy to the above, consider now the sets $\mathcal{A}^N_{1,2}=\{(\mathcal{F}_{h, g, N}, \mathcal{A}_{1,2}(V, J_T)), \ T=T_0.g\in \mathfrak{h},\ g \in Mp(2n,\mathbb{R}), \ h \in \R^{2n},\ N\in \N_0\}$ of finite sums of complex vector spaces and commutative algebras ($T_0=iI \in \mathfrak{h}$ here and in the following). We define maps
\begin{equation}\label{action1N}
\begin{split}
\mu^N_{1,2}:G \times \mathcal{A}^N_{1,2}\rightarrow \mathcal{A}^N_{1,2},\ &(h,t,g)\cdot (\mathcal{F}_{h_0, g_0, N}, \mathcal{A}_{1,2}(V, J_T))= (\mathcal{F}_{h+\rho(g^{-1})h_0, T.g, N}, \mathcal{A}_{1,2}(V, J_{T.g})),\\
 &h,h_0\in \R^{2n},\ t \in \R,\ g_0, g \in Mp(2n),\ N\in \N_0,\ T=T_0.g_0.
\end{split}
\end{equation}
We have induced actions of $G$ on the associated set $\mathcal{K}^N_{1,2}=\{\kappa^{1,2}_{h,g, N},\ g \in Mp(2n,\mathbb{R}), h \in \R^{2n}, N\in \N_0\}$ of representations of algebras $\mathcal{A}_1(V)$ resp. $\mathcal{A}_2(V)$. For $N \in N_0$, let $K_{N}$ be the set of $n$-tuples $(k_{1}\dots k_{n}) \in N_0^n$ so that $\sum_{i=1^n}k_{i}\leq N$. Note that $\N_0^n$ comes with a natural partial order $\leq$. Recall that a chain in $\N_0^n$ is a sequence of elements $k_1\leq k_2\dots$. We will say a subset $K\subset K_N$ is chain-incident to $0 \in \N_0^n$ if with any $k \in K$, $K$ contains all chains in $K_N$ which have $k$ as a maximal element. Let $\mathcal{P}_0(K_N)$ be the set of all subsets $K\subset K_N$ which are chain incident to $0$ in the above sense.
\begin{prop}\label{classN}
$\mu^N_{1,2}$ define transitive $G$-actions on the {\it sets} $\mathcal{A}^N_{1,2}$ whose action on the first coordinate of $\mathcal{A}^N_{1,2}$ equals for any $N\in \N_0$ the right action
\[
\tilde \mu^N: G \times {\rm pr}_1\circ\mathcal{A}^N_{1,2}\rightarrow {\rm pr}_1\circ  \mathcal{A}^N_{1,2}, \ ((h,t,g),\mathcal{F}_{h_0, g_0, N}) \mapsto \oplus_{k=0}^N\C \cdot \pi((h,t))L(g^{-1}) f_{h_0,g_0, k}.
\]
The isotropy group of this action, here defined as the set of $(h,t,g) \in G$ that fixes $\mathcal{A}_{1,2}(V, J_T)$ as a subset of ${\rm End}(\mathcal{S}(\R^n))$ at a given point of $\mathcal{A}^N_{1,2}$ and preserves $Q_k\subset L^{2}(\mathbb{R}^{n}),\ k=1, \dots, N$ is isomorphic (conjugated) to $G_0\cap G_U$. On the other hand, the  equivalence classes of indecomposable {\it sub-representations} of pairs $\mathcal{K}^N_{1,2}$ are distinguished by the elements $\mathcal{P}_0(K_N)\times G/G_0$ in the sense that if $K_1,K_2 \in \mathcal{P}_0(K_N)$ then $(\oplus_{k\in K_1} \C \cdot f_{h_{0}, g, k}, \mathcal{A}_{1,2}(V, J_{T_0.g}))\subset \kappa^{1,2}_{h_0,g,N}$ and $(\oplus_{k\in K_2} \C \cdot f_{h_{1}, g_1, k}, \mathcal{A}_{1,2}(V, J_{T_0.g_1}))\subset  \kappa^{1,2}_{h_1,g_1,N}$ for $N>0$ are equivalent (as pairs of algebras and representations) if and only if $K_1=K_2$ and there exists $\hat g \in \hat g_0G_0\hat g_0^{-1},\ \hat g_0 \in G$ so that $\hat g \cdot \kappa^{1,2}_{h_1,T_1,N}= \kappa^{1,2}_{h_0,T_0, N}$ while in the case $N=0$ we get exactly the irreducible case of Proposition \ref{class}.
\end{prop}
\begin{proof}
To show that the assignment (\ref{indecomp}) defines a representation of $\mathcal{A}_{1,2}(V, J_T)$ on $\mathcal{F}_{h_0, T, N}$ consider the well-known commutation relation $[a_i-iJ_0a_i, a]=2(a_i-iJ_0a_i),\  i\in \{1,\dots,n\}$ resp. $[a_i-iJ_Ta_i, a^g]=2(a_i-iJ_Ta_i),\  i\in \{1,\dots,n\}$ for $T=T_0.g$, thus $a(a_i-iJ_0a_i)h_{k_{1}\dots k_{n}}=(\mu_k-2)(a_i-iJ_0a_i)h_{k_{1}\dots k_{n}}$, thus $(a_i-iJ_0a_i)h_{k_{1}\dots k_{n}}=h_{k_{1}\dots k_{i}-1\dots k_{n}}$ for $i\in \{1,\dots,n\}$, which gives already the assertion for $T=iI$, while we have $(a_i-iJ_Ta_i)h^g_{k_{1}\dots k_{n}}=h^g_{k_{1}\dots k_{i}-1\dots k_{n}}$ if $T=T_0.g$. The claim that $(\oplus_{k\in K} \C \cdot f_{h_{1}, T_1, k}, \mathcal{A}_{1,2}(V, J_T))\subset  \kappa^{1,2}_{h_1,T_1,N}$ with $K\subset K_N$ chain incident to $0$ define indecomposable subrepresentations then follows from the normal form of indecomposable representations of $\C[x]$, which are of the form $J_{\lambda,n}:\C^n\rightarrow \C^n$ with
\begin{equation}\label{decompN}
J_{\lambda,n}e_i=\lambda e_i+e_{i-1}, \ i>1\quad J_{\lambda,n}e_i=\lambda e_i,\ i=1,
\end{equation}
where $e_i$ are a basis of $\C^n$. That the indecomposable subrepresentations of $\mathcal{A}_{1,2}(V, J_T)$ acting via $\sigma$ on $\mathcal{F}_{h_0, T, N}$ are necessarily of the form above is again a consequence of the Cauchy-Kowalewskaja theorem, in complete analogy to the proof of Proposition \ref{class}. Again by a direct calculation one checks that if $\kappa^{1,2}_{h,g_1, N}$ and $\kappa^{1,2}_{\tilde h,g_2, N}$ are equivalent as pairs of algebras and representations then $\mu^N(\hat h, t,g).\kappa^{1,2}_{h,g_1, N}$ is equivalent to $\mu(\hat h,t, g).\kappa^{1,2}_{\tilde h,g_2, N}$ for any $(\hat h,t, g) \in G$. Then it suffices to check that if $\kappa^{1,2}_{h,g, N}$ for $g \in Mp(2n), h \in \R^{2n}$ is equivalent to $\kappa^{1,2}_{0,Id_{Mp(2n,\R)}, N}\in \mathcal{K}^N_{1,2}$, then $h=0$. We check the case $\mathcal{K}^N_{1}$. By (\ref{decompN}) we have to show that if $T=T_0.g$
\[
\sigma\circ \alpha_{J_T}(a_i)f_{h, g, k}=f_{h, g, k-1},  \ k \in \N^N,
\]
where $k-1=(k_{1}\dots k_i-1\dots k_{n}) \in \N^N$, for all $a_i$, then $h=(h_1,h_2)=0$. But 
\[
\sigma\circ \alpha_{J_T}(a_i)f_{h,g,k}=((h_2)_j+\sum_{i}T_{ji}(h_1)_i) f_{h, g, k}+ f_{h, g, k-1}.
\]
by the invertibility of $Im(T)$, we then infer $h_1=0$. But then it also follows that $h_2=0$. The case $\mathcal{K}^N_{2}$ is proven analogously. All other assertions are again in complete analogy of the proof of Proposition \ref{class} resp. follow from the discussion above the Proposition and we suppress them here.
\end{proof}
Analogous to the irreducible case, fixing $T \in \mathfrak{h}$, $\sigma_T=\sigma\circ \alpha_{J_T}$ as giving an algebra homomorphism $\sigma_T:\mathcal{A}_{1,2}(V)\rightarrow {\rm End}(\mathcal{S}(\R^n))$, thus a representation of $\mathcal{A}_{1,2}(V)$ on the set of smooth vectors of $L$, we arrive at the following
\begin{folg}\label{equivalenceN}
The set of equivalence classes of indecomposable subrepresentations of ${\rm Im}(\sigma_T)\subset {\rm End}(\mathcal{S}(\R^n))$ on $\mathcal{S}(\R^n)$, where $T \in \mathfrak{h}$ and is fixed, is isomorphic to $\bigcup_{N\in \N_0} \mathcal{P}_0(K_N)\times G/G_0$, to be more precise it is for fixed $K \in \mathcal{P}_0(K_N)$ explicitly given by the $G/G_0$-orbit of $\mu$ through $(\oplus_{k\in K} \C \cdot f_{0, g, k}, \mathcal{A}_{1,2}(V, J_T)), T=T_0.g$ in $\mathcal{K}^N_{1,2}$. On the other hand the set of all $\sigma_T, \ T\in \mathfrak{h}$ and their corresponding set of indecomposable representations on $\mathcal{S}(\R^n)$ is isomorphic to $\bigcup_{N\in \N_0} \mathcal{P}_0(K_N)\times G/G_0\cap G_U$ by the same identifications.
\end{folg}

\section{Symplectic spinors and Frobenius structures}

In this section, we will exhbit the main concept of 'Higgs pairs' resp. 'Frobenius structures' via symplectic spinors in a generality that will be sufficient to deal with the different manifestations of these structures over symplectic manifolds $M$ with certain additional data, i.e. the presence of a Hamiltonian system or a Lagrangian submanifold. In all cases, the assumption that $c_1(M)= 0\ {\rm mod}\ 2$ (here a nearly complex structure is chosen) will be necessary and sufficient to define the appropriate lift of the symplectic frame bundle.

\subsection{Symplectic spinors and Lie derivative}\label{spinorsconn}

Let $(M,\omega)$ be a symplectic manifold of dimension $2n$. For $p\in M$ we denote by $R_p$ the set of symplectic bases in $T_pM$, that is the $2n$-tuples $e_{1}, \dots,e_{n},f_{1}, \dots, f_{n}$
so that
\[
\omega_{x}(e_{j},e_{k})=\omega_{x}(f_{j},f_{k}) = 0, \ 
\omega_{x}(e_{j},f_{k}) = \delta_{jk} \quad \textrm{for} \ j,k =1,\dots,2n.
\]
The symplectic group $Sp(2n)$ acts simply transitively on $R_p, \ p \in M$ and we denote by $\pi_{R}:R:=\bigcup_{p\in m} R_p \rightarrow M$ the symplectic frame bundle. By the Darboux Theorem $R$ it is a locally trivial $Sp(2n)$-principal fibre bundle on $M$. As it is well-known, the $\omega$-compatible almost complex structures $J$ are in bijective correspondence with the set of $U(n)$-reductions of $R$. Given such a $J$, we call local sections of the associated $U(n)$-reduction $R^J$ of the form $(e_1,\dots, e_n, f_1,\dots,f_n)$ unitary frames. These frames are characterized by
\[
g(e_j,e_k)=\delta_{jk} \quad g(e_j,f_k)=0, \quad Je_j=f_j,
\]
where $j,k=1,\dots,n$ and $g(\cdot,\cdot)=\omega(\cdot,J\cdot)$. Now a metaplectic structure of $(M,\omega)$ is a $\rho$-equivariant $Mp(2n)$-reduction of $R$, that is:
\begin{Def}\label{metaplecticst}
A pair $(P,f)$, where $\pi_P:P\rightarrow M$ is a $Mp(2n,\mathbb{R})$-principal bundle on $M$ and $f$ a bundle morphism $f:P \rightarrow R$, is called metaplectic structure of $(M,\omega)$, if the following diagram commutes:
\begin{equation}\label{metadiag}
\begin{CD}
P \times Mp(2n,\mathbb{R})  @>>> P \\
@VV{f\times\rho}V              @VV{f}V \\
R \times Sp(2n,\mathbb{R}) @>>>   R
\end{CD}
\end{equation}
where the horizontal arrows denote the respective group actions.
\end{Def}
It follows that $f:P\rightarrow R$ is a two-fold connected covering. Furthermore it is known (\cite{habermann}, \cite{kost}) that $(M,\omega)$ admits a metaplectic structure if and only if $c_{1}(M)=0 \ mod \ 2$. In that case, the isomorphism classes of metaplectic structures are classified by $H^{1}(M,\mathbb{Z}_{2})$. $\kappa$ defines a continuous left-action of $Mp(2n,\mathbb{R})$ on $L^{2}(\mathbb{R}^n)$, acting unitarily on $L^{2}(\mathbb{R}^{n})$. Combining this with the right-action of $Mp(2n)$ on a fixed metaplectic structure $P$, we get a continuous right-action on $P\times L^{2}(\mathbb{R}^n)$ by setting
\begin{equation}\label{balance}
\begin{split}
(P \times L^{2}(\mathbb{R}^{n})) \times Mp(2n) \ &\rightarrow \
P \times L^{2}(\mathbb{R}^{n}) \\
((p,f),g) \ &\mapsto \ (pg,\kappa(g^{-1})f).
\end{split}
\end{equation}
The symplectic spinor bundle $\mathcal{Q}$ is defined to be its orbit space
\[
\mathcal{Q}=P \times_{\kappa}L^{2}(\mathbb{R}^{n}) := (P \times
L^{2}(\mathbb{R}^{n}))/Mp(2n)
\]
w.r.t. this group action, so $\mathcal{Q}$ is the $\kappa$-associated vector bundle of $P$. We will refer to its elements in the following by  $[p,u]$, $p\in P$, $u\in L^{2}(\mathbb{R}^{n})$. Note that if $\pi_P$ is the projection $\pi_P: P \rightarrow M$ in $P$, then $\mathcal{Q}$ is a locally trivial fibration $\tilde \pi: \mathcal{Q} \rightarrow M$ with fibre $L^{2}(\mathbb{R}^{n})$ by setting $\tilde \pi([p,u])= x$ if $\pi_P(p) = x$. Then continuous sections $\phi$ in $\mathcal{Q}$ correspond to continuous $Mp(2n)$-equivariant mappings $\hat \phi:P\rightarrow L^2(\mathbb{R}^n)$, that is $\hat \phi(pq)=\kappa(q^{-1})\hat \phi(p)$ for $p\in P$. Hence we define {\it smooth} sections $\Gamma(\mathcal{Q})$ in $\mathcal{Q}$ as the continuous sections whose corresponding mapping $\hat \phi$ is smooth as a map $\hat \phi:P\rightarrow L^2(\mathbb{R}^n)$. It then follows (\cite{habermann}) that $\hat \phi(p)\in \mathcal{S}(\mathbb{R}^n)$ for all $p\in P$, so smooth sections in $\mathcal{Q}$ are in fact sections of the subbundle
\[
\mathcal{S} = P \times_{\kappa}\mathcal{S}(\mathbb{R}^{n}).
\]
Note that due to unitarity of $L$, the usual $L^2$-inner product on $L^{2}(\mathbb{R}^{n})$ defines a fibrewise hermitian product $<\cdot, \cdot>$ on $\mathcal{Q}$.\\
Given a $U(n)$-reduction $R^J$ of $R$ w.r.t. a compatible almost complex structure $J$ on $M$ and a fixed metaplectic structure $P$, we get a $\hat U(n):=\rho^{-1}(U(n))$-reduction $\pi_{P^J}:P^J \rightarrow M$ of $P$ by setting $P^J:=f^{-1}(R^J)$, where $f$ is as in Definition \ref{metaplecticst}. So we get by denoting the restriction of $\kappa$ to $\hat U(n)$ by $\tilde \kappa$ an isomorphism of vector bundles
\begin{equation}\label{spinorbundle}
\mathcal{Q}\simeq \mathcal{Q}^J:=P^J \times_{\tilde \kappa}L^{2}(\mathbb{R}^{n}).
\end{equation}
Correspondingly we define $\mathcal{S}^J$ so that $\mathcal{S}^J\simeq \mathcal{S}$. At this point, the Hamilton operator $H_0$ of the harmonic oscillator on $L^2(\mathbb{R}^n)$ gives rise to an endomorphism of $\mathcal{S}$ and a splitting of $\mathcal{Q}$ into finite-rank subbundles as follows. Let $H_0:\mathcal{S}(\mathbb{R}^n)\rightarrow \mathcal{S}(\mathbb{R}^n)$ be the Hamilton operator of the $n$-dimensional harmonic oscillator as given by
\[
(H_0u)(x)=-\frac{1}{2}\sum_{j=1}^{n}(x_{j}^{2}u-\frac{\partial^{2}u}{\partial x_{j}^{2}}), \ u \in \mathcal{S}(\mathbb{R}^n).
\]
\begin{prop}[\cite{habermann}]\label{decomp}
The bundle endomorphism $\mathcal{H}^J:\mathcal{S}^J\rightarrow\mathcal{S}^J$ declared by $\mathcal{H}^J([p,u])=[p,H_0u],\ p\in P, u\in \mathcal{S}(\mathbb{R}^n)$ is well-defined. Let $\mathcal{M}_l$ denote the eigenspace of $H_0$ with eigenvalue $-(l+\frac{n}{2})$. Then the spaces $\mathcal{M}_l,\ l \in \mathbb{N}_0$ form an orthogonal decomposition of $L^2(\mathbb{R}^n)$ which is $\tilde \kappa$-invariant. So $\mathcal{Q}^J$ decomposes into the direct sum of finite rank-subbundles
\[
\mathcal{Q}^J_l=P^J \times_{\tilde \kappa}\mathcal{M}_l, \quad {\rm s.t.}\ {\rm rank}_{\mathbb{C}}\mathcal{Q}^J_{k}=\left(\begin{matrix}n+k-1 \\ k \end{matrix}\right)
\]
where we defined $\mathcal{Q}^J_l=\{q \in \mathcal{S}:\mathcal{H}^J(q)=-(l+\frac{n}{2})q\}$. 
\end{prop}
Occasionally, we will use the dual spinor bundle $Q'$ of $Q$. To define this, note that if we topologize the Schwartz space $\mathcal{S}(\mathbb{R}^{n})$ by the countable family of semi-norms
\[
p_{\alpha,m}(f)={\rm sup}_{x\in \mathbb{R}^n}(1+|x|^m)|(D^\alpha f)(x)|, \  f\in \mathcal{S}(\mathbb{R}^{n}),
\]
then the topology of $(\mathcal{S}(\mathbb{R}^{n}),\tau)$ is induced by a translation-invariant complete metric $\tau$, hence manifests the structure of a Frechet-space. Furthermore $\kappa:Mp(2n)\rightarrow \mathcal{U}(\mathcal{S}(\mathbb{R}^{n}))$ still acts continuously, which follows by the decomposition (\ref{metalinear})-(\ref{fourier}) and the fact that multiplication by monomials and Fourier transform act continuously w.r.t. $\tau$, which is a standard result. Then, denoting the dual space of $(\mathcal{S}(\mathbb{R}^{n}),\tau)$ as $\mathcal{S}'(\mathbb{R}^{n})$, we can consider for any pair $T \in \mathcal{S}'(\mathbb{R}^{n}), g \in Mp(2n)$ the continuous linear functional $\hat \kappa(g)(T) \in \mathcal{S}'(\mathbb{R}^{n})$ defined by
\begin{equation}\label{metadual}
(\hat \kappa(g)(T))(f)=T(\kappa(g)^*f), \ f \in \mathcal{S}(\mathbb{R}^{n}).
\end{equation}
Thus we have an action $\hat \kappa:Mp(2n)\times \mathcal{S}'(\mathbb{R}^{n})\rightarrow \mathcal{S}'(\mathbb{R}^{n})$ which extends $\kappa:Mp(2n)\rightarrow \mathcal{U}(\mathcal{S}(\mathbb{R}^{n}))$ and is continuous relative to the weak-$*$-topology on $\mathcal{S}'(\mathbb{R}^{n})$. Note that since the inclusion $i_1:\mathcal{S}(\mathbb{R}^{n})\subset L^2(\mathbb{R}^n)$ is continuous, we have the continuous triple of embeddings $\mathcal{S}(\mathbb{R}^{n})\subset L^2(\mathbb{R}^n) \subset \mathcal{S}'(\mathbb{R}^{n})$. Here $L^2(\mathbb{R}^n)$ carries the norm topology and the inclusion $i_2:L^2(\mathbb{R}^n) \hookrightarrow \mathcal{S}'(\mathbb{R}^{n})$ is given by $i_2(f)(u)=(f,\overline u)_{L^2(\mathbb{R}^n)}$ where the latter denotes the usual $L^2$-inner product on $\mathbb{R}^n$. We thus define in analogy to (\ref{spinorbundle})
\[
\mathcal{Q}'=P^J \times_{\hat \kappa}\mathcal{S}'(\mathbb{R}^{n}),
\]
where here, $\hat \kappa:U(n)\rightarrow {\rm Aut}(\mathcal{S}'(\mathbb{R}^{n}))$ means restriction of $\hat \kappa$ to $U(n)$ (using the same symbol). Now any fixed section $\varphi \in \Gamma(\mathcal{Q}')$ may be evaluated on any $\psi \in \Gamma(\mathcal{Q})$ by writing $\varphi=[\overline s, T], \psi=[\overline s, u]$ w.r.t. a local section $\overline s: U\subset M\rightarrow P^J$ and smooth mappings $T:U\rightarrow \mathcal{S}'(\mathbb{R}^{n})$, $u: U\rightarrow \mathcal{S}(\mathbb{R}^{n})$ by setting 
\[
\varphi(\psi)|U(x)= T(u)(x), \ x \in U\subset M.
\]
It is clear that this extends to a mapping $\varphi:\Gamma(\mathcal{Q})\rightarrow C^\infty(M)$.\\

A connection $\nabla: \Gamma(TM) \rightarrow \Gamma(T^{*}M \otimes TM)$ on $(M, \omega)$ is called {\it symplectic} iff $\nabla\omega = 0$. As is well-known (\cite{tondeur}), there always exist symplectic connections, even torsion free symplectic connections on any symplectic manifold, but the latter are not unique. However, if $J$ is an $\omega$-compatible almost complex structure, the formula
\begin{equation}\label{bla}
(\nabla_{X}\omega)(Y,Z)=(\nabla_{X}g)(JY,Z)+g((\nabla_{X}J)(Y),Z).
\end{equation}
shows that the additional assumption $\nabla J=0$ would force a torsion-free symplectic connection to be the Levi-Civita connection of a Kaehler manifold. So in general, symplectic connection preserving $J$ are not torsion-free. Note that symplectic connections are in bijective correspondence to connections $Z: TR \rightarrow \mathfrak{sp}(2n,\mathbb{R})$ on the symplectic framebundle $R$ (cf. \cite{habermann}). Let $Z:TR \rightarrow \mathfrak{sp}(2n,\mathbb{R})$ be the connection on $R$ corresponding to the symplectic connection $\nabla$ on $M$. Then $Z$ uniquely lifts to a connection one-form $\overline Z: TP \rightarrow \mathfrak{mp}(2n,\mathbb{R})$ on $P$ so that $\overline Z= \rho_{*}^{-1}\circ Z\circ f_{*}$, since $\rho_{*}$ is an isomorphism, $\overline Z$ is well-defined. For $s:U \subset M \rightarrow R$ being a local section, $\overline s:U \subset M \rightarrow P$ a local lift to $s$ inmto $P$, $X\in \Gamma(TM)$ and $u:U\rightarrow \mathcal{S}(\mathbb{R}^n)$, we have the induced covariant derivative $\nabla: \Gamma(\mathcal{Q}) \rightarrow \Gamma(T^{*}M \otimes \mathcal{Q})$ expressed on the local section $\varphi =[\overline s, u]$ as
\begin{equation} \label{spinorabl}
\nabla_{X} \varphi = [\overline s,du(X) + L_{*}(\overline Z \circ \overline s_{*}(X))u].
\end{equation}
We then have:
\begin{lemma}[\cite{habermann}]\label{cliffharm}
Symplectic Clifford-multiplication, spinor derivative and Hermitian Product in $\mathcal
S$ are compatible as follows:
\[
\begin{split}
<X \cdot \varphi, \psi > \ &=\ - <\varphi,X \cdot \psi> \\
\nabla_{X}(Y\cdot \varphi) \ &= \ (\nabla_{X}Y)\cdot \varphi+ Y \cdot
\nabla_{X}\varphi \\
X<\varphi,\psi> \ &= \ <\nabla_{X} \varphi, \psi> + <\varphi, \nabla_X \psi>.
\end{split}
\]
\end{lemma}
Since we will mostly deal with symplectic connections satisfying $\nabla J = 0$, the question arises if $\overline Z:TP
\rightarrow \mathfrak{mp}(2n,\mathbb{R})$ reduces to a $\hat {\overline Z}:TP^J \rightarrow \hat {\mathfrak{u}}(n)$ in $P^J$, so 
that $i_{*} \hat {\overline Z}\ =\ i^{*}\overline Z$. Here $i: \hat U(n) \hookrightarrow Mp(2n,\mathbb{R})$ and $i: P^J \hookrightarrow P$ are the respective inclusions. Under this condition, the spinor derivatives corresponding to $P$ and $P^J$ are identical. Indeed one has
\begin{lemma}[\cite{habermann}]
If $\nabla$ is a symplectic covariant derivative over $M$ and we have $\nabla J=0$, then the corresponding connection $Z$ in $R$ reduces to $Z^J$ in $R^J$ In the above sense. The latter lifts to a connection $\hat Z^J$ over $P^J$ as before:
\[
\begin{CD}
T P^J   @>{{\hat Z^J}}>> \hat {\mathfrak{u}}(n) \\
@VV{f^J_{*}}V              @VV{\rho_{*}}V \\
TR^J  @>Z^J>>   		\mathfrak{u}(n)
\end{CD}
\]
Here $f^J$ is the restriction of $f:P \rightarrow R$ to $P^J$.
\end{lemma}
We finally briefly describe the Lie derivative on symplectic spinors associated to a locally Hamiltonian symplectic diffeomorphism on $(M,\omega)$ as introduced in \cite{haberklein}. Recall that a family of vector fields $X_t \in \Gamma(TM), t \in I$ ($I$ is $\R$ or a small nghbd of $0$) on $(M, \omega)$ is called locally Hamiltonian if $i_{X_t}\omega$ is closed. Then its flow $\psi_t, \ t \in I$ satisfies $\psi_t^*(\omega)\omega$ for any $t$, that is $\psi_t \in {\rm Symp_0}(M, \omega)$, where the latter is the connected component of the identity of the symplectmorphism group and there is for any $t \in I$ the distinguished isotopy $\Psi_\tau,\ \tau \in [0,t]$, connecting $\Psi_t$ to the identity. Any symplectomorphism $\phi$ on $M$ induces an automorphism in $R$ by
\[
\begin{split}
\phi_{*}\ :\ &R \rightarrow R \\
&(e_{1},\dots,e_{n},f_{1},\dots,f_{n}) \mapsto
(\phi_{*}e_{1},\dots,\phi_{*}e_{n},\phi_{*}f_{1},\dots,\phi_{*}f_{n}).
\end{split} 
\]
lifting (non-uniquely) to an automorphism $\hat \phi_*:P\rightarrow P$ so that if $f:P\rightarrow R$ is the projection, then $f\circ \hat \phi_*=\phi_*\circ f:P\rightarrow R$. Assuming $M$ and hence $P$ connected this lift depends only on the choice of branch over $p \in R$. Since by the above, sections of  $\mathcal{S}$ are $Mp(2n, \R)$-equivariant maps $\varphi:P \rightarrow \mathcal{S}(\mathbb{R}^{n})$, we can define an action of $\phi$ on smooth sections of $\mathcal{Q}$ by setting
\[
(\phi^{-1})_{*}\varphi=\varphi\circ\tilde\phi_{*}:P \rightarrow
\mathcal{S}(\mathbb{R}^{n}) 
\]
$(\phi^{-1})_{*}\varphi$ remains $Mp(2n,\R)$ equivariant and hence defines a smooth spinor field over $M$.
For $\nabla$ a symplectic connection we have that (cf. \cite{habermann})
\begin{equation}\label{nablaphi}
\nabla^{\phi}_{(\phi^{-1})_{*}X}(\phi^{-1})_{*}Y=(\phi^{-1})_{*}(\nabla_{X}Y)
\end{equation}
is also a symplectic connection and the associated covariant derivative on spinors is given by
\begin{equation}\label{nablaphi2}
\nabla^{\phi}_{(\phi^{-1})_{*}X}(\phi^{-1})_{*}\varphi=(\phi^{-1})_{*}(\nabla_{X}\varphi).
\end{equation}
Let now $\psi_t, \ t\in I$ be a locally Hamiltonian flow, that is $i_{X_t}\omega$ is closed.
By requiring $(\psi_{0}^{-1})_{*}=id_{\Gamma(\mathcal{Q})}$ and by the continuity of the family
$(\psi_{t}^{-1})_{*}:\Gamma(\mathcal{Q}) \rightarrow \Gamma(\mathcal{Q})$, the latter is unambigously defined for all $t \in [0,1]$. One defines
\begin{Def}
The Lie derivative of $\varphi \in \Gamma(\mathcal{Q})$ in the direction of a locally Hamiltonian vector field $X=X_t, \ t \in I$ is given by
\[
\mathcal{L}_{X}\varphi\ =\ \frac{d}{dt}(\phi_{t}^{-1})_{*}\varphi_{\vert t=0},
\]
where $\{\psi_{t}\}_{t\in I}$ is the flow of $X_t$ on $M$.
\end{Def}
Then it is proven in \cite{haberklein}:
\begin{theorem}\label{lieabl}
Let $\nabla$ be a torsion-free symplectic connection and $X_t, t \in I$ a locally Hamiltonian vector field on $(M,\omega)$. Then
\begin{equation}\label{liebla}
\mathcal{L}_{X}\varphi\ =\ \nabla_{X}\varphi+\frac{i}{2}\sum_{j=1}^{n}\lbrace
\nabla_{e_{j}}X\cdot f_{j}-\nabla_{f_{j}}X\cdot e_{j}\}\cdot \varphi \quad
\textrm{for} \ \varphi \in \Gamma(\mathcal{Q}),
\end{equation}
where $e_{1},\dots,e_{n},f_{1},\dots,f_{n}$ is an arbitrary symplectic frame.
\end{theorem}
Note that, from the proof of the theorem in \cite{haberklein} we see that for a non-torsion-free symplectic connection $\nabla$  an additional term appears, that is one gets:
\begin{equation}\label{liebla2}
\mathcal{L}_{X}\varphi\ =\ \nabla_{X}\varphi+\frac{i}{2}\sum_{j=1}^{n}\lbrace
\nabla_{e_{j}}X\cdot f_{j}-\nabla_{f_{j}}X\cdot e_{j}\}\cdot \varphi + 
\frac{i}{4}\sum_{j=1}^{n}\lbrace i\omega(\nabla_{e_{j}}X, f_{j})-i\omega(\nabla_{f_{j}}X, e_{j})\}\varphi,
\quad
\end{equation}
for $\varphi \in \Gamma(\mathcal{Q})$. It is interesting to note the symmetry between the two last terms: up to a constant $\frac{1}{2}$, the last term replaces the symplectic Clifford multiplication by contraction with terms of the form $i_Y\omega$.
Given two $X_{f},X_{g}$ Hamiltonian vector fields over $M$ associated to functions $f,g$, their commutator is Hamiltonian
with Hamiltonian $\omega(X_{f},X_{g})$. For the spinor derivative one has:
\begin{folg}[\cite{haberklein}]\label{heisenberg}
Let $\varphi \in \Gamma(\mathcal{Q})$ and let the vectorfield $X,Y$ be Hamiltonian.
Then
\[
\mathcal{L}_{[X,Y]}\varphi\ =\ [\mathcal{L}_{X},\mathcal{L}_{Y}]\varphi.
\]
\end{folg}

\subsection{Frobenius structures and spectral covers}

Let $(M,\omega)$ be a symplectic (smooth) manifold of dimension $2n$ so that $c_1(M)= 0\ {\rm mod}\ 2$, $J$ is a compatible almost complex structure, $\nabla$ a symplectic connection and $\mathcal{Q}$ the symplectic spinor bundle wrt a choice of metaplectic structure $P$. Denote by $\textbf{sCl}_{\C}(TM,\omega)=\bigcup_{x\in M}\textbf{sCl}_{\C}(TM_x,\omega_x)$ the (infinite dimensional) vector bundle of (complexified) symplectic Clifford algebras, acting as fibrewise bundle endomorphisms on $\mathcal{Q}$. In the following, we will denote by $\mathcal{L}\subset \mathcal{Q}$ also (finite) sums and tensor products of arbitrary subbundles $\mathcal{L}\subset \mathcal{Q}$, with the action of $\textbf{sCl}_{\C}(TM,\omega)$ resp. a given spinor connection $\nabla$ extended in the usual way (where with 'usual' we mean unless indicated otherwise that if $a \in \textbf{sCl}_{\C}(TM,\omega)$ acts as endomorphism on $\mathcal{L}_{1,2}\subset \mathcal{Q}$, then it acts as $a\otimes {\bf 1}+ {\bf 1}\otimes a$ on $\mathcal{L}_1\otimes \mathcal{L}_2$ and $\nabla$ extends as ungraded derivation). Note that we understand the term 'subbundle' here in a general sense: a smoothly varying family of (finite or infinite-dimensional) subspaces $\mathcal{L}_x\subset \mathcal{Q}_x,\ x \in M$ whose 'dimension' (if finite) is locally constant on $M$ (compare the remarks above Theorem \ref{genclassN}). Let now $\mathcal{A}\subset \textbf{sCl}_{\C}(TM,\omega)$ be a subbundle of $\textbf{sCl}_{\C}(TM,\omega)$ so that its fibres $A_x$ for any $x \in M$ are commutative associative (not necessarily free) subalgebras with unity over $\R$ or $\C$ of $\textbf{sCl}_{\C}(TM_x,\omega_x)$ and so that there is an $\R$- (or $\C$)-linear injection $i:T_xM \hookrightarrow  A_x$ for any $x \in M$. Let $\mathcal{L}\subset \mathcal{Q}$ be a (finite or infinite dimensional) subbundle of $\mathcal{Q}$ so that $(\mathcal{A}_x,\mathcal{L})$ is for any $x\in M$ a representation of $\mathcal{A}_x$ as an algebra over $\C$. We denote by
\[
\star:TM \rightarrow End(\mathcal{L}),\ (X, \varphi) \mapsto X\star \varphi, 
\]
the restriction of the linear action of $\mathcal{A}$ on $\mathcal{L}$ to $TM$. Unless indicated otherwise, we assume that for any section $\varphi \in \Gamma(\mathcal{L})$, we have $\nabla \varphi\in \Gamma(\mathcal{L})$ if $\nabla$ is the spinor connection induced by $\nabla$ and $\nabla$ and $\mathcal{L}$ are defined throughout $M$ except possibly (where indicated) on the zero-dimensional part of a set of the form $\mathcal{S}=\mathcal{C}\cup \mathcal{D}$ with $\mathcal{C}=\bigcup_{i=1}^k\mathcal{C}_i$ and $\mathcal{D}=\bigcup_{i=1}^k\mathcal{D}_i$ where the $\mathcal{C}_i$ are isolated zero-dimensional subsets of $M$ and $\mathcal{D}_i\subset M$ is an $(n-1)$-dimensional closed submanifold.
\begin{Def}\label{frobenius}
We will say that the $5$-tuple $(\mathcal{L}, \mathcal{A}, \nabla, <\cdot,\cdot>, \mathcal{E})$, where $<\cdot ,\cdot>$ is the spinor scalar product on $\mathcal{L}$, and $\mathcal{E}\in \Gamma(M,T_{\C}M\otimes \mathcal{L}^*)$ is a {\it Frobenius structure} with singularity data $\mathcal{S}$ if (in addition to the above) the following relations are satisfied.
\begin{enumerate}
\item Write $\star:TM \rightarrow End(\mathcal{L})$ as the $1$-form $\Omega \in \Omega^1(M,End(\mathcal{L}))$. Then $\Omega\wedge\Omega=0$. If $\Omega=A_1+iA_2$, where $A_1, A_2 \in \Omega^1(M,End(\mathcal{L}))$ take values in the (formally) self-adjoint operators w.r.t. to $<\cdot, \cdot>$ and $A_i\wedge A_j=0,  \ i,j=1,2$, we say the structure is semi-simple.
\item $d_{\nabla}\Omega=0$, that is for any $X,Y \in \Gamma(TM)$ and $\varphi \in \Gamma(\mathcal{L})$ we have
\[
\nabla_{X}(Y\star \varphi) \ = \ (\nabla_{X}Y)\star \varphi+ Y \star \nabla_{X}\varphi.
\]
\item $\nabla(\nabla \mathcal{E})=0$.
\item $\mathcal{E}$ is an $\mathcal{L}^*$-valued locally Hamiltonian vector field on $M$, that is there exists a closed one form $\alpha \in \Omega^1(M, \mathbb{C}\otimes \mathcal{L}^*)$ so that
\begin{equation}\label{hamilton}
i_{\mathcal{E}}\omega=\alpha.
\end{equation}
If for any (bundle of) irreducible subrepresentations $(\mathcal{A},\mathcal{L}_i)$ we have $\Omega|_{\mathcal{L}_i}= f_i^*(\frac{dz}{z}) \in \Omega^1(M\setminus \mathcal{D}_i, \C)$ on $M$ outside $\mathcal{D}_i=\{x \in M:\ f_i=0\}$ and locally on $M\setminus \mathcal{S}_i$ (where $\mathcal{S}_i=\mathcal{C}_i \cup \mathcal{D}_i$ and $\mathcal{C}_i=\{x \in M:\Omega|_{\mathcal{L}_i}=0\}$) we have $(\alpha|_{\mathcal{L}_i})({\bf 1}) =f_i^*({\rm log(z)}\frac{dz}{z}|dz|^{-2})$ for some globally defined smooth function $f_i:M\rightarrow \mathbb{C}$ and a choice of branch of ${\rm log}$ and a global section ${\bf 1}:M\rightarrow \mathcal{L}_i$ we say that the Frobenius structure is {\it rigid}. Note that then, $\Omega|_{\mathcal{L}_i}=f_i^*(\frac{dz}{z}) \in H^1(M\setminus \mathcal{D}_i, \mathbb{Z})$.
\item For any subbundle of irreducible (hence one-dimensional) subrepresentations $(\mathcal{A},\mathcal{L}_i)_x, \ x \in M$ of $(\mathcal{A},\mathcal{L})_x$ there exist functions $b_i, d_i,  \in C^0(M,\mathbb{C})$ so that for $\varphi, \psi \in \Gamma(\mathcal{L}_i)$, $Y \in \Gamma(TM)$ and $\mathcal{E}_i:=\mathcal{E}|_{\mathcal{L}_i}(1)$ we have
\begin{equation}\label{charge}
\begin{split}
\mathcal{E}_i.<\varphi, \psi>- <\mathcal{L}_{\mathcal{E}_i}\varphi, \psi> - <\varphi, \mathcal{L}_{\mathcal{E}_i}\psi>&=d_i<\varphi, \psi> \\
\mathcal{L}_{\mathcal{E}_i}(Y\star \varphi)-\nabla_{\mathcal{E}_i}Y\star \varphi- Y\star \mathcal{L}_{\mathcal{E}_i}\varphi=b_i\varphi
\end{split}
\end{equation}
and if $\nabla$ is torsion free the latter equation reads
\[
\mathcal{L}_{\mathcal{E}_i}(Y\star \varphi)-[\mathcal{E}_i,Y]\star \varphi- Y\star \mathcal{L}_{\mathcal{E}_i}\varphi=\nabla_{Y}\mathcal{E}_i\star \varphi+b_i\varphi.
\]
\end{enumerate}
Note that the $b_i,d_i$ are in general not required to be constant. If they are, the Frobenius structure will be called {\it flat}. If $(\nabla, \mathcal{L})$ do only exist on $M\setminus \bigcup_{i=1}^k \mathcal{C}_i$, we will say the Frobenius structure is singular at $\mathcal{C}=\bigcup_{i=1}^k \mathcal{C}_i$, if $(\nabla, \mathcal{L})$ are defined on $M$ we will say the Frobenius structure is regular. If in the singular situation, the functions $f_i$ defined in 4. are still globally well-defined, we say the Frobenius structure is rigid singular. If $\mathcal{E}$ does not exist globally on $M\setminus \mathcal{C}$, but there exists an open covering $\mathcal{U}$ of $M\setminus \mathcal{C}$ so that $\mathcal{E}_U\in \Gamma(M,T_{\C}M\otimes \mathcal{L}^*)$ satisfies (3.), (5.) and (\ref{hamilton}) on each $U \in \mathcal{U}$ with $\alpha_U \in \Omega^1(U, \mathbb{C}\otimes \mathcal{L}^*)$ we say the Frobenius structure is weak. If furthermore in this case the $\alpha_U$  can be chosen so that on $M\setminus \mathcal{S}_i$ with $\mathcal{S}=\bigcup_{i=1}^k \mathcal{C}_i\cup \mathcal{S}_i$ we have $(\alpha_U|_{\mathcal{L}_i})(1) =f_i^*({\rm log(z)}\frac{dz}{z}|dz|^{-2})$ for globally defined functions $f_i:M\rightarrow \mathbb{C}$ satisfying $\Omega|\mathcal{L}_i= f_i^*(\frac{dz}{z}) \in \Omega^1(M\setminus \mathcal{D}_i, \C)$ and there exist coverings $p_i:\tilde M_i\rightarrow M\setminus \mathcal{D}_i$ so that the $p_i^*(\alpha_U|_{\mathcal{L}_i})(1)$ and the corresponding Euler vector fields $(\mathcal{E}_i)_U=p_i^*(\alpha_U|_{\mathcal{L}_i})(1)^{\omega_\perp}$ assemble to globally defined objects on $\tilde M_i\setminus p_i^{-1}(\mathcal{C}_i)$ so that if $\xi_i=f_i^*(\frac{dz}{z}) \in H^1(M\setminus \mathcal{D}_i, \mathbb{Z})$ we have $p_i^*(\xi_i)$ is exact on $\tilde M_i$, then we will say the Frobenius structure is {\it weakly rigid}. We finally consider the following notion (which will not be central in this article, but occurs in important examples):
\begin{enumerate}
\setcounter{enumi}{5}
\item We call two rigid (regular, all notions are defined in analogy for the singular case outside $\mathcal{C}$) Frobenius structures $(\mathcal{A}_{1,2},\mathcal{L}_{1,2})$ with respective Euler vector fields $\mathcal{E}_{1,2}\in \Gamma(M,T_{\C}M\otimes \mathcal{L}_{1,2}^*)$ and $\alpha_{1,2} \in \Omega^1(M, \mathbb{C}\otimes \mathcal{L}_{1,2}^*)$ satisfying (\ref{hamilton}) {\rm dual}, if there are smooth functions $\Theta_{1,2}:\mathcal{L}_{1,2}\rightarrow \C$ so that if $D_{1,2}$ denote the zero divisors of $\Theta_{1,2}$ and $\frac{dz}{z} \in \Omega^1(\C^*)$ the logarithmic $1$-form, there is for any (subbundle of) irreducible subrepresentations $(\mathcal{A},\mathcal{L}^i_{1})$ a corresponding irreducible $(\mathcal{A},\mathcal{L}^i_{2})$ so that over $M\setminus D_2$ resp. $M\setminus D_1$ we have
\begin{equation}\label{duality}
\Omega_1|\mathcal{L}^i_{1}= ({\bf 1})^*(\Theta_{2}|\mathcal{L}^i_{2})^*(\frac{dz}{z}), \quad \Omega_2|\mathcal{L}^i_{2}= ({\bf 1})^*(\Theta_{1}|\mathcal{L}^i_{1})^*(\frac{dz}{z}), \quad 
\end{equation}
where ${\bf 1}:M\rightarrow \mathcal{L}^i_{1,2}$ are as in (4.) and we identify $\Omega^1(M, {\rm End}(\mathcal{L}^i_{1,2}))\simeq \Omega^1(M, \C )$. Furthermore, locally over any open $U\subset M$ we have $\alpha_{1,2}({\bf 1}) =(\Theta_{2,1}\circ ({\bf 1})^*({\rm log(z)}\frac{dz}{z}|dz|^{-2}) \in \Omega^1(U, \C)$ for some choice of branch of logarithm. If there exist global sections $\vartheta_{1,2} \in \Gamma(\mathcal{L}_{1,2})$ and a $\delta \in \Gamma(\mathcal{Q}')$ so that
\begin{equation}\label{duality2}
({\bf 1})^*\Theta_{1,2}=<\vartheta_{1,2}, \delta> \in C^\infty(M,\C),
\end{equation}
the dual rigid Frobenius structures will be called a {\it dual pair}. If only one of the equations in (\ref{duality}) resp. (\ref{duality2}) is satisfied we will call $(\mathcal{A}_{1,2},\mathcal{L}_{1,2})$ weakly dual resp. a weakly dual pair. If furthermore one can chose for a rigid Frobenius structure $(\mathcal{A},\mathcal{L})$ a $\Theta:\mathcal{L}\rightarrow \C$ and $\vartheta \in \Gamma(\mathcal{L})$ satisfying the above, then $(\mathcal{A},\mathcal{L})$ is called {\it self-dual}.
\end{enumerate}
\end{Def}
{\it Remark.} The condition $d_{\nabla}\Omega=0$ in (2.) above means equivalently that if $c_i(X, \phi, \psi)=<A_i(X)\phi, \psi>, i=1,2, \phi, \psi \in E,\ X\in TM$ s.t. $c_i \in T^*M\otimes {\rm Sym}^2(E)$, then
\[
d_\nabla c_i=0, \quad i=1,2,
\]
and this is equivalent to $\nabla c_i \in {\rm Sym}^2(T^*M)\otimes {\rm Sym}^2(E)$, thus is symmetric. To see the latter note that $d_\nabla c_i$ is the antisymmetrization of $\nabla c_i\in (T^*M)^{2\otimes}\otimes {\rm Sym}^2(E)$ in the first two entries. On the other hand we have an isomorphism $\Phi:{\rm Sym}(E)\rightarrow {\rm Sym}^2(E), \ A \mapsto <A(\cdot),\cdot>$ (here ${\rm Sym}(E)$ denote the symmetric endomorphisms of $E$ wrt $<\cdot, \cdot>$) and since $\Phi(A_i)=c_i$, it follows $\Phi\circ d_\nabla A_i= d_\nabla c_i$, hence the first equivalence. Note that $d_{\nabla}\Omega=0$ is implied by the second condition in (2.).\\
Note that in this article, we will only deal with finite dimensional (generalized, in the above sense) smooth subbundles $\mathcal{L}\subset \mathcal{Q}$, furthermore in all our examples $(\mathcal{A},\mathcal{L})$ will be, if not irreducible, then decomposable into the sum of irreducible one-dimensional representations of $\mathcal{A}$ and $\mathcal{A}$ will be semi-simple and the above definition is tentative in these that it is modeled on these examples. We will henceforth assume that $\mathcal{L}\subset \mathcal{Q}$ is a finite dimensional subbundle. Since $\mathcal{A}$ is commutative, $\star:TM \rightarrow End(\mathcal{L})$ gives over any point $x \in M$ a (in general non-faithful) representation of ${\rm Sym}^*(TM)$ on $\mathcal{L}_x$, that is a morphism
\[
\star:{\rm Sym}^*(TM)\rightarrow End(\mathcal{L}), (y_1\odot\dots\odot y_k)_x(\varphi)=(y_1\star\dots\star y_k)_x(\varphi),
\]
where $y_1, \dots y_k\in T_xM$, $\varphi \in \mathcal{L}_x$ and $\odot$ denotes symmetric product. Any smooth local section $U\subset M \rightarrow {\rm Sym}^*(TM)$ can be viewed as a smooth function on $T^*M$ over $U\subset M$ (being polynomial in the fibres) by setting
\[
(y_1\odot\dots\odot y_k)_x(\mu)=(\mu(y_1)\cdot\dots\cdot \mu(y_k)), \ \mu \in T_x^*M,\ x\in U,
\]
we will call the sheaf over $M$ of such functions by $p_*\mathcal{O}_{T^*M}$, where $p:T^*M\rightarrow M$ is the canonical projection. Thus $\mathcal{L}$ gets the structure of an $\mathcal{O}_{T^*M}$-module and we arrive at 
\begin{Def}[\cite{audin}]
We define the spectral cover $L$ of a (finitely generated) Frobenius structure $(\mathcal{A},\mathcal{L})$ as the support of $\mathcal{L}$ as an $\mathcal{O}_{T^*M}$-module, that is the set of prime ideals $\mathfrak{p}$ in $\mathcal{O}_{T^*M}$ such that there exists no element $s$ in the multiplicative subset $\mathcal{O}_{T^*M}\setminus \mathfrak{p}$ so that $s\cdot \mathcal{L}=0$.
\end{Def}
It then follows that the prime ideals in ${\rm Supp}(\mathcal{L})_x$ correspond to the irreducible factors in the minimal polynomial of $\Omega_x(\cdot) \in \Omega^1(M, End(\mathcal{L}))$ associated to the common generalized eigenspaces of the endomorphisms $\Omega_x(X_i)$, when $X_i$ are a basis of $T_xM$. In especially, if there is at least one local vectorfield $X\in TU, U\subset M$ so that the minimal and characteristic polynomials coincide, we have that $L$ is given over $U$ by the vanishing locus of the map
\[
P_X: \Gamma(T^*U) \rightarrow End(\mathcal{L}), \quad P(\alpha)={\rm det}(\Omega(X)-\alpha(X)Id_{\mathcal{L}}),
\]
for all local vectorfields $X \in \Gamma(TU)$. Thus in this case, we have
\[
L\simeq {\rm Spec}\left(\frac{\mathcal{O}_{T^*M}}{I_\Omega}\right)
\]
where $I_\Omega$ is the ideal in $\mathcal{O}_{T^*M}$ generated by the characteristic polynomial of $\Omega$, acting on $\mathcal{L}$. If $(\mathcal{A},\mathcal{L})$ is {\it semi-simple}, then $\Omega$ is diagonalizable and the bundle $\mathcal{L}\rightarrow M$ splits as as sum
\begin{equation}\label{splitting}
\mathcal{L}=\bigoplus_{i=1}^k\mathcal{L}_i
\end{equation}
of eigenline bundles of the operators $\Omega$. If moreover all eigenvalues are distinct, we say $(\mathcal{A},\mathcal{L})$ is regular semi-simple. Then for $i=1,\dots, k$ there exists locally a one-form $\alpha_i \in \Omega^1(U, \mathbb{C})$ realizing the zero locus of $P_X$ corresponding to $\mathcal{L}_i$ for all $X \in TU$. The next observation is well-known in the theory of Frobenius manifolds (cf. \cite{audin}).
\begin{prop}\label{lagrangian}
If $(\mathcal{A},\mathcal{L})$ is regular semi-simple and $\nabla$ is torsion-free, then the $\alpha_i$ are closed for $i=1,\dots, k$.
\end{prop}
\begin{proof}
Let $\varphi_i \in \Gamma_U(\mathcal{L}_i)$ span $\mathcal{L}_i$ over some open nghbd $U$ of $x \in M$, respectively. By (2.) of Definition \ref{frobenius}, we have for any $Y \in T_xM$ that we extend to a vector field $Y\in \Gamma(TU)$ on a nghbd of $x$ that so that $(\nabla Y)_x=0$, that is
\[
\nabla_{X}(Y\star \varphi_i)  = Y \star \nabla_{X}\varphi_i, 
\]
at $x \in U$ for $X \in T_xM$. Writing $Y\star \varphi=\alpha_i(Y)\varphi_i$ for $\alpha_i \in \Omega^1(U, \mathbb{C})$ and since $\nabla$ is torsion-free i.e. determines $d$, it suffices to show that $\nabla\alpha_i=0$, thus $d(\alpha_i(Y))=0$ for any set of sections $Y \in \Gamma(TU)$ parallel at $x$ as above so that the set of $Y|T_xM=Y_x$ spans $T_xM$. Writing
\[
\nabla_{X}\varphi_i=\sum_{j=1}^k a_{ij}(X)\varphi_j
\]
for some one forms $a_{ij}(\cdot)$, we infer from the previous equation
\[
\nabla_{X}(\alpha_i(Y) \varphi_i) = \sum_{j=1}^k a_{ij}(X)\alpha_j(Y)\varphi_j
\]
and thus
\[
d(\alpha_i(Y))(X)\varphi_i = \sum_{j=1}^k a_{ij}(X)(\alpha_j-\alpha_i)(Y)\varphi_j.
\]
The $(\alpha_i)(Y)$ being distinct and the set $\varphi_i$ being linearly independent we infer comparing coefficients that $a_{ij}=0$ for $i\neq j$ and hence $d(\alpha_i(Y))=0$.
\end{proof}
{\it Remark.} Note that if $\nabla$ is not torsion-free, we can only deduce $\nabla \alpha_i=0$ in the above situation. However in the case of  standard semisimple (irreducible) Frobenius structures $(\mathcal{A},\mathcal{L}_i)$ considered here (cf. Theorem \ref{genclass}) the closedness of the $\alpha_i$ will follow already from the requirement that $\nabla$ preserves $\mathcal{L}_i$, so the (non-ramified) spectral covers will be always Lagrangian in this case.\\
Given a vectorfield $\mathcal{E} \in \Gamma(M, T_{\C}M\otimes \mathcal{L}^*)$ satisfying (3.) and (4.) of Definition \ref{frobenius}, the 'scaling of structure' property (5.) will actually follow (cf. Theorem \ref{lieabl}) if we demand that if for $\varphi \in \Gamma(\mathcal{L}_i)$ and any $i$, where $\mathcal{L}_i$ is a given irreducible representation of $\mathcal{A}$, $L_*(\mathfrak{sp}(2n,\mathbb{R}))$ leaves $\mathbb{C}\cdot \varphi$ invariant. This is of course a rather strong assumption. Instead, from the linearity of $\mathcal{E}_i=\mathcal{E}|_{\mathcal{L}_i}(1)$, that is the property $\nabla(\nabla \mathcal{E}_i)=0$ it follows that for any $i\in \{1,\dots,k\}$ the term $\mathcal{L}_{\mathcal{E}_i}-\nabla_{\mathcal{E}_i}$ in Theorem \ref{lieabl} is in local 'normal Darboux coordinates' in some sense (see below) the Schroedinger-equation associated to the locally Hamiltonian vectorfield $\mathcal{E}_i$. In our specific Frobenius structures, the $\varphi \in \Gamma(\mathcal{L}_i)$ thus will always satisfy a (time-dependent) Schroedinger-equation associated to the normal-order quantization of the (locally linear, complex) Hamiltonian vectorfield $\mathcal{E}_i$. Since the appearance of such a Schroedinger-equation is of some importance here (cf. Lax \cite{lax}), we will recall the result from \cite{haberklein} here for our present setting.
\begin{prop}\label{ham1}
There is, for any $i=1,\dots,k$, a symplectic coordinate system $\Phi:U\rightarrow \R^{2n}$ in a neighbourhood of any $x \in U\subset M$ (unique up to choice of symplectic basis in $x \in M$), so that $(\mathcal{E}_i)_x$ is the Hamiltonian vectorfield to a quadratic Hamiltonian function $H_i:\R^{2n}\rightarrow {\C}$. Then, $(\mathcal{H}_i)_x:=(\mathcal{L}_{\mathcal{E}_i}-\nabla_{\mathcal{E}_i})_x:\mathcal{S}(\mathbb{R}^{n})\rightarrow \mathcal{S}(\mathbb{R}^{n})$ is the (Fourier-transform-conjugated) normal-ordering quantized Hamilton operator associated to $H_i$. $\mathcal{H}_i$ is in general non-selfadjoint.
\end{prop}
\begin{proof}
Note that fixing a symplectic basis in $T_xM$ and using Fedosov's associated weakly normal Darboux coordinates (\cite{fedosov}) at $x \in U\subset M$, we infer that since $\nabla(\nabla\mathcal{E}_i))=0$ on $U$, that in these coordinates
\[
\mathcal{E}_i=J_0 \nabla H_i\circ  \Phi^{-1}(z)= J_0\nabla <z,Qz> + O(|z|^\infty), \ z \in \R^{2n}, \ Q\in M(2n, \mathbb{C}),
\]
that is, on $U$ we have $\mathcal{E}_i(z)=Az + O(|z|^\infty)$, where $A \in {\rm sp}(2n, \C)$. Then, by linearly extending $L_*$ to the complexification of $\mathfrak{sp}(2n, \R)$, we get by computations analogous to (\cite{haberklein}, Corollary 3.3) that
\[
L_*\circ (\rho_*^{-1}(A^\top))=-i\mathcal{H}_i,
\]
where $\mathcal{H}_i$ is the (Fourier-transform-conjugated) normal-ordering quantization Hamiltonian associated to $H_i$. Setting $S_t=exp(tA^\top)\in Sp(2n,\R)$, lifting $S_t$ to the path $M_t \in Mp(2n,\R)$ with $M_0=Id$ and choosing a local frame $s: U\subset M\rightarrow R$ over $U$ with lift $\overline s: U\subset M\rightarrow P$, we get as in \cite{haberklein} if $\phi_t:\Gamma(\mathcal{Q})\rightarrow \Gamma(\mathcal{Q})$ is the family of automorphisms induced by the flow of $\mathcal{E}_i$ for small $t$ and $\varphi=[\overline s, \psi]$ over $U$:
\[
(\phi_t^{-1})_{*}\varphi= [\overline s, \mathcal{F}^{-1}\circ L(M_t)\circ {\mathcal{F}}\psi]
\]
where ${\mathcal{F}}$ denotes the Fourier-transform. Differentiating at $t=0$ gives the assertion.
\end{proof}
'Fourier-transform-conjugated' thus means, that in contrary to the usual convention, we replace $q_j$ in $H_i$ by $\frac{\partial}{\partial x_j}$ and $p_j$ by the multiplication operator $ix_j$ and extend complex-linearly. We finally give a sufficient condition (at least for the semi-simple case) for a vector field $\mathcal{E}_i$ satisfying (3.) and (4.) in Definition \ref{frobenius} to also satisfy the condition (5.). This condition is satisfied in all our examples and can be essentially stated as $\mathcal{H}_i$ being the Hamiltonian operator to the 'ladder operator' $\Omega$, where $\mathcal{H}_i$ is the 'Hamiltonian' (locally) associated to $\mathcal{E}_i$ as in the previous theorem. The condition is in especially satisfied if $\mathcal{E}_i$ is of the (local) form $\mathcal{E}_i=u_i (du_i)^{\#_\omega}$ where $u_i$ is a local primitive of $\alpha_i$ over $U\subset M$ while $\Omega$ is of the form examined in Section \ref{coherent} and an appropriate $\hat U(n)$-reduction of $P$ (an $\omega$-compatible complex structure) is chosen. Note that for $\beta\in \Omega^1(M, \C)$, we denote $\beta^{\#_\omega} \in \Gamma(TM)$ 
the vectorfield so that 
\[
\omega(\beta^{*}, J\beta^{*})\cdot i_{\beta^{\#_\omega}}\omega=\beta,\ {\rm i.e.}\  \beta(J\circ\beta^{\#_\omega})=1,
\]
where $(\cdot)^*: T^*M\rightarrow TM$ denotes the usual duality given by $\omega$. We will denote the inverse of $(\cdot)^{\#_\omega}:T^*M\rightarrow TM$ with the same symbol. Note that while $(du_i)^{\#_\omega}$ is singular on the critical locus $\mathcal{C}_i$ of $du_i \in \Omega^1(M, \C)$, $\mathcal{E}_i=u_i (du_i)^{\#_\omega}$ is well-defined on any open sets where a local primitive of $\alpha_i$ exists by choosing $u_i$ so that $u_i|\mathcal{C}_i=0$. We denote the Fourier transform on symplectic spinors associated to an $\hat U(n)$-reduction of $P$ by (cf. \cite{kleinf}) by $\mathcal{F}$. Note that in the following, we will set for a $2$-tensor $R: E\otimes E \rightarrow V$, where $E$ is a Hermitian vector space and $V$ is a vector space $Tr(R)=\sum_{i=1}^n R(e_i,e_i)$ for a unitary basis $e_i$. We then have:
\begin{prop}\label{euler}
Let $(\mathcal{A},\mathcal{L})$ be semi-simple. If for any $x \in M$ and a small open set $U\subset M$ containing $x$, the Hamiltonian $(\mathcal{H}_i)_x$ associated to a symplectic frame in $x$, an eigenline bundle ${\mathcal{L}_i}$ over $U$, $\varphi \in \Gamma_U(\mathcal{L}_i)$ and to $\mathcal{E}_i$ satisyfing (3.) and (4.) in Definition \ref{frobenius} satisfies 
\begin{equation}\label{hamiltonian}
(\mathcal{H}_i)_x\varphi=(c_1{\rm Tr}\left(\mathcal{F}^{-1}\circ \Omega_i^{t}\cdot\Omega_i\circ \mathcal{F}\right)+ c_2{\rm Tr}\left(\Omega_i\cdot\Omega_i\right)+c_3)_x\cdot\varphi,\ c_1, c_2, c_3 \in \mathbb{C},\ c_1\neq 0,
\end{equation}
where $\Omega_i=\Omega|\mathcal{L}_i$, $\Omega_i^t$ denotes the adjoint wrt $<\cdot, \cdot>$, then $\mathcal{E}_i$ obeys (5.) in Definition \ref{frobenius}. Assume that $\mathcal{E}_i$ is of the form $\mathcal{E}_i=u_i (du_i)^{\#_\omega}$ where $u_i$ is a local primitive of the eigenform $\alpha_i$ of $\Omega_i$ corresponding to the splitting (\ref{splitting}). Choose an $\omega$-compatible complex structure $J$ on $M$ that satisfies $\nabla J=0$. Then $\mathcal{E}_i$ satisfies (3.) and (4.) in Definition \ref{frobenius}. If  $\Omega_i(X)\varphi=(X-iJX)\cdot \varphi, \varphi \in \Gamma(\mathcal{L}_i)$, then (\ref{hamiltonian}) holds for constants $c_1, c_2$ determined by $\mathcal{E}_i$. Note that if $E_i \in C^\infty(M, \C)$ is determined by ${\rm pr}^\perp_{\mathcal{L}_i}((\mathcal{H}_i)_x\varphi)=E_i(x)\varphi, \ x \in M, \varphi \in \mathcal{L}_i$ where ${\rm pr}^{\perp}_{\mathcal{L}_i}$ denotes pointwise orthogonal projection wrt $<\cdot,\cdot>$ in $\mathcal{Q}_x,\ x \in M$ on $(\mathcal{L}_i)_x$ we have the equality 
\begin{equation}\label{charge2}
d_i = 2{\rm Re}(E_i)=2 {\rm Re}\left(c_1{\rm Tr}(|\alpha_i|^2)+ c_2{\rm Tr}(\alpha_i^2) +c_3\right).
\end{equation}
\end{prop}
\begin{proof}
Assume $\Omega_x\varphi \neq 0$ (otherwise the assertion is trivial). The first part of the first assertion (the first line of (\ref{charge})) follows by considering that for $\varphi, \psi \in \Gamma_U(\mathcal{L}_i)$, invoking the third line of Lemma \ref{cliffharm} and (\ref{liebla2}) and noting that
\[
(\Omega_x^{t}\cdot \Omega_x \varphi, \psi)=(\Omega_x \varphi, \Omega_x \psi)=\alpha_i(x)\overline \alpha_i(x)(\varphi, \psi)
\]
the first summand in (\ref{hamiltonian}) produces a contribution of $2|\alpha_i|^2$ to $d_i$ while the other summands are obvious. The second line of $(\ref{charge})$ follows by considering that $[\Omega_i(e_l), \Omega_i^t(e_k)]=ib\delta_{lk}$ for some arbitrary local unitary frame $(e_k)_{k=1}^n$ of $T_{\C}U, U \subset M$ and a local function $b \in C^\infty(U), U \subset M$ using again (\ref{liebla2}) and the second line in Lemma \ref{cliffharm}; then $b_i$ over $U$ in (\ref{charge}) is given by $b_i=ib\cdot c_1\sum_{k=1}^n\Omega_i(e_k)$. Note that these observations also prove (\ref{charge2}).\\ 
For the second assertion, first note that if we set $\mathcal{E}_i=u_i (du_i)^{\#_\omega}$, then it follows from a direct calculation involving $\nabla\omega=0$ and $du_i((du_i)^{\#_\omega})=1$ that $\nabla(\nabla \mathcal{E}_i)=0$. We claim that the form $\tilde \alpha_i$ in (\ref{hamilton}) which is on $M\setminus \mathcal{C}_i$ given by
\[
\tilde \alpha_i=u_i\frac{\alpha_i}{\omega(\alpha_i^{*}, J\alpha_i^{*})},
\]
is closed. For this, note that for $X \in \Gamma(TM)$ and since $d\alpha_i=0$ we have
\begin{equation}\label{bla6}
X.\omega(\alpha_i^{*}, J\alpha_i^{*})=X.\alpha_i(J\alpha_i^{*})=J\alpha_i^*.(\alpha_i(X))+\alpha([J\alpha_i^{*}, X]).
\end{equation}
For any $x \in M\setminus \mathcal{C}_i$ s.t. $\alpha_i\neq 0$ there is a neighbourhood $x \in U\subset M$ so that $(x_1=u_i, x_2, \dots,x_{2n})\subset \R^{2n}$ are coordinates on $M$, that is a diffeomorphism $\phi:U\rightarrow \R^{2n}$ s.t. $\Phi_*(u_i)=x_1$ adapted to the foliation given by $u_i=const.$ on $U$, that is $(u_i=c, x_2, \dots,x_{2n}), c \in \R$ are local coordinates on the leaves $\mathcal{F}_c \subset U$ of this foliation on $U$ and we can assume that $dx_2=\phi^*(du_i\circ J)$ on $U$. Choosing $X\in \Gamma(\mathcal{F}_c)$ to be one of the coordinate vector fields $X_i=\phi^{-1}_*(\frac{\partial}{\partial x_i}),\ i\geq 2$ we see that (\ref{bla6}) vanishes. Thus $\alpha_i\wedge d(\omega(\alpha_i^{*}, J\alpha_i^{*}))=0$, implying $d\tilde \alpha_i=0$. Considering now $\mathcal{E}_i=u_i(du_i)^{\#_\omega} = (\tilde \alpha_i)^*$ we see that as long as $M\setminus \mathcal{C}_i$ is open, $\mathcal{E}_i$ satisfies (4.) in Definition \ref{frobenius} on $M$.\\
Assume now first that $\nabla$ is torsion-free. Then by Proposition \ref{ham1} and Theorem \ref{lieabl}, we have
\[
\mathcal{H}_x\varphi_i =(\mathcal{L}_\mathcal{E}-\nabla_\mathcal{E})_x\varphi_i =\frac{i}{2}\sum_{j=1}^{n}\lbrace
\nabla_{e_{j}}\mathcal{E}\cdot f_{j}-\nabla_{f_{j}}\mathcal{E}\cdot e_{j}\}\cdot \varphi_i
\]
for $\varphi \in \Gamma_U(\mathcal{L}_i)$ for any $i \in \{1,\dots, k\}$. Here we chose a symplectic frame $(e_1, \dots e_n, e_{n+1}=f_1,\dots, e_{2n}=f_{n})$ at $x \in U$ and extend over $U$ so that $\nabla e_j=0, \nabla f_j=0 \ j=1, \dots, n$ at $x$. Since $\nabla J=0$, we can assume that $f_j=J e_j,\ j=1, \dots, n$ over $U$. Let $du_i=\alpha_i \in \Omega^1(U, {\rm End}(\mathcal{L}_i))$ the eigenform of $\Omega$, acting on $\mathcal{L}_i$, with $u_i \in C^\infty(U, \mathbb{C})$ its local primitive. We show the assertion (\ref{hamiltonian}) as an equality of endomorphisms of $(\mathcal{L})_x$. Then note that since $du(e_j)=-idu(f_j), \  j=1,\dots, n$ by definition of $\Omega$, $(du_i)^{\#_\omega}(\cdot )$, interpreted as element of $TM=(T^*M)^*$ evaluated on $\varphi \in \Gamma(\mathcal{L}_i)$ and on $X^* \in \Gamma(T^*U)$ equals $(X^*, \varphi) \mapsto (X-iJX)\cdot \varphi$. Hence using the basis above, we can write on $U$ if $du_i=\sum_{j=1}^{2n}\beta_je_j^{\#_\omega}$ and using that $\beta_{j+n}=-i\beta_j, \ j=1, \dots, n$ by the definition of $\Omega$ and since $e_j^{\#_\omega}=\omega(e_j, \cdot), j=1,\dots, n$, $f_j^{\#_\omega}=\omega(Je_j, \cdot), j=n+1,\dots, 2n$: 
\[
\mathcal{E}_i\cdot \varphi=u_i(\sum_{j=1}^{n}\beta_je_j+\beta_{n+j}e_{j+n})\cdot \varphi= u_i\sum_{j=1}^n\beta_j(e_j-if_j)\cdot \varphi
\]
where $\beta_j \in C^\infty(U)$. Since $\nabla e_i=\nabla f_i=0$ at $x\in U$ for all $i=1, \dots, n$, we have with this identification
\begin{equation}\label{bla2}
(\nabla_{e_k}\mathcal{E}_i)\cdot \varphi =\left(\sum_{j=1}^n du_i(e_k)\beta_j(e_j-if_j)+u_i\sum_{j=1}^n d\beta_j(e_k)(e_j-if_j)\right)\cdot \varphi.
\end{equation}
Note that in both formulae above, $\cdot$ denotes symplectic Clifford multiplication (not Frobenius multiplication $\star$). Now consider the calculation for any $j\in \{1,\dots,n\}$:
\[
\begin{split}
\nabla_{e_j-if_j}\mathcal{E}_i\cdot (e_j+if_j)&= \nabla_{e_j-if_j}\mathcal{E}\cdot e_j + i \nabla_{e_j-if_j}\mathcal{E}\cdot f_j\\
&=\nabla_{e_j}\mathcal{E}\cdot e_j- i \nabla_{f_j}\mathcal{E}\cdot e_j+ i\nabla_{e_j}\mathcal{E}\cdot f_j+ \nabla_{f_j}\mathcal{E}\cdot f_j,
\end{split}
\]
while
\[
\begin{split}
\nabla_{e_j+if_j}\mathcal{E}_i\cdot (e_j-if_j)&= \nabla_{e_j+if_j}\mathcal{E}\cdot e_j - i \nabla_{e_j+if_j}\mathcal{E}\cdot f_j\\
&=\nabla_{e_j}\mathcal{E}\cdot e_j+ i \nabla_{f_j}\mathcal{E}\cdot e_j- i\nabla_{e_j}\mathcal{E}\cdot f_j+ \nabla_{f_j}\mathcal{E}\cdot f_j.
\end{split}
\]
Substracting both entities and summing over $j$ yields
\[
\sum_{j=1}^{n}\lbrace\nabla_{e_j-if_j}\mathcal{E}\cdot (e_j+if_j)-\nabla_{e_j+if_j}\mathcal{E}\cdot (e_j-if_j)\rbrace=-2i\sum_{j=1}^{n}\lbrace i \nabla_{f_j}\mathcal{E}\cdot e_j- i\nabla_{e_j}\mathcal{E}\cdot f_j\rbrace=4\mathcal{H}_x.
\]
Plugging ${e_j-if_j}$ resp. ${e_j+if_j}$ into the argument of $\nabla_{(\cdot)}\mathcal{E}_i$ in (\ref{bla2}), we see that the terms on the left hand side of the equation are at $x$ linear combinations of $\mathcal{F}^{-1}\circ \Omega^t\cdot\Omega\circ \mathcal{F}$ (note that adjoining by $\mathcal{F}$ interchanges $\Omega$ and $\Omega^t$) and $\Omega\cdot \Omega$ in the second case. This gives the assertion in the case that $\nabla$ is torsion-free. If $\nabla$ is not torsion-free, we get by (\ref{liebla2}) an additional constant $c_3$ in the asserted formula. Finally $c_1 \neq 0$ follows since $\Omega_x \neq 0$ and (\ref{bla2}).
\end{proof}
We will call (semi-simple, weak) Frobenius structures whose multiplication and Euler vector field are induced by a compatible complex structure satisfying $\nabla J=0$ (a $\hat U(n)$-reduction $P^J$ of $P$) in the sense of Proposition \ref{euler}, that is $\Omega$ is given by the map $X\mapsto (X-iJX) \in {\rm End}(\mathcal{L})$ and $\mathcal{E}$ is on appropriate open sets $U\subset M$ of the form $\mathcal{E}_i=u_i (du_i)^{\#_\omega}$ for local primitives $u_i$ of the eigenforms $\alpha_i \in \Omega ^1(U, \C)$ of $\Omega$ on each irreducible suprepresentation $\mathcal{L}_i$ of $\mathcal{A}$, {\it standard}. Such Frobenius structures thus depend on the choice of a $\hat U(n)$-reduction $P^J$ of a given metaplectic structure $P$ on $M$. We will say two standard (not neccessarily irreducible) Frobenius structures are equivalent if the underlying $\hat U(n)$-structures $P^J$ are isomorphic and the pairs $(\mathcal{A}_x, \mathcal{L}_x)$ are equivalent as algebra representations for any $x \in M$. Then it already follows that the respective $\Omega \in \Omega^1(M, End(\mathcal{L}))$ are conjugated and the respective spectral covers $L$ coincide. We have the following classification result in the case of trivial $\hat U(n)$-reductions of $P$. Note that any section $s:M\rightarrow H_n\times_\rho Mp(2n, \R)/G_0$ (cf. Section \ref{coherent}) gives rise to a map ${\rm pr}_1\circ s:M\rightarrow H_n\simeq \R^{2n}$.
\begin{prop}\label{classi}
Assume $R^J$ is an $U(n)$-reduction of the symplectic frame bundle $R$ of $M$ with associated symplectic connection satisfying $\nabla J=0$ and that has a global section $\tilde s:M\rightarrow R^J$ which lifts to a global section $\overline s: M\rightarrow P^J$ in the corresponding $\hat U(n)$ reduction $P^J$ of $P$, where $P^J$ is a given $\hat U(n)$-reduction of $P$. Then the set of equivalence classes of irreducible semi-simple (weak, in the case of $(M, J)$ Kaehler regular, in general singular) standard Frobenius structures whose underlying $\hat U(n)$-structures are isomorphic to $P^J$, is (after eventually homotoping the pair $(\nabla,J)$ with the deformation of $J$ preserving the isomorphy class of $P^J$) parametrized by the set of maps $s:M\rightarrow H_n\times_\rho Mp(2n, \R)/G_0$ so that $d({\rm pr}_1\circ s)=0$ with the notation of Section \ref{coherent}, using that $T^*M\simeq M\times \R^{2n}$.
\end{prop}
\begin{proof}
Let $J_0=J$ and $R^{J_0}$ be the corresponding trivial $U(n)$-reduction of $R$ over $M$. A global section $\tilde s:M \rightarrow R^{J_0}$ is of the form $\tilde s(x)=(e_1, \dots, e_n, f_1, \dots, f_n)$, $f_n=J_0e_n$ for any $x \in M$, let $\overline s: M\rightarrow P^{J_0}$ be the corresponding lift defining a trivialization $P^{J_0}\simeq M\times Mp(2n, \R)$. Replacing $J_0$ by $J=gJ_{0}g^{-1}$ for global sections $g: M\rightarrow Sp(2n, \R)$ resp. lifts $\overline g: M\rightarrow Mp(2n, \R)$ induced by $Sp(2n,\R)$ resp. $Mp(2n,\R)$ acting on the second factor in $R^{J_0}$ resp. $P^{J_0}$, parametrizes the set of $\hat U(n)$-reductions of $P$ which are equivalent to $P^{J_0}$. To each smooth map $h:M\rightarrow (\R^{2n}, 0)\subset H_n$ we can associate over any $x \in M$ the pair $(\C\cdot f_{h(x), T(x)}, \mathcal{A}_2(\R^{2n}, J_T(x)))$, $J_T(x)=J_x$ giving (relative to $\overline s$) the standard irreducible Frobenius structure associated to $J$ and $h$. Hence we are left with considering Proposition \ref{class} resp. Corollary \ref{equivalence}, which give the result immediately since the equivalence classes of pairs of irreducible representations and algebras $\kappa^{2}_{h,T}$ are parameterized by $H_n\times_\rho Mp(2n, \R)/G_0$, once we prove that a symplectic connection $\nabla$ preserving $J$ preserves the section $\mu_2(s, (\C\cdot f_{0, iI},\mathcal{A}_2(\R^{2n}, J_{iI}))):M\rightarrow \mathcal{A}_2$ where $s(x)=(h(x), \overline s(x))$ in the sense that it preserves the line bundle ${\rm im}({\rm pr}_1\circ\mu_2(s, (\C\cdot f_{0, iI},\mathcal{A}_2(\R^{2n}, J_{iI}))))$ (understood relative to the global section $\overline s.\overline g:M\rightarrow P$) and satisfies (2.) in Definition \ref{frobenius}. But the latter follows from $\nabla J=0$, for the former note that on a nghbhd $U$ around any $x \in M$ we can find a Darboux coordinate system on $U\simeq \R^{2n}\simeq T^*\R^n$ centered at $x$ of the form $x_s=(x_1, \dots, x_n, \xi_1, \dots, \xi_n),\ x_i, \xi_i \in C^\infty(\R^n)$ so that $\xi_1=dx_1={\rm pr}_1\circ s|U$ (given that $d({\rm pr}_1\circ s)=0$). Now we claim that in the case $\nabla$ is torsion free (thus $M$ Kaehler) and $\omega (\xi, \cdot)= \xi_1, \xi \in \Gamma(T\R^n)$ we have $\nabla_X \xi=0$ on $U$ for all $X \in \Gamma(U)$. To see this note that in this case, $0=d\xi_1(X,Y)=\omega(\nabla_X\xi, Y)-\omega(\nabla_Y\xi, X), \ X,Y \in \Gamma(TU)$. Note that $\nabla_X\xi \notin {\rm ker}(\xi_1)$ for any $X \in \Gamma(U)$. So let $Y(x)=J\xi(x)$ for the fixed $x \in U$ above, $\nabla Y=0$, $\nabla X=0$ on $U$ while $X(x) \in {\rm ker}(\xi_1)$. Then we infer from $d\xi_1=0$ that $\omega(\nabla_X\xi, Y)=0$ at $x$. If $X(x)\in \C\cdot J\xi(x)$ while $\nabla X=0$ on $U$ we deduce using $\nabla J=0$ that $\omega(\nabla_{J\xi}J\xi, J\circ Y)=0$ at $x$, so $\nabla_X\xi=0$ for all $X$ on $U$. If $\nabla$ is not torsion-free it is easy to see that the torsion term in $d\xi_1(X,Y)$ for $X,Y$ tangent to the coordinate leaves of the above Darboux coordinates, is zero except for $d\xi_1(\xi,J\xi)=\xi.\omega(\xi, J\xi)=\xi.|\xi|^2$ with $\nabla J=0$ being essential. We now claim that we can homotope $(\nabla, J)$ on $M\setminus \mathcal{C}$ to $(\tilde \nabla, \tilde J)$ defined on $M\setminus \mathcal{C}$, so that we have $\tilde \nabla J=0$, $\tilde \nabla\omega=0$ and on $M\setminus \mathcal{C}$, we have $\tilde \nabla \xi_1=0$. On $U_r=\{y \in U, \xi(y)\neq 0\}$ we can homotope $J$ to $\tilde J:TU\rightarrow TU$ where $\tilde J$ is equal to $\frac{1}{|\xi|^2}\cdot J$ on ${\rm ker}(\xi_1)$, equal to $|\xi|^2\cdot J$ on $J({\rm ker}(\xi_1))$ and equal to $J$ else and on $U\setminus U_r$. As one checks, $\tilde J$ defines a global $\omega$-compatible almost complex structure $\tilde J:TM\rightarrow TM$ on $M\setminus \mathcal{C}$. There is furthermore a symplectic connection $\tilde \nabla$ so that $\tilde \nabla \tilde J=0$. Then, as in the torsion-free case, we infer $\tilde \nabla_X\xi=0$ for any $X \in \Gamma(U)$. We will denote $\tilde J, \tilde \nabla$ in the following as $J, \nabla$ again (and assume to work on $M\setminus \mathcal{C}$ in the non-torsion-free case). Note that $x_s$ defines a frame $\overline s_D:U \rightarrow P$ by taking the $g=\omega(\cdot,J\cdot)$-gradients of $(x_1, \dots,x_n)\in C^\infty(\R^n)$ and the $\omega$-duals of $(\xi_1, \dots, \xi_n)$ and lifting these to $T^*\R^n$. Let $L$ be the isotropic foliation of $T^*U$ given by $(x_1, \xi_1)=c \in \R^2$. We thus can apply the construction of Fedosov \cite{fedosov} to the symplectic reduction $T^*U/L$ and thus find a (weakly) normal Darboux coordinate system $x_s^\nabla=(\tilde x_i, \tilde \xi_i), i=1,\dots, n$ on $U$ relative to the symplectic connection $\nabla$ that coincide with $x_s|L$ up to infinite order at $x$, that is $(x_1, \xi_1)-(\tilde x_1, \tilde \xi_1) =O(|x_s|^\infty)$ and we can thus assume that $x_s$ is weakly normal Darboux at $x$. Considering the family of line bundles and commutative algebras over $U$
\[
\begin{split}
(\mathcal{L}, \mathcal{A})|U&=[\overline s.\overline g,\mu_2(s, (\C\cdot f_{0, iI},\mathcal{A}_2(\R^{2n}, J_{iI})))]\\
&=[\overline s_D,\mu_2((\tilde h(x), \overline s(x). \hat g), (\C\cdot f_{0, iI},\mathcal{A}_2(\R^{2n}, J_{iI})))]
\end{split}
\]
for an appropriate section $\hat g:U\rightarrow Mp(2n, \R)$ and $\tilde h: U\rightarrow \R^{2n}$ using the formulae in Proposition \ref{class} and noting that the coordinates $\tilde h_i$ of $s$ wrt $\overline s_D$ on $U$ satisfy $d\tilde h_i|x= O(|x|^\infty)$) we arrive at the assertion. Note that we used that assuming the frame $\overline s^J=\overline s.\overline g:U\rightarrow P$ takes values in the $J$-reduction $P^J$ of $P$, $\nabla$ is reduced to $P^J$ and satisfies furthermore $\nabla \overline s^J=0$, where we think of $\nabla$ as a connection on the trivial $\hat U(n)$-bundle $P^J$, we can since $\nabla \overline s_D|x=0$ also assume that $\nabla \hat g|x=0$, considering $\hat g$ as a section of the trivial bundle $U \times Mp(2n, \R)$.
\end{proof}
Note that we here identified the algebras $\mathcal{A}_2(\R^{2n}, J_T)$ for any $T \in \mathfrak{h}$ according to Lemma \ref{algebras}. If we do not identify them, the irreducible semisimple standard (weak) Frobenius structures would be parametrized by sections $s:H_n\times_\rho Mp(2n, \R)/(G_0\cap G_U)$ satisfying the same integrability conditon with the notation of Section \ref{coherent}, Proposition \ref{class}. To generalize the above, consider any closed subgroup $\tilde G\subset \hat U(n)\subset Mp(2n,\R)$, let $G=H_n\times_\rho Mp(2n, \R)$ and let $i:\tilde G \hookrightarrow G, i(\tilde G)=H_n\times_\rho \tilde G$ the standard embedding. Let $BG$, $B\tilde G$ be the classifying spaces of $G$ and $\tilde G$, respectively, $EG\rightarrow BG$ the universal bundle. Then it is well-known that a principal $G$-bundle $\hat P$ over $M$ can be reduced to a $\tilde G$-bundle $Q$ that is $\hat P \simeq Q\times_{\tilde G, i} G$ for some $\tilde G$-bundle $Q$ (where the notation $Q\times_{\tilde G, i} G$ refers to the balanced product induced by $i:\tilde G\rightarrow G$, compare (\ref{balance})), if there exists a lift of the classifying map $f:M\rightarrow BG$ for $\hat P$ so that following diagram commutes:
\begin{equation}\label{universal}
\begin{diagram}
&  &B\tilde G=EG \times_G G/\tilde G \\
&\ruTo^{\tilde f} &\dTo\\
M  &\rTo_{f} &BG
\end{diagram}
\end{equation}
and the homotopy class of lifts $\tilde f$ parametrize the isomorphism classes of $(Q,i)$-reduction of $\hat P$. The homotopy class of lifts of $f$ in turn defines a homotopy-class of sections $s:M\rightarrow f^*(EG \times_G G/\tilde G)\simeq \hat P \times_G G/\tilde G$. Let now $P_{\tilde G}$ be a fixed $\tilde G \subset \hat U(n)\subset Mp(2n,\R)$-reduction of a given metaplectic structure $P$, let again $i:\tilde G \hookrightarrow G=H_n\times_\rho Mp(2n, \R)$ resp. $i:Mp(2n,\R) \hookrightarrow G$ be the standard embeddings. Consider the $G$-principal bundle $\hat P$ induced by $i$, that is
\begin{equation}\label{reduction}
\hat P=P_{\tilde G}\times_{i,\tilde G} G\simeq P\times_{i, Mp(2n,\R)} G.
\end{equation}
Note that we show in Proposition \ref{higgs} that $\hat P$ is isomorphic to the bundle $P_G$ introduced in (\ref{tangent}). Then, by the above $P_{\tilde G}$ is tautologically a $\tilde G\subset G$-reduction of the $G$-bundle $\hat P$ and the isomorphism classes of $\tilde G$ reductions of $\hat P$ are parametrized by the above arguments by the homotopy classes of sections of
\begin{equation}\label{pg}
\hat P_{G/\tilde G}=\hat P\times_{G}(G/{\tilde G})\rightarrow M.
\end{equation}
On the other hand, two isomorphic $\tilde G\subset \hat U(n)$-reductions of $\hat P$ with $\tilde G\hookrightarrow H_n\times_\rho Mp(2n, \R)$ the standard embedding are also isomorphic as $\tilde G$-reductions $P_{\tilde G}$ of $P$ since the latter are in bijective correspondence with the homotopy classes of global sections of the associated bundle $P \times_{\tilde G} Mp(2n, \R)/{\tilde G}$ (considering (\ref{pg}) mod $H_n\subset G$). For the following, let $P_G=\pi_P^*(TM) \times_{Mp(2n, \R)} P$ be defined as in (\ref{tangent}) below and denote by ${\rm pr}_1:P_G\rightarrow T^*M$ the map $pr_1((y,q), x), (p, x))=((gy,q),x),\ x \in M,\ y \in \R^{2n},\ p,q \in P_{\tilde G},\ q=p.g,\ g \in \tilde G$ (using the complex structure $J$ corresponding to $P_{\tilde G}$ to identify $TM\simeq T^*M$). Note further that ${\rm pr}_1$ factors to a map $\tilde {\rm pr}_1: P_{G/\tilde G}\rightarrow T^*M$ when considering $P_{G/\tilde G}$ as a quotient $r:P_G\rightarrow P_{G/\tilde G}$, that is ${\rm pr}_1=\tilde {\rm pr}_1\circ r$. Using the above, we can then deduce:
\begin{theorem}\label{genclass}
For a given closed subgroup $\tilde G\subset \hat U(n)\subset Mp(2n,\R)$ and a fixed metaplectic structure $P$ on $M$, the set of irreducible semi-simple (weak, in the case of $(M, J)$ Kaehler regular, in general singular) standard Frobenius structures whose underlying $\tilde G$-structure $P_{\tilde G}$ is a $\tilde G$-reduction of $P$ and the associated symplectic connection preserves the almost complex structure $J$ associated to $P_{\tilde G}$, is in bijective correspondence to the set of sections $s$ of $P_{G/\tilde G}$ in (\ref{pg}), so that we have $\hat s=\tilde {\rm pr}_1\circ s:M\rightarrow T^*M$ is closed, that is $d\hat s=0$. Furthermore two such structures $s_1, s_2$ are equivalent if and only if $s_1$ and $s_2$ are homotopic and $j\circ s_1=j\circ s_2$ if we understand $s_i$ as equivariant maps $s_i:\hat P \rightarrow (H_n\times_\rho Mp(2n, \R))/{\tilde G}$ for $i=1,2$ and $j:(H_n\times_\rho Mp(2n, \R))/{\tilde G}\rightarrow (H_n\times_\rho Mp(2n, \R))/G_0$ is the canonical projection.
\end{theorem}
\begin{proof}
The proof follows by the remarks before this Proposition, Proposition \ref{class} resp. Corollary \ref{equivalence} and considering the fact that any section $s$ of $\hat P_{G/\tilde G}$, considered as an equivariant map $s:\hat P\rightarrow G/\tilde G$ defines an equivariant map $\hat s:\hat P \rightarrow  \mathcal{A}_2$ (for $V=\R^{2n}$ in $\mathcal{A}_2$ and $\mathcal{A}_2$ given as in (\ref{action1})) by setting
\begin{equation}\label{equiv}
p \in \hat P \mapsto \mu_2\left(s(p),(\C\cdot f_{0, iI},\mathcal{A}_2(\R^{2n}, iI)\right).
\end{equation}
Consider now the quotient bundle $E_G\rightarrow E_G/\tilde G=B_{\tilde G}$ which is a $\tilde G$-bundle over $B_{\tilde G}$ which we denote by $\tilde E_G$ and $E_G, B_{\tilde G}$ are as above, we can thus form the associated bundle $\mathcal{E}=\tilde E_{G}\times_{\tilde \mu_2\circ i} \mathcal{A}_2\rightarrow B_{\tilde G}$. 
Note that by $\tilde \mu_2$ we denote the action of (closed subgroups and quotients of) $G$ on $\mathcal{A}_2$ given by the explicit isomorphism $\mathcal{A}_2\simeq G/G_0\cap G_U$ and the action of $G$ on $G/G_0\cap G_U$. Then since $f^*(E_G/\tilde G)=\hat P_{G/\tilde G}$, where $f$ is a classifying map $f:M\rightarrow BG$ for $\hat P$, we see that any section $s$ of $\hat P_{G/\tilde G}$ defines a one dimensional line bundle associated to the $\tilde G$-bundle $s^*(\tilde E_G)\rightarrow M$
\[
\mathcal{E}_M=s^*(\tilde E_{G})\times_{\tilde G, \tilde\mu_2\circ i}\mathcal{A}^0_2\rightarrow M,\ \mathcal{A}^0_2:=(\C\cdot f_{0,iI}, \mathcal{A}_2(\R^{2n}, iI)),
\]
being a line-subbundle of $s^*(\mathcal{E})$ and we claim that $\mathcal{E}_M$ induces a Frobenius structure over $M$ and that all irreducible semisimple standard Frobenius structures arise in this way. This is seen by considering that any $s$ as above defines a reduction of $\hat P$ to $\tilde G$ so that $P_{\tilde G,s} =s^*(\tilde E_{G})$, thus $\mathcal{E}_{\tilde G}:=P_{\tilde G,s}\times_{\mu_2\circ i} \mathcal{A}^0_2 = \mathcal{E}_M$. Then note that the equivariant map $\hat s:\hat P \rightarrow  \mathcal{A}_2$ in (\ref{equiv}) defines a global smooth section $\hat s$ of the associated fibre bundle $\mathcal{E}_{G}=\hat P\times_{G, \tilde \mu_2}\mathcal{A}_2$. Let ${\rm ev}_1:\mathcal{A}_2\rightarrow L^2(\R^n)$ be the map that assigns to a pair $(\C\cdot f_{h, T},\mathcal{A}_2(\R^{2n}, T)) \in \mathcal{A}_2$ the (union of the points of the) complex line $\C\cdot f_{h, T}\subset L^2(\R^n)$ and let $\mathcal{W}={\rm ev}_1(\mathcal{A}_2)\subset L^2(\R^n)$ the image of this map in $L^2(\R^n)$. Then we have that $\tilde s={\rm ev}_1\circ \hat s:\hat P\rightarrow {\rm Gr}_1(\mathcal{W})$ where ${\rm Gr}_1(\mathcal{W})={\rm im}({\rm pr}_1(\mathcal{A}_2))$ denotes the $1$-dimensional subspaces of $\mathcal{W}$ which are of the form $\C\dot f_{h,T}$, is a continuous (smooth) section wrt the gap metric topology of $L^2(\R^n)$ that defines a smooth section $\tilde s: M\rightarrow \mathcal{E}_{{\rm Gr}_1(\mathcal{W})}$ of the fibre bundle 
\[
\mathcal{E}_{{\rm Gr}_1(\mathcal{W})}=\hat P\times_{G, {\rm ev}_1\circ \tilde \mu_2\circ (i, i_{\mathcal{W}})}{\rm Gr}_1(\mathcal{W})\simeq P_{\tilde G}\times_{\tilde G, {\rm ev}_1\circ \tilde \mu_2\circ (i, i_{\mathcal{W}})}{\rm Gr}_1(\mathcal{W})\rightarrow M,
\]
where $i_{\mathcal{W}}:{\rm Gr}_1(\mathcal{W})\rightarrow \mathcal{A}_2$ is just the embedding $i_{\mathcal{W}}(\C\cdot f_{h,T})=(\C\cdot f_{h,T}, \mathcal{A}_2(\R^{2n},T))$. The image of $\tilde s$, that is the (generalized) subbundle $\hat {\mathcal{E}}_M:={\rm im}(\tilde s)\subset \mathcal{E}_{{\rm Gr}_1(\mathcal{W})}$ we claim to be isomorphic to $\mathcal{E}_{\tilde G}=\mathcal{E}_M$. We have to compare the two $\tilde G$-principal bundles
\[
\begin{diagram}
P_{\tilde G,s}=s^*(\tilde E_{G}) &\rTo  &\hat P_{G/\tilde G} &\lTo 
& P_{\tilde G}= s_0^*(\tilde E_{G})\\
&\rdTo &\uTo^{s}\uTo_{s_0} &\ldTo \\
&  &M  &  
\end{diagram}
\]
where $s_0$ corresponds to $P_{\tilde G}$ in the sense of (\ref{universal}) and the discussion below it. Note that it follows from the definition of $\hat P$ that $s_0:\hat P\rightarrow G/\tilde G$ is the map which equals $s_0(p)=Id_{G/\tilde G},\ p \in P_{\tilde G}\subset \hat P$ and is extended to $\hat P$ according to $s_0(p.g)=Id_{G/\tilde G}.g$, where with $P_{\tilde G}\subset \hat P$ we here mean the standard inclusion. Now writing $g(p)s_0(p)=s(p)$ for some equivariant function $g:\hat P\rightarrow G/\tilde G$, we infer that $P_{\tilde G,s}=s^*(\tilde E_{G})$ is embedded as $P_{\tilde G,s}=\{g(p).p\subset \hat P: p \in P_{\tilde G}\}$. Using this we infer that if $p \in P_{\tilde G,s}$ and $(p, (\C\cdot f_{0,iI}, \mathcal{A}_2(\R^{2n}, iI)) \in P_{\tilde G,s}\times_{\mu_2\circ i} \mathcal{A}^0_2$, this defines an element in $\mathcal{E}_{G}=\hat P\times_{G, \tilde \mu_2} \mathcal{A}_2$. Let $g.p \in P_{\tilde G} \subset \hat P$ for some $g \in G/\tilde G$, then in $\mathcal{E}_{G}$ we have $(p, (\C\cdot f_{0,iI}, \mathcal{A}_2(\R^{2n}, iI))) \sim (g.p, \tilde \mu_2\left(g^{-1},(\C\cdot f_{0,iI}, \mathcal{A}_2(\R^{2n}, iI))\right)) \in P_{\tilde G}\times_{G, \tilde \mu_2\circ i} \mathcal{A}_2$ and thus we arrive at the assertion $\hat {\mathcal{E}}_M=\mathcal{E}_M$. That $\nabla$, provided it preserves $J$, also preserves (after eventually modifying the pair $(\nabla, J)$ preserving the isomorphy class of $P_{\tilde G}$ as indicated in the proof of Proposition \ref{classi} and restricting the above data to $M \setminus \mathcal{C}$) the line subbundle $\hat {\mathcal{E}}_M={\rm im}(\tilde s)$ if the integrability condition $d\hat s=0$ with $\hat s=\tilde {\rm pr}_1\circ s:M\rightarrow T^*M$ is satisfied then follows by choosing a local section $P^J$ over $U\subset M$ that determines around any $x \in U$ normal Darboux coordinates on a nghbhd $U_x \subset U$ of $x$ so that its coordinate vector fields are parallel at $x \in M$ wrt $\nabla$. Then expressing $\hat s|U_y$ in these coordinates, noting that the $\omega$-dual of $\hat s$, $\hat s^* \in \Gamma(TU)$ satisfies $\nabla \hat s^*=0$ (cf. the proof of Proposition \ref{classi}) and invoking the explicit formulas in Proposition \ref{class} we see that $\nabla$ preserves $\hat {\mathcal{E}}_M={\rm im}(\tilde s)$ in $x$. An alternative proof of this fact will be given in the course of the proof of Proposition \ref{higgs}.
\end{proof}
Note that in the above the same remark applies as under Proposition \ref{classi}: {\it not} identifying the (isomorphic) algebras $\mathcal{A}_2(\R^{2n}, J_T)$ for any $T \in \mathfrak{h}$ two Frobenius structures $s_1$ and $s_2$ are equivalent if and only if $s_1$ and $s_2$ are homotopic and $j\circ s_1=j\circ s_2$ where in this case $j:(H_n\times_\rho Mp(2n, \R)/{\tilde G}\rightarrow (H_n\times_\rho Mp(2n, \R)/G_0\cap G_U$ is the canonical projection. Note further that the proof of Theorem \ref{genclass} illustrates two ways to understand an irreducible, semi-simple standard Frobenius structure associated to an equivariant map $s:\hat P \rightarrow (H_n\times_\rho Mp(2n, \R))/{\tilde G}$, that is a section of $P_{G/\tilde G}$, on one hand $s$ induces a section of the bundle $\mathcal{E}_{{\rm Gr}_1(\mathcal{W})}=P_{\tilde G}\times_{\tilde G, {\rm ev}_1\circ \tilde \mu_2\circ (i, i_{\mathcal{W}})}{\rm Gr}_1(\mathcal{W})$, on the other hand an irreducible semi-simple standard Frobenius structure can be understood as a line-bundle associated to the $\tilde G$-bundle $P_{\tilde G, s}=s^*(\tilde E_{G})$, namely $\mathcal{E}_M=P_{\tilde G, s}\times_{\tilde\mu_2\circ i}\mathcal{A}^0_2$. Note that this is a correspondence between fibre bundles, there is a priorily no interpretation of tangent vectors of $M$ as elements of ${\rm Aut}(\mathcal{E}_M)$, unless of course by using the above 'reciprocity'. To be precise, if $[s_U,\tilde u] \in \Gamma_U(\mathcal{E}_M), U\subset M$ is a local section, where $s_U:U\subset M\rightarrow P_{\tilde G,s}$ a local section, $\tilde u:U\subset M\rightarrow \mathcal{A}^0_2$, then we write $\tilde s_u=g^{-1}(s_U) s_U:M\rightarrow P_{\tilde G}$ with the equivariant function $g:\hat P\rightarrow G/\tilde G$ from the proof above. Frobenius multiplication of $[\tilde s_U, X],\ X:U\rightarrow \R^{2n}$ and $[s_u, \tilde u]$ is then given by
\begin{equation}\label{frobeniusm}
[\tilde s_U, X]\cdot [s_u, \tilde u]= [s_u, L(g(s_U))\circ \sigma_{g^{-1}(s_U)(iI)}(X)\circ L(g^{-1}(s_U))\tilde u]
\end{equation}
where $\sigma_T:\mathcal{A}_{2}(V)\rightarrow {\rm End}(\mathcal{S}(\R^n)),\ T \in \mathfrak{h}$ was defined above Corollary \ref{equivalence}.
Note that this follows from the representation of Frobenius multiplication in the bundle $\mathcal{E}_{{\rm Gr}_1(\mathcal{W})}=P_{\tilde G}\times_{\tilde G, {\rm ev}_1\circ \tilde \mu_2\circ (i, i_{\mathcal{W}})}{\rm Gr}_1(\mathcal{W})$ and the equivalence of $\mathcal{E}_M$ and the image of the section $\tilde s$ of $\mathcal{E}_{{\rm Gr}_1(\mathcal{W})}$ induced by $s$ above as vector bundles associated to $\tilde G$-reductions of the $G$-bundle $\hat P\simeq P_G$. We will see below that the degree $1$-part of the corresponding Frobenius algebra has an interpretation in terms of a 'Cartan-geometry type' connection in $P_{\tilde G,s}$.\\
As was noted above, $P_{\tilde G}$ resp. $P_{\tilde G,s}$ can be considered as subsets of $\hat P$ by considering $P_{\tilde G,s}=s^{-1}(Id_{G/\tilde G})\subset \hat P$ when writing $s$ as an equivariant map $s:\hat P\rightarrow G/\tilde G$, analogously for $s_0$ and $P_{\tilde G}$. Then the function $g:\hat P \simeq P_G \rightarrow G/\tilde G$ satisfying $s(p)=g(p)s_0(p)$ can be labelled as an automorphism of $\hat P$ that maps $P_{\tilde G}$ to $P_{\tilde G,s}$. Assume that for $s$ and $s_0$, composed with the canonical projection $\pi_{Mp}: G\rightarrow Mp(2n,\R)/\tilde G$ and considering $P\subset \hat P$ we have that $\tilde s=\pi_{Mp}\circ s|P:P\rightarrow Mp(2n,\R)/\tilde G$ and $\tilde s_0=\pi_{Mp}\circ s_0|P:P\rightarrow Mp(2n,\R)/\tilde G$ are isotopic, thus defining equivalent $\tilde G$-reductions $P_{\tilde G}$ of $P$.
Recall the projection ${\rm pr}_1:P_G\rightarrow T^*M$ factoring to $\tilde {\rm pr}_1: P_{G/\tilde G}\rightarrow T^*M$ as above, assume that $s$, considered as a map $s:M \rightarrow P_{G/\tilde G}$, satsifies $d\hat s=0$, with $\hat s=\tilde {\rm pr}_1\circ s:M\rightarrow T^*M$. Call the associated automorphisms $g:P_G \rightarrow G/\tilde G$ satisfying the latter condition {\it closed}. Define the automorphisms ${\rm Aut}_c(\hat P,P_{\tilde G})$ as the set of automorphisms $g:\hat P\rightarrow G/\tilde G$ which fix a given equivalence class of $\tilde G$-reductions $P_{\tilde G}$ of $P$ in the above way and are {\it closed} in the above sense. Then we have by Theorem \ref{genclass}:
\begin{folg}\label{classf}
The elements of ${\rm Aut}_c(\hat P,P_{\tilde G})$ are in one to one correspondence with the irreducible standard (weak, in general singular) Frobenius structures whose underlying $\tilde G\subset \hat U(n)$-reductions of $P$ are equivalent to $P_{\tilde G}$.
\end{folg}
Note that in the above, we considered a given $\tilde G\subset \hat U(n)$-reduction $P_{\tilde G}$ of the $Mp(2n,\R)$-bundle $P$ as inducing  canonically a $\tilde G$-reduction of the $G$-bundle $\hat P$ by the definition of $\hat P$ in (\ref{reduction}) and used the same notation for both, strictly speaking different, objects.\\
Based on Proposition \ref{classN}, we will finally review the analogue of Theorem \ref{genclass} in the case of not neccessarily irreducible, but indecomposable (standard) Frobenius structures $(\mathcal{A}, \mathcal{L})$, where we understand {\it standard} as in the previous semi-simple case, that is $\Omega \in \Omega^1(M, End(\mathcal{L}))$ is realized by the map $X\mapsto (X-iJX)$, acting by symplectic Clifford multiplication on $\mathcal{L}$. To avoid a trivial re-statement of Proposition \ref{genclass} and for later applications, we will furthermore allow for pairs $(\mathcal{A}, \mathcal{L})$ on the smooth symplectic manifold $M$ so that the generalized subbundle $\mathcal{L}\subset \mathcal{Q}$ is smooth, that is there are for any $U\subset M$ smooth local generators $\xi_i:U\subset \mathcal{L},\ {\rm span}_{\C}\{\xi_i(x)\}_{i=1}^l=\mathcal{L}_x, x \in U$ but the fibre dimension $x\mapsto {\rm dim}(\mathcal{L}_x)$ is not neccessarily locally constant on $M$ for any $x \in M$, the points in $M$ where this fails to hold being called singular. We then assume in the following that there exists a finite increasing sequence of closed submanifolds $\emptyset =X_0\subset X_1\subset X_2\subset \dots \subset X_r=M$, so that the dimension of the fibres of $\mathcal{L}|X_k\setminus X_{k-1},\ k\in \{1,\dots,r\}$ is locally constant, thus on any stratum $X_k, \ k\in \{1,\dots,r\}$, the set of singular points of $\mathcal{L}|X_k$ is given by $X_{k-1}$. Finally we will assume for the following that the normal bundle $N_{k-1}$ of $X_{k-1}$ in $TM$ is trivial and for the on any $X_k\setminus X_{k-1}$ locally constant numbers $d_k={\rm dim}(\mathcal{L}|X_k\setminus X_{k-1})$ we have $d_{k-1}-d_k={\rm dim}\ N_{k-1}-{\rm dim}\ N_{k}$. Denote by $\mathcal{X}=\{X_0, X_1, \dots, X_r\}$ the set of strata on $M$. We will call a Frobenius structure $(\mathcal{A}, \mathcal{L})$ in the above sense smoothly subordinated to $\mathcal{X}$.\\
Let for the below $\tilde G=\hat U(n)\subset Mp(2n,\R)$ or $\tilde G=\hat O(n)\subset Mp(2n,\R)$. As above, we then let $P_{\tilde G}$ be a fixed $\tilde G$-reduction of a given metaplectic structure $P$ on $M$, let again $i:\tilde G \hookrightarrow G=H_n\times_\rho Mp(2n, \R)$ resp. $i:Mp(2n,\R) \hookrightarrow G$ be the standard embeddings and the corresponding principal bundle $\hat P$ induced by $i$ as defined in (\ref{reduction}), recall that $\hat P\simeq P_G=\pi_P^*(TM) \times_{Mp(2n, \R)} P$ as we prove below in Proposition \ref{higgs} where $P_G$ is introduced in (\ref{tangent}). Then, by the above, $P_{\tilde G}$ is tautologically a $\tilde G\subset G$-reduction of the $G$-bundle $\hat P$ and the isomorphism classes of $\tilde G$-reductions of $\hat P$ are parametrized again by the homotopy classes of sections of $\hat P_{G/\tilde G}$ in (\ref{pg}). The generalization of Theorem \ref{genclass} to the case of indecomposable standard Frobenius structures on a stratified (but smooth) symplectic manifold $M$ is then, fixing the same notations for ${\rm pr}_1:P_G\rightarrow T^*M$ and $\tilde {\rm pr}_1: P_{G/\tilde G}\rightarrow T^*M$ as discussed above Theorem \ref{genclass} given by the following. Recall that in above Proposition \ref{classN}, we defined $\mathcal{P}_0(K_N)\subset \N_0^n$ to be the set of all subsets $K\subset K_N\subset \N_0^n$ which are chain incident to $0$. We call the set of all $\{i \in \{1,\dots,n\}: \exists (k_1, \dots, k_n)\in K_N \ {\rm s.t.}\ k_i\geq 1\}$, the singular support of $K_N$, denoted ${\rm supp}_s(K_N)$.
\begin{theorem}\label{genclassN}
For a given closed subgroup $\tilde G\subset \hat U(n)\subset Mp(2n,\R)$ and a fixed metaplectic structure $P$ on $M$, the set of  indecomposable (weak, in the case of $(M, J)$ Kaehler regular, in general singular) standard Frobenius structures smoothly subordinated to a stratification $\mathcal{X}$ of $M$ as described above whose underlying $\tilde G$-structure $P_{\tilde G}$ is a $\tilde G$-reduction of $P$ and the given symplectic connection preserves the almost complex structure $J$ associated to $P_{\tilde G}$, is in bijective correspondence to the set of pairs $(s, K)$ where $s$ is a smooth section of $P_{G/\tilde G}$ in (\ref{pg}) and $K:\mathcal{X}\rightarrow \bigcup_{N\in \N_0} \mathcal{P}_0(K_N)$ is so that $|{\rm supp}_s(K_N(X_k))|={\rm dim}(N_k),\ k \in \{1,\dots,r\}$, so that we have $\hat s=\tilde {\rm pr}_1\circ s:M\rightarrow T^*M$ is closed, that is $d\hat s=0$. Furthermore two such structures $(s_1, K_1), (s_2,K_2)$ are equivalent if and only if $K_1=K_2$ as maps on $\mathcal{X}$ and $s_1$ and $s_2$ are homotopic and $j\circ s_1=j\circ s_2$ if we understand $s_i$ as equivariant maps $s_i:\hat P \rightarrow (H_n\times_\rho Mp(2n, \R))/{\tilde G}$ for $i=1,2$ and $j:(H_n\times_\rho Mp(2n, \R))/{\tilde G}\rightarrow (H_n\times_\rho Mp(2n, \R))/G_0$ is the canonical projection.
\end{theorem}
\begin{proof}
The proof is largely analogous to the proof of Theorem \ref{genclass}. Any fixed trivialization $\mathfrak{n}_k: X_k\rightarrow N_K\subset TM$ defines a further reduction $P_{\tilde G_k}$ of $P_{\tilde G}|X_k$ to the subgroup $\tilde G_k\subset \tilde G$, where $P_{\tilde G_k}$ consists of the set of those frames in $P_{\tilde G}|X_k$ which are adapted to the decomposition of $TM|X_k= N_k\oplus TX_k$: we take any fixed connected component of $P_{\tilde G_k}$, all choices will lead to isomorphic structures, we assume that the ${\rm dim}\ N_k$ elements of any frame corresponding to ${\rm supp}_s(K_N(X_k))$ in $\pi_k(p),\ p\in P_{\tilde G_k}, \pi_k: P_{\tilde G_k} \rightarrow R_{\tilde G_k}$, where $R_{\tilde G_k}$ is the corresponding reduction of $R|X_k$, define a global section of $R_{Mp(2n,\R)/\tilde G_k}=P|X_k\times_{Mp(2n,\R)}(Mp(2n,\R)/{\tilde G_k})\rightarrow X_k$. Then noting that $\hat P|X_k=P_{\tilde G_k}\times_{i_k,\tilde G_k} G,\ i_k: \tilde G_k\hookrightarrow G$ we can extend $s$ as a section of $P_{G/\tilde G}|X_k\rightarrow X_k$ via the above to a (homonymous) section $s$ of $\hat P_{G/\tilde G_k}=\hat P|X_k\times_{G}(G/{\tilde G_k})\rightarrow X_k$. Then $K:\mathcal{X}\rightarrow \bigcup_{N\in \N_0} \mathcal{P}_0(K_N)$, restricted to $X_k\subset M, \ k\in \{1, \dots,r\}$ and $s$ considered as an equivariant map $s:\hat P|X_k\rightarrow G/\tilde G_k$ and the fact that $K|X_k \in \mathcal{P}_0(K_N)$ for some $N \in \N_0$ with $|{\rm supp}_s(K_N(X_k))|={\rm dim}(N_k),\ k \in \{1,\dots,r\}$, define an equivariant map $\hat s:\hat P|X_k \rightarrow  \mathcal{A}^N_2$ (for $V=\R^{2n}$ in $\mathcal{A}^N_2$ and $\mathcal{A}^N_2$ given as in (\ref{action1N})) by setting
\begin{equation}\label{equivN}
p \in \hat P|X_k \mapsto \mu^N_2\left(s(p),(\oplus_{k\in K(X_k)} \C \cdot f_{0, iI, k}, \mathcal{A}_{2}(\R^{2n}, iI))\right), \ k\in \{1,\dots, r\}.
\end{equation}
with the notation from the discussion above Proposition \ref{classN}. Then the equivariant map $\hat s:\hat P|X_k \rightarrow  \mathcal{A}^N_2$ in (\ref{equivN}) defines a global smooth section $\hat s$ of the associated fibre bundle $\mathcal{E}_{G}=\hat P|X_k\times_{G,  \mu^N_2}\mathcal{A}^N_2$. Let ${\rm ev}_1:\mathcal{A}^N_2\rightarrow L^2(\R^n)$ be the map that assigns to a pair $(\mathcal{F}_{h, T, N},\mathcal{A}_2(\R^{2n}, T)) \in \mathcal{A}^N_2$ the (union of the points of) the vector space $\mathcal{F}_{h, T, N}\subset L^2(\R^n)$ and let $\mathcal{W}={\rm ev}_1 (\mathcal{A}^N_2)\subset L^2(\R^n)$ the image of this map in $L^2(\R^n)$. Then we have that $K$, restricted to $X_k, \ k\in\{1,\dots,r\}$ and  $\tilde s={\rm ev}_1\circ \hat s:\hat P\rightarrow {\rm Gr}_1(\mathcal{W})$ where ${\rm Gr}_N(\mathcal{W})={\rm im}({\rm pr}_1(\mathcal{A}^N_2))$ denotes the set of $\leq M^N_n$-dimensional subspaces of $\mathcal{W}$, define a continuous (smooth) section wrt the gap metric topology of $L^2(\R^n)$ $\tilde s: M\rightarrow \mathcal{E}_{{\rm Gr}_N(\mathcal{W})}$ and thus a section of the (family of) fibre bundles 
\[
\mathcal{E}_{{\rm Gr}_N(\mathcal{W})}|X_k=\hat P|X_k\times_{G, {\rm ev}_1\circ \mu^N_2\circ (i, i_{\mathcal{W}})}{\rm Gr}_N(\mathcal{W})\simeq P_{\tilde G_k}\times_{\tilde G_k, {\rm ev}_1\circ \mu^N_2\circ (i_k, i_{\mathcal{W}})}{\rm Gr}_N(\mathcal{W})\rightarrow X_k\subset M,
\]
where $i_{\mathcal{W}}:{\rm Gr}_N(\mathcal{W})\rightarrow \mathcal{A}^N_2$ is the embedding $i_{\mathcal{W}}(\mathcal{F}_{h, T, N})=(\mathcal{F}_{h, T, N}, \mathcal{A}_2(\R^{2n},T))$. The image of $\tilde s$, that is the (generalized) subbundle $\hat {\mathcal{E}}_M|X_k:={\rm im}(\tilde s|X_k)\subset \mathcal{E}_{{\rm Gr}_N(\mathcal{W})}|X_k$ we claim to be isomorphic to $\mathcal{E}_M|X_k$ defined by $K|X_k$ and the section $s$ of $\hat P_{G/\tilde G_k}\rightarrow X_k$ asssociated to the $\tilde G_k$-bundle $P_{\tilde G_k,s} =s^*(\tilde E_G)\rightarrow M$ (where here we understand $\tilde E_G\rightarrow E_G/\tilde G_k=B_{\tilde G_k}$ as defined over $X_k$) as
\[
\mathcal{E}_M|X_k=s^*(\tilde E_{G})|X_k\times_{\tilde G_k, \mu^N_2\circ i}(\mathcal{A}^0_2)^N\rightarrow X_k,\ (\mathcal{A}^0_2)^N:=(\oplus_{k\in K(X_k)} \C \cdot f_{0, iI, k}, \mathcal{A}_2(\R^{2n}, iI)).
\]
The proof of this claim as the remaining assertions are proven in analogy to the proof of Theorem \ref{genclass} and thus omitted here. Note finally that the smoothness $\mathcal{E}_M\simeq \hat {\mathcal{E}}_M$ follows from the fact that the underlying defining section $s:M\rightarrow \hat P_{G/\tilde G}$ is globally smooth on $M$.
\end{proof}
In generalization of and with the notation above Corollary \ref{classf}, that is ${\rm Aut}_c(\hat P,P_{\tilde G})$ is the set of automorphisms $g:\hat P\rightarrow G/\tilde G$ that give a fixed equivalence class of $\tilde G$-reductions $P_{\tilde G}$ of $P$ and are {\it closed} in the sense described above Corollary \ref{classf} we can then state for a given stratification $\mathcal{X}$ and a set of trivializations of their normal bundles $N_k\subset TM|X_k$ of $M$:
\begin{folg}\label{classfN}
The set of pairs $(g, K)$ where $g$ is an element of ${\rm Aut}_c(\hat P,P_{\tilde G})$ and $K$ maps $\mathcal{X}$ to $\bigcup_{N\in \N_0} \mathcal{P}_0(K_N)$ so that $|{\rm supp}_s(K_N(X_k))|={\rm dim}(N_k),\ k \in \{1,\dots,r\}$, is in one to one correspondence with the indecomposable (weak, in general singular) standard Frobenius structures smoothly subordinated to $\mathcal{X}$ whose underlying $\tilde G\subset \hat U(n)$-reductions of $P$ are equivalent to $P_{\tilde G}$.
\end{folg}

\subsection{Spectrum, structure connection and formality}

Inspecting (\ref{bla2}), we see that the expression $\nabla \mathcal{E}_i$ is for any $i\in \{1, \dots, k\}$ for a semi-simple Frobenius structure $(\mathcal{A}, \mathcal{L})$ with $\Omega$-eigenline-bundle splitting $\mathcal{L}= \bigoplus_{i=1}^k\mathcal{L}_i$ an element of $\Omega^1(M, {\rm End}(\mathcal{L}_i))$. In the usual case of Frobenius structures (\cite{dubrovin}) defining a module structure of $\mathcal{O}_{T^*M}$ on $TM$, the presence of a 'flat structure' and canonical coordinates $e_j, j=\{1,\dots, 2n\}$ satisfying $e_j\circ e_i=\delta_{ij}e_i$ and the $e_i$ being orthogonal wrt a given metric imply that the expression $\nabla \mathcal{E}$, $\mathcal{E}$ being the Euler vector field, as an endomorphism of $TM$ is diagonalized by coordinates $t_i$ defining the flat structure with its eigenvalues manifesting the {\it spectrum} of the Frobenius structure. In our situation, we are thus tempted to call the (in general non-closed) forms 
\begin{equation}\label{spec1}
\tilde \alpha_i:=\nabla \mathcal{E}_i \in \Omega^1(M, {\rm End}(\mathcal{L}_i))\simeq \Omega^1(M, \C), \ i\in \{1,\dots,k\},
\end{equation}
where the identification is here given by $\cdot$, the spectrum of the semi-simple Frobenius structure $(\mathcal{A}, \mathcal{L})$, but in more restricted cases we are able to come up with something more intelligible. In the following, we will always assume that if $\mathcal{E}_i$ does not exist globally on $M$, we have chosen a covering $\pi:\tilde M_i\rightarrow M$ so that $\pi^*\alpha_i$ is exact (for instance that associated to ${\rm ker}\ (ev_{\alpha_i}: \pi_1(M)\rightarrow  \R)$), hence $\mathcal{E}_i$ is well defined on $\tilde M$. We will continue to write $M$ instead of $\tilde M_i$, where this causes no confusion. As above we will denote $\mathcal{C}_i$ the critical set of $\alpha_i$ resp. $\pi^*\alpha_i$ in $M$ resp. $\tilde M$, that is the set where $\alpha_i=0$ (resp. $\pi^*(\alpha_i)=0$).
\begin{propdef}\label{spectrum}
Let for the following $M$ be connected and compact or compact with boundary.
\begin{enumerate}
\item Assume $(\mathcal{A}, \mathcal{L})$ is a semi-simple standard Frobenius structure with $\nabla J=0$ and $k=n$ and that $(du_i)^{\#_\omega}(du_j\circ J)=\delta_{ij}$ and $\{u_i,u_j\}=0$ for all $i,j \in \{1,\dots, n\}$, that is for any $x \in M$ the vectors $((du_i)^{\#_\omega}, J(du_i)^{\#_\omega})_x, i=1,\dots, n$ are proportional to a unitary basis of $(T_xM, \omega_x, J_x)$. Assume that $\nabla$ is torsion-free (thus $M$ Kaehler). Then ${\rm ker}\nabla \mathcal{E}_i=(\C\cdot(du_i\circ J)^{\#_\omega})^\perp$, where $\perp$ here refers to orthogonality wrt $\omega (\cdot, J\cdot)$. Furthermore $w_i=\nabla_{(du_i\circ J)^{\#_\omega}} \mathcal{E}_i \in {\rm End}(\mathcal{L}_i)\simeq \Omega^0(M,\C), i=1, \dots, n$, are constant, thus define a set of $w_i \in \C$ which we will call the spectral numbers of $(\mathcal{A}, \mathcal{L})$.
\item Assume $(\mathcal{A}, \mathcal{L})$ is a semi-simple {\it rigid} standard Frobenius structure with $\nabla J=0$ and that $M$ is formal, that is all (higher) Massey products on $H^*(M, \mathbb{C})$ vanish and that $\mathcal{C}_i$ is generic, thus $H^1(\mathcal{C}_i, \C)=0$. Then for any $i \in K\subset \{1, \dots, k\}$ so that the eigenform $\alpha_i$ of $\Omega$ on $\mathcal{L}_i$ has a non-trivial cohomology class, that is $0\neq [\alpha_i] \in H^1(M, \mathbb{C})$ the corresponding form $\tilde \alpha_i=\nabla \mathcal{E}_i \in \Omega^1(\tilde M_i, \C)$ has a non-vanishing closed part $\tilde \alpha_i^c$ wrt the Hodge decomposition of $H^1(\tilde M_i, \mathbb{C})$. Assume there is a {\rm canonical} set $\gamma_j \in H_1(M, \mathbb{Z}), j \in \{1,\dots,r\}$ of generators of $H_1(M, \mathbb{Z})$ and write for each $i \in K$ $PD[\alpha_i]=\sum_{j=1}^r a_{ij}\gamma_j$. If $\tilde M_i\rightarrow M$ is non-trivial choose a lift the $\gamma_j$ to (in general non-closed) paths $\tilde \gamma_j$ in $\tilde M_i$. We define the evaluation $w_{ij}=[\tilde \alpha_i^c](a_{ij}\tilde \gamma_j), \ i \in K, j \in \{1, \dots, r\}$ as the spectral numbers of $(\mathcal{A}, \mathcal{L})$.
\item Assume $(\mathcal{A}, \mathcal{L})$ is a semi-simple standard Frobenius structure with $\nabla J=0$ so that $\nabla$ is torsion-free (thus $M$ Kaehler). Then $\nabla \mathcal{E}_i$ is closed for all $i \in \{1,\dots, k\}$. Assume $[\alpha_i]\neq 0$ for all $i=1,\dots, k$. We define $w_{ij}=[\nabla \mathcal{E}_i](a_{ij}\tilde \gamma_j), \ i \in \{1,\dots, k\}, j \in \{1, \dots, r\}$ as in (2.) with $\gamma_j \in H_1(M, \mathbb{Z}), j \in \{1,\dots,r\}$ generating $H_1(M, \mathbb{Z})$ and $PD[\alpha_i]=\sum_{j=1}^r a_{ij}\gamma_j$.
\item If for (1.) or (2.) of the above the common assumptions hold and assume in addition that rational multiples of $0 \neq [\alpha_i]\in H^1(M, \mathbb{Q})$ for all $i=1,\dots, n$ generate $H^1(M, \mathbb{Q})$ and with the above notations, $n=r$. Then the respective definitions of spectral numbers coincide for an appropriate set of generators $\gamma_j \in H_1(M, \mathbb{Z})/{\rm Tor}$. If for any other subset of (1.)-(3.) the common assumptions are satisfied, then the respective definitions of the spectrum coincide.
\end{enumerate}
\end{propdef}
\begin{proof}
{\it Proof of (1.)} With the assumptions of (1.), choose for each $x \in M$ a ngbhd $x_0\in U\subset M$ and a unitary frame $(e_1, \dots, e_{2n})$ of $TM|U$ that is proportional at any $y \in U$ to the dual $\{(du_j)^{\#_\omega}, J(du_j)^{\#_\omega}\}_{j=1}^{n}$ basis to $((du_i), du_i\circ J), i=1,\dots, n$. Then $(e_1, \dots, e_{2n})(y)$ determines at any point $y \in U$ and a neighbourhood $U_y\subset U$ of $y$ a normal Darboux coordinate system on $U_y$ so that $(\tilde e_1, \dots, \tilde e_{2n})_y$ is the associated coordinate frame on $U_y$ that obeys $\nabla_{\tilde e_i} \tilde e_j=0, \ j=1, \dots 2n$ at $y$. Then write for any such pair $(U_y,y),\ y\in U$ $(\alpha_i|U_y)(z)=du_i(z)=\sum_{j=1}^{2n}\hat \beta^y_{ij}(z)(\tilde e_j^{\#_\omega})_y(z),\ i=1, \dots, 2n,\ z \in U_y$ for $\hat \beta^y_{i,j} \in C^\infty(U_y)$ and $\hat \beta^y_{ij}(y)= 0, i\neq j$. Note that we will often drop the parameter superscript $y$ in $\hat \beta^y_{ij}$ below. Also, write $\alpha_i=\sum_{j=1}^{2n}\beta_{ij}e_j^{\#_\omega}$ on $U$. Note that the factor $|\alpha_i|^{-2}=\omega(\alpha_i^{*}, J\alpha_i^{*})^{-1} \in C^\infty (U_y)$ is absorbed in the $\beta_{ij}$ resp. $\tilde \beta_{ij}$ here by the definition of $\tilde e_j^{\#_\omega}$. While $\alpha_i$ and $\mathcal{E}_i$ are defined on $M$, the coefficients $\beta_{ij}$ are thus only smooth on $M\setminus \mathcal{C}_i$. 
We then infer from (\ref{bla2}) that for all $i \in \{1,\dots, n\}$ on $U$ by taking for any $y \in U$ the derivative in any Darboux coordinate system on $U_y$ that
\begin{equation}\label{eulerbla2}
\begin{split}
\nabla_{\cdot}\mathcal{E}_i\cdot \varphi &= \left(\sum_{j=1}^{n} du_i(\cdot)\beta_{ij}(e_j-if_j) + u_i d\hat \beta_{ij}(\cdot)(e_j-if_j)\right)\cdot \varphi\\
&= \left(\sum_{j=1}^{2n} du_i(\cdot)\beta_{ij}\alpha_i(e_j)+u_i d\hat \beta_{ij}(\cdot)\alpha_i(e_j)\right)\varphi.
\end{split}
\end{equation}
The second sum, evaluated at $(du_i)^{\#_\omega}$ is equal to
\[
\sum_{j=1}^{2n} u_i ((du_i)^{\#_\omega}.du_i(e_{j+n}))du_i(e_j).
\]
From the definition of $\nabla$ as the Levi-Civita connection of $(M, J)$ and $(du_i)^{\#_\omega}(du_j\circ J)=\delta_{ij}$ we see that $\nabla_{(du_i\circ J)^{\#_\omega}}.du_i=0$ for $i\neq k$, so the above term is $0$ unless $i=k$ and thus ${\rm ker}\nabla \mathcal{E}_i=(\C\cdot(du_i\circ J)^{\#_\omega})^\perp$.
Note that if the symplectic connection $\nabla$ is torsion-free, the closedness of the form $\nabla_{\cdot}\mathcal{E}_i \in \Omega^1(M, {\rm End}(\mathcal{L}_i))$ follows in general since in that case and relative to $(e_1, \dots, e_{2n})$, the exterior derivative is a linear combination of $\nabla_{e_i}$ and $\nabla(\nabla \mathcal{E})=0$. To prove the second assertion in (1.), we have to prove that $w_i=\nabla_{(du_i\circ J)^{\#_\omega}} \mathcal{E}_i \in {\rm End}(\mathcal{L}_i)$ are constant on $M$. Set on local open sets $Y_i=\nabla_{(du_i\circ J)^{\#_\omega}}\mathcal{E}_i$, write $\nabla_{(du_i\circ J)^{\#_\omega}}\mathcal{E}_i\star \varphi=\beta_i \varphi$ for $\varphi \in \Gamma(\mathcal{L}_i)$ for some $\beta_i \in C^{\infty}(U)$ and consider for any $X \in \Gamma(U)$ 
\[
(X.\beta_i) \varphi =\nabla_{X}(Y_i\star \varphi) -Y_i \star \nabla_{X}\varphi =(\nabla_{X}Y_i)\star \varphi
\]
but since $(\nabla_{X}Y_i)=\nabla_X\nabla_{J(du_i)^{\#_\omega}}\mathcal{E}_i=0$, we find that $\beta_i$ is locally constant and thus constant on $M$. We give a second proof of the constancy of the spectrum $w_i$ using the defining second equation for the Euler vector field in (5.) of Definition \ref{frobenius} for the case $Y=J(du_i)^{\#_\omega}=\frac{\partial}{\partial u_i}$ and $\nabla$ torsion-free. Note that $\nabla_{Y}\mathcal{E}_i$ acts on $\varphi \in \Gamma(\mathcal{L}_i)$ by $\cdot$ and by $\star$ and by (\ref{eulerbla2}) both actions differ locally by elements of $\C$ depending on $Y$. So we consider
\[
\mathcal{L}_{\mathcal{E}_i}(Y\star \varphi)-[\mathcal{E}_i,Y]\star \varphi- Y\star \mathcal{L}_{\mathcal{E}_i}\varphi=\nabla_{Y}\mathcal{E}_i\star \varphi +b_i\varphi,
\]
where for a semisimple standard Frobenius structure $b_i\in \C$ as follows from the evaluation of $[\Omega_i(e_l), \Omega_i^t(e_k)]=ib\delta_{lk}$. Since $Y\star \varphi=(du_i\circ J)^{\#_\omega}\star \varphi$ multiplies $\varphi \in \Gamma(\mathcal{L}_i)$ by a function being invariant under the flow of $\mathcal{E}_i$ and thus commutes with $\mathcal{L}_{\mathcal{E}_i}$, we infer from the latter formula that
\[
-[\mathcal{E}_i,Y]\star \varphi=\nabla_{Y}\mathcal{E}_i\star \varphi+b_i\varphi.
\]
Since the flow of $\mathcal{E}_i$ preserves $\omega$, we have 
\[
\mathcal{L}_{\frac{\partial}{\partial u_i}}(u_i (du_i)^{\#_\omega})=(\mathcal{L}_{\frac{\partial}{\partial u_i}}(u_idu_i))^{\#_\omega}=(di_{\frac{\partial}{\partial u_i}}(u_idu_i))^{\#_\omega}
\]
and again using $(du_i\circ J)^{\#_\omega}\cdot\varphi= -i(du_i)^{\#_\omega}\cdot\varphi$ we see that $\mathcal{L}_{{\frac{\partial}{\partial u_j}}}$ of the latter expression is $0$ for all $i,j$ since $\mathcal{L}_{{\frac{\partial}{\partial u_j}}}du_i=0$ for all $j \in \{1,\dots,2n\},\ i \in \{1,\dots, n\}$. Note that in the situation of (1.) we have the following explicit formula for $\nabla_{\frac{\partial}{\partial u_k}}\mathcal{E}_k$, evaluated on $\varphi \in \Gamma(\mathcal{L}_k)$:
\[
w_k=(\nabla_{\frac{\partial}{\partial u_k}}\mathcal{E}_k)\cdot \varphi = \beta_{kk}\left(\beta_{kk}+u_k d\tilde \beta_{kk}(\frac{\partial}{\partial u_k})\right)\varphi
\]
where the last term is multiplication of $\varphi$ by an element of $\C$.\\
{\it Proof of (2.)} Consider now the assumptions of (2.), that is $0\neq [\alpha_i] \in H^1(M, \mathbb{C})$ for $i \in K\subset \{1, \dots, k\}$ and $M$ is formal. We examine the conditions under which $\tilde \alpha_i$, as given by (\ref{eulerbla2}), is closed. Note that if we would have $\beta_{ij}=\hat \beta_{ij}$ for all $i,j$ in the last line of (\ref{eulerbla2}), the sum of the terms involving no derivative of $\alpha_i(e_j)$ in the exterior derivative of $\tilde \alpha_i$ would give zero by $d^2=0$. Also, we have to analyze the remaining terms in $d(\tilde \alpha_i)$.\\
Thus, let $i \in K$ and $\alpha_i|U=\sum_{j=1}^{2n}\beta_{ji}e_j^{\#_\omega}$ for a special local unitary frame $(e_1,\dots,e_{2n})$ of $TM$ over $U\subset M$ that we define now. Let for any $x \in M$ a nghbhd $x \in U\subset M$. Define now a unitary frame $(e_1,\dots,e_{2n})$ in $TU$ s.t. ${\rm span}((e_2,\dots,e_{n},e_{n+2},\dots,e_{2n}))={\rm ker}(\alpha_i)^\perp\subset TU$ as follows. By parallel transport along geodesic rays we get an orthogonal frame $(e_2,\dots,e_{n},e_{n+2},\dots,e_{2n}))$ in ${\rm ker}\ (\alpha_i)_x)\subset T_xU$ by the projected connection $\hat \nabla:={\rm pr}_{{\rm ker}(\alpha)^\perp}(\nabla)$ so that $(e_1=\frac{du_i^{\#_\omega}}{||du_i^{\#_\omega}||},e_2,\dots,e_n, Je_1,\dots, e_{2n})$ gives a unitary frame on $U$. Conclude that thus on $U$ $(e_1,\dots, e_n, Je_{1}, \dots,Je_{n})$ gives a unitary frame with the property that ${\rm ker}(\alpha_i)^\perp=\C \cdot (e_1,Je_1)$ and ${\rm ker}(\alpha_i)={\rm span}_{\C}(e_2,\dots, e_n, Je_{2}, \dots,Je_{n})$ on $U$. \\
Assume first that $M$ is compact or compact with boundary. Let $\mathcal{U}=\{U_l\}_{l=1}^m$ be an open covering of $M$ so that for each $y \in U_l$ for appropriate $U_y\subset U_l$ we choose as in the proof of (1.) locally normal Darboux coordinates corresponding to the special frames $(e^l_1,\dots,e^l_{2n})$ on each $U_l$ constructed in the previous paragraph. Let $1 =\sum_{l=1}^m\rho_l$ with $\overline{{\rm supp}(\rho_l)}\subset U_l$ be a decomposition of unity subordinated to $\mathcal{U}$, furthermore $\alpha_i|U_l=\sum_{j=1}^{2n}\beta^l_{ij}(e^l_j)^{\#_\omega}$ for $e^l_j$ the above local symplectic frames over the $U_l$ (the coefficients of $\alpha_i$ corresponding to the Darboux coordinates on the open sets $U_y\subset U_l$ will again be denoted as $\hat \beta^{k,y}_{ij}$ resp. with superfix $y$ suppressed). Then on $M_c:=M\setminus \mathcal{C}_i$ we write
\begin{equation}\label{sum}
\alpha_i=\sum_{k=1}^m\sum_{j=1}^{2n}\beta^{k}_{ij}\rho_k(e^k_j)^{\#_\omega}=\sum_{k=1}^m\sum_{j=1}^{2n}\tilde \beta^{k}_{ij}\hat e^k_{j},
\end{equation}
where we have defined $\tilde \beta^{k}_{ji}=\beta^k_{ji}\rho^{1/2},\ k\in \{1,\dots,m\}$ and $\hat e^k_{j}=\rho_k^{1/2}(e^k_j)^{\#_\omega}, \ k\in \{1,\dots,m\},\ j=\{1,\dots,2n\}$. Noting that $\mathcal{C}_i$ is generically $0$-dimensional, thus $H^1(\mathcal{C}_i, \C)=0$, we will in the following considerations write $M$ instead of $M_c$. We then see with Theorem 4.1 of Deligne et al. (\cite{deligne2}) that there exists a decomposition $\Omega^1(M, \C)=N\oplus C$ where $C$ is the set of closed elements of $\Omega^1(M, \C)$ and $N$ is an appropriate complement of $C$ in $\Omega^1(M, \C)$ (we can take the complement given by the Hodge decomposition of $\Omega^*(M, \C)$), so that if $\alpha_i \in \Omega^1(M, \C)$ is closed and in the ideal generated by $N$, then $\alpha_i$ is exact. Assuming $[\alpha_i]\neq 0$, there is thus at least one index-pair $(j,k)\in (\{1,\dots, 2n\}, \{1,\dots, m\})$ so that in the sum $\alpha_i=\sum_{j=1}^{2n}\sum_{k=1}^{m}\tilde \beta^{k}_{ij}\hat e^k_{j}$ the space spanned (over $\C$) by $\{\hat e^k_{j}\}_{j,k}$ has nontrivial (orthogonal) projection to $\Lambda^*(C)$. Let $R\subset (\{1,\dots, 2n\}, \{1,\dots, m\})$ be the subset of those indices $(j,k) \in (\{1,\dots, 2n\}, \{1,\dots, m\})$ so that $\{\hat e^k_{j}\}_{(j,k) \in R}$ has this property, denote $\tilde e^k_{j}={\rm pr}_C^\perp(\hat e^k_{j}), (j,k) \in R$ where ${\rm pr}_C^\perp$ is the orthogonal projection onto $C$ along $N$. Restricting now the summation in both summands of (\ref{eulerbla2}) to $(j,k) \in R$ and setting 
\[
\hat \alpha_i^c\varphi= \left(\sum_{j,k\in R} du_i(\cdot)\beta^k_{ij}\rho_k\alpha_i(e_j) +  u_i d\hat \beta^k_{ij}(\cdot)\rho_k\alpha_i(e_j)\right)\varphi,
\]
we now claim that $d \hat \alpha_i^c=0$. We first claim that ${\rm ker}(d(\rho^{1/2}\beta_{ij}))^\perp={\rm ker}(d(\rho^{1/2}\hat \beta^k_{ij}))^\perp$ for $(j,k) \in R$. To see this fix $y \in U_y\subset U_k, k \in {\rm pr}_2(R)$ and consider the forms
\[
\gamma_i=\sum_{j \in {\rm pr_1}(R)}\tilde \beta^{k}_{ij}\hat e^k_{j}, \quad \tilde \gamma_i=\sum_{j \in {\rm pr_1}(R)}\hat \beta^{k,y}_{ij}(\tilde e^{y,k}_j)^{\#_\omega}
\]
Note that if we can prove that $\nabla$ preserves the subspaces in $\Gamma(T^*U_y)$ spanned over $C^\infty(U_y, \C^*)$ by $(\hat e^k_{1}, \hat e^k_{1}\circ J)$ resp. $(\tilde e^{y,k}_1)^{\#_\omega}, (\tilde e^{y,k}_1)^{\#_\omega}\circ J)$, then it follows from $\nabla J=0$, $\nabla$ thus preserving the orthogonal complement of $(\hat e^k_{1}, \hat e^k_{1}\circ J)$ resp. $(\tilde e^{y,k}_1)^{\#_\omega}, (\tilde e^{y,k}_1)^{\#_\omega}\circ J)$ on $U_y$ and from $\gamma_i(y)=\tilde \gamma_i(y)$, that $\gamma_i=\tilde \gamma_i^n$ on $U_y$. From this in turn the claim follows since $d\hat e^k_{j}=d(\tilde e^{y,k}_j)^{\#_\omega}=0$ on $U_y$ by the choice of our $(j,k) \in R$ (in the case of $\hat e^k_{j}$) resp. the choice of the normal Darboux coordinate system (in the case of $(\tilde e^{y,k}_j)^{\#_\omega}$). Now $\nabla (\tilde e^{y,k}_1)^{\#_\omega}=0$ and $\nabla (\tilde e^{y,k}_1)^{\#_\omega}\circ J=0$ in $y$ follows again from the property of normal Darboux coordinates in $U_y$, thus it remains to prove that $\nabla$ preserves $C^\infty(U_y, \C)\cdot(\hat e^k_{1},\hat e^k_{1}\circ J)$ in $T^*U_y$. But that follows from the fact that by our choices, $\alpha_i|U = C^\infty(U, \C)\cdot(\hat e^k_{1}+i \hat e^k_{1}\circ J)$ and $\nabla\alpha_i=0$ by the proof or Proposition \ref{lagrangian}.\\
Thus by the above partial summation over $(k,j) \in R$ wrt the partition of unity $(\mathcal{U}, \{\rho_i\}_{i=1}^m)$ in (\ref{eulerbla2}) we arrive at $\hat \alpha^c_i \in \Omega^1(\tilde M_i, {\rm End}(\mathcal{L}_i))\simeq \Omega^1(\tilde M_i, \C)$ and it is straightforward to show now that $\hat \alpha_i^c$ is in fact closed. Since $\hat \alpha_i^c$ gives a nontrivial direct summand of $\tilde \alpha_i$ in $C$ wrt the Hodge decomposition of $\Omega^1(M, \C)$, we see that $\tilde \alpha_i^c$, the orthogonal projection of $\tilde \alpha$ onto $C$ is nontrivial (it actually coincides with $\hat \alpha_i^c$). Consider now the (assumed) canonical basis $\gamma_j \in H_1(M, \mathbb{Z}), j=1,\dots, r$ write $PD[\alpha_i]=\sum_{j=1}^r a_{ij}\gamma_j, i \in K$, choose arbitrary lifts $\tilde \gamma_i$ to $\tilde M_i$. Define $w_{ij}=[\tilde \alpha_i^c](a_{ij}\tilde \gamma_j), \ i \in K, j \in \{1, \dots, r\}$ as the spectral numbers of $(\mathcal{A}, \mathcal{L})$. Note that the evaluation of $\tilde \alpha^c$ on the $\tilde \gamma_i$ with given basepoint is well-defined by the closedness of $\tilde \alpha^c$. Furthermore note that since $(\mathcal{A}, \mathcal{L})$ is by assumption rigid, there exist integral cohomology classes $b_i \in H^1(M, \Z)$ and $c_i \in \C$ so that $[\alpha_i]= c_i\cdot b_i \in H^1(M, \C)$. Then by Farber (\cite{farberb}, proof of Theorem 2.4), the $w_{ij}$ do not depend on the choice of base point of the lift $\tilde \gamma_i$ of $\gamma_i$ to $\tilde M_i$.\\
Let now the assumptions in (1.) and (2.) be simultanously satisfied while the canonical set of generators $\gamma_i, i=1,\dots, n$ of $H_1(M, \mathbb{Z})/{\rm Tor}$ being given by rational multiples of $PD[\alpha_i] \in H_1(M, \mathbb{Q})$. Lifting the $\gamma_i$ to paths $\tilde \gamma_i: [0,1]\rightarrow \tilde M_i$ we write again $PD[\alpha_i]=\sum_{j=1}^r a_{ij}\gamma_j, i \in K, a_{ij} \in \mathbb{Q}$ and define $w_{ij}=[\tilde \alpha_i^c](a_{ij}\tilde \gamma_j), \ i \in \{1,\dots,n\}, j \in \{1, \dots, r\}$. Note that we have $a_{ij}=0, i\neq j$, thus $\alpha_j([\hat \gamma_j])=1$ with $[\hat \gamma_j]=a_{jj}[\tilde \gamma_j]$ and arclength-parameterize $\tilde \gamma_j:[0,c]\rightarrow \tilde M_j$ so that wrt to an appropriate parallel symplectic (unitary) frame along $\tilde \gamma_j$ we have ${\rm pr}^{\perp}_{\C\dot (du_j\circ J)^{\#_\omega}}\circ \tilde \gamma_j'(t)=(d(t),0,\dots,0),\ d:[0,c]\rightarrow  \R$ where ${\rm pr}^{\perp}_{\C\dot (du_j\circ J)^{\#_\omega}}$ is the orthogonal projection wrt (the lifted) $(\omega, J)$ on $\C\dot (du_j\circ J)^{\#_\omega}\subset T\tilde M_j$. Then writing $(\alpha_j)(\tilde \gamma_j'(t))=(f(t)d(t), 0, \dots,0)$ for a function $f:[0,c]\rightarrow \R$  we have by the constancy of $w_{jj}=\tilde \alpha_j^c((du_i\circ J)^{\#_\omega})$ on $M$
\[
\begin{split}
\tilde \alpha_j^c([\hat \gamma_j])=& a_{jj}\int_{[0,c]}\tilde \alpha_j^c(\tilde \gamma'_j(t))dt= a_{jj}\int_{[0,c]}\tilde \alpha_j^c((du_j\circ J)^{\#_\omega})f(t)d(t)dt\\
=& w_{jj}\int_{[0,c]}a_{jj}f(t)d(t)dt=w_{jj}
\end{split}
\]
where we used the orthogonality of the $\alpha_i$. It is finally easy to see that if the assumptions in (1.) and (2.) hold, then $\tilde \alpha_i^c=\hat \alpha_i^c=\tilde \alpha_i$ which proves our assertion.
\end{proof}
Before we give (in \cite{kleinham}) an alternative algebraic criterion for extracting spectrality information from $\nabla \mathcal{E} \in \Omega^1(M, {\rm End}(\mathcal{L})\otimes \mathcal{L}^*)\simeq \Omega^1(M, \C\otimes \mathcal{L}^*)$ as in Proposition \ref{spectrum} (2.) above, we discuss how to interpret the so-called structure connection (or {\it Dubrovin connection}), which is for a parameter $z\in \C$ informally written as
\begin{equation}\label{dubrovin1}
\tilde \nabla_X\varphi_i =(\nabla_X +z\Omega_i(X))\varphi_i, \ \varphi_i \in \Gamma(\mathcal{L}_i), X \in \Gamma(TM),
\end{equation}
in the language of Section \ref{spinorsconn}. Here, $\Omega \in \Omega^1(M,{\rm End}(\mathcal{L})\otimes \mathcal{L}^*)$ is as in Definition \ref{frobenius}, $\nabla$ is the connection on $\Gamma(\mathcal{L})$ being induced by a fixed symplectic connection on $M$ and we assume that a compatible almost complex structure $J$ is chosen so that $\nabla J=0$ and $(\Omega, \mathcal{L}, \nabla)$ is standard and semisimple wrt the decomposition $\mathcal{L}=\bigoplus_{i=1}^k\mathcal{L}_i$. For a given metaplectic structure $\pi_P:P\rightarrow M$, consider the fibrewise direct product of $P$ with the pull back bundle $\hat \pi_M:\pi_P^*(TM)\rightarrow P$, considered as a bundle $\pi_P\circ \hat \pi_M:\pi_P^*(TM)\rightarrow M$ over $M$, twisted by the right action of $Mp(2n, \R)$, that is we set
\begin{equation}\label{tangent}
\begin{split}
P_G&=\pi_P^*(TM) \times_{Mp(2n, \R)} P\\
&=\{((y,q), x), (p, x)): x \in M, y \in \R^{2n}, p,q \in P, \pi_P(p)=\pi_P(q)=x\}/Mp(2n, \R),
\end{split}
\end{equation}
where we factor through the obvious 'diagonal' right action $\left(\hat g,((y,q), x), (p, x)\right) \mapsto \left((y,q.\hat g), x), (p.\hat g, x)\right),\ \hat g \in Mp(2n, \R)$. Consider the right action of $G=H_n\times_\rho Mp(2n, \R)$ on $P_G$ given for $h=(h_1, h_2) \in H_n$ and $\hat g \in Mp(2n, \R)$ so that $\rho(\hat g)=g$ by
\begin{equation}\label{tangentaction}
\tilde \mu: G\times P_G\rightarrow P_G, \ \tilde \mu\left((h, \hat g),(((y,q),p),x)\right)=\left(((\rho(g)^{-1}(y)+h, q.\hat g), p),x \right).
\end{equation}
We claim (proof below) that $\hat \mu$ defines a transitive right $G$-action on $P_G$ that induces the structure of a principle $G$-bundle on $P_G$ and that furthermore $P_G$ is isomorphic to the balanced product $\hat P_G=P\times_{Mp(2n, \R), {\rm Ad}} G$ which is the $G=H_n\times_\rho Mp(2n, \R)$-principal bundle (compare (\ref{reduction})) induced as a balanced product by the action of the (inverse) adjoint map  
\[
{\rm Ad}: Mp(2n, \R)\rightarrow {\rm End}(G),\ (g_1, (h, g_0))\mapsto (h, {\rm Ad}(g_1^{-1})(g_0)),\ g_0,g_1 \in Mp(2n, \R),\ h \in \R^{2n},
\]
on the second factor in $P\times G$ (while the principal fibre action is the usual $G$ action). We will denote by $\phi_{Ad}:P \rightarrow \hat P_G$ the corresponding extension homomorphism. Let now $L\subset \R^{2n}$ be any real Lagrangian subspace, that is $\omega_0|L=0$ and consider the associated maximal parabolic subgroup of $Sp(2n, \R)$ and its preimage under $\rho$:
\[
\mathfrak{P}_L=\{S \in Sp(2n, \R): SL=L\},\ \hat {\mathfrak{P}}_L=\rho^{-1}(\mathfrak{P}_L).
\]
Assume there exists a reduction of a given $\hat U(n)$-reduction $P^J$ of $P$ to $\hat U_L(n):=\hat U(n)\cap  \hat {\mathfrak{P}}_L$ which we call $P_L^J$. Consider then the extension of $P_L^J$ induced by the adjoint ${\rm Ad}: \hat U(n)\rightarrow {\rm End}(G)$ resp. its restriction to $\hat U_L(n)$, given by $P^J_{L,G}:=P^J_L\times_{\hat U_L(n), {\rm Ad}} G_L$, where $G_L$ is given by the subgroup
\begin{equation}\label{glemb}
G_L= H_n\times_\rho \hat U_L(n)\subset H_n\times_\rho Mp(2n, \R)=G,
\end{equation}
and we denote the corresponding extension map as $\phi_{{Ad}}:P_L^J\rightarrow P^J_{L,G}$. $P^J_{L,G}$ is a $G_L$-principal bundle and since $i^J:P^J_{L,G}\rightarrow P$ is an inclusion, we have an equivalence of $G$-principal fibre bundles
\begin{equation}\label{reductionbla}
\hat P_G\simeq P^J_{L,G}\times_{G_L, ({\rm Ad},id)} G,\ G_L\subset G,
\end{equation}
where $({\rm Ad},id):G_L\rightarrow G$ acts wrt the product structure on $G_L$ resp. $G$, where $id:L\hookrightarrow H_n$ is the identity. We denote the corresponding extension map by $\phi_{{\rm Ad},i}: P^J_{L,G}\rightarrow \hat P_G$. Thus $\hat P_G$ is the extension of $P^J_{L,G}$ from $G_L$ to $G$ given by $(Ad, i)$ and thus also $P^J_{L,G}\times_{G_L, ({\rm Ad},i)} G\simeq P_G$ as $G$-principal bundles. Let now $G^0_L=\{0\}\times_\rho \hat U_L(n)\subset G$ and $s:\hat P_G \rightarrow G/G^0_L$ be any equivariant smooth map inducing a section of the fibration
\begin{equation}\label{reductionbla2}
\hat P_{G/G^0_L}=\hat P_G\times_{{\rm Ad},G} G/G^0_L\rightarrow M.
\end{equation}
Then as in the discussion over Theorem \ref{genclass}, we can associate to any section of $\hat P_{G/G^0_L}$ a $G^0_L$-reduction of $\hat P_G$ and the isomorphism classes of these reductions are in bijective correspondence to the homotopy classes of sections $s:\hat P_G \rightarrow G/G^0_L$. We denote a representative of such a $G^0_L$-reduction of $\hat P_G$ associated to $s$ by $\hat P_{L,s}$. Note that the map $s_0:\hat P_G \rightarrow G/G^0_L$ given by $s_0(p)= 1\cdot G^0_L,  \ p \in P_L^J$ and equivariantly extended to $\hat P_G$ corresponds to the standard reduction $P_L^J$ of $\hat P_G$.\\
Note finally that $P^J_L$, $P^J$ and $P$ are reductions of $\hat P_G$ to the subgroups $\hat U_L(n), \hat U(n), Mp(2n, \R) \subset G$, respectively under the homomorphism ${\rm Ad}: Mp(2n, \R)\rightarrow {\rm End}(G)$ resp. its various restrictions. For the following, let $\mathfrak{g}=\mathfrak{sp}(2n,\mathbb{R})\oplus \mathfrak{h}_n$ and $\mathfrak{sp}(2n,\mathbb{R})=\mathfrak{u}(2n,\mathbb{R})\oplus\mathfrak{p}$ the Cartan decomposition with associated projections ${\rm pr}_{\mathfrak{u}}:\mathfrak{g}(2n,\mathbb{R})\rightarrow \mathfrak{u}(2n,\mathbb{R})$, ${\rm pr}_{\mathfrak{h}_n}:\mathfrak{g}\rightarrow \mathfrak{h}_n$. Let $\mathfrak{p}_L\subset \mathfrak{sp}(2n,\mathbb{R})$ be the Lie algebra of $\mathfrak{P}_L\subset Sp(2n, \R)$. Let furthermore ${\rm pr}_L:\mathfrak{h}_n\rightarrow L\times \{0\}$ be the projection onto the maximally abelian subspace in $\mathfrak{h}_n$ given by $L$, let $\mathfrak{l}\subset \mathfrak{h}_n$ be the commutative sub-Lie algebra defined its image and $\mathfrak{l}_{\C}$ be its complexification. Recall that if $P$ is a $G$-principal bundle then the {\it tensorial} $1$-forms of type $Ad$ on $\pi_P:P\rightarrow M$ with values in in the Lie algebra of $G$, $\mathfrak{g}$, are those $1$-forms $w:TP \rightarrow\mathfrak{g}$ which vanish on ${\rm ker}(d\pi_P)$ and such that $(R_g)^*w=Ad(g^{-1})w$ for all $g \in G$, where $R_g,\ g \in G$ denotes the right action of $G$ on $P$. Note that there is an isomorphism between the vector space of tensorial $1$-forms on $P$ and the vector space of $1$-forms on $M$ with values in the associated bundle $\underline {\mathfrak{g}}=P \times_{Ad} \mathfrak{g}$, written $\Omega^1(M,\underline {\mathfrak{g}})$. Returning to the above, note that $\hat P_G$, given a connection $\hat Z:TP\rightarrow \mathfrak{sp}(2n,\mathbb{R})$ and a $\hat U_L(n)$-reduction $\hat Z_L^J$ of $\hat Z$ to $P_L^J$, carries a tautological connection $\hat Z_{G,L}^{J,\omega}: T\hat P^J_{G, L} \rightarrow \mathfrak{g}_{L,\C}:=\mathfrak{u}(2n,\mathbb{R})\cap \mathfrak{p}_L \oplus \mathfrak{h}_{n,\C}$ which consists of the sum of the canonical extension $\hat Z^J_{G, L}$ of $\hat Z_L^J$ to $P^J_{G, L}$ (see below) and the tensorial $1$-form on $P^J_{G, L}$ which is given by 
\begin{equation}\label{fouriermu}
\begin{split}
w_0:&TM\rightarrow P^J_{G,L} \times_{Ad}\mathfrak{g}_{L,\C},\quad p \in P^J_{G,L}, (g,h=(h_1,\dots, h_{2n})) \in G_L,\\
w_0(X)&=((p, (g,h)), (0,\sum_{l=1}^{n}\left({\rm Ad}(g^{-1})(h)_l\Phi_p(X)_l+i {\rm Ad}(g^{-1})(h)_{l+n}\Phi_p(X)_{l+n}\right)a_l),
\end{split}
\end{equation}
where $\Phi_p:TM \rightarrow \R^{2n}$ is the isomorphism determined by $p \in  P^J_{G,L}$, $(a_l)_{l=1}^{n}$ is the standard basis in $\R^{n}\times \{0\}\subset \R^{2n}$ and $\mathfrak{g}_{L,\C}$ refers to the (complexification of) Lie algebra of $G_L=H_n\times_\rho \hat U_L(n)$ as described above Lemma \ref{fock}. Note that by definition, $\hat Z_{G,L}^{J,\omega}$ takes values in $\mathfrak{g}^0_{L,\C}:=\mathfrak{u}(2n,\mathbb{R})\cap \mathfrak{p}_L \oplus \mathfrak{l}_{\C}\subset \mathfrak{g}_{L,\C}$. Note that for a given almost complex structure $J$ and for $P_G$ as in (\ref{tangent}), we denote by $pr_1:P_G\rightarrow TM\simeq T^*M$ the map $pr_1((y,q), x), (p, x))=((gy,q),x),\ x \in M,\ y \in \R^{2n},\ p,q \in P_L^J, \ q=p.g,\ g \in \hat U_L(n)$. Note that as above Theorem \ref{genclass}, ${\rm pr}_1$ factors to a map $\tilde {\rm pr}_1: \hat P_{G/G^0_L}\rightarrow T^*M$ when considering $\hat P_{G/G^0_L}$ as a quotient $r:P_G\rightarrow \hat P_{G/G^0_L}$, that is ${\rm pr}_1=\tilde {\rm pr}_1\circ r$. For a given equivariant map $s:\hat P_G \rightarrow G/G^0_L$ and the map $s_0:\hat P_G \rightarrow G/G^0_L$ corresponding to the standard reduction $P_L^J$ of $\hat P_G$ as above consider the smooth map $\hat g:\hat P_G \rightarrow G/G^0_L$ satisfying $s(p)=\hat g(p)s_0(p)$. We then have
\begin{prop}\label{higgs}
Assume $\nabla$ is a given symplectic connection, $\hat Z:TP\rightarrow \mathfrak{sp}(2n,\mathbb{R})$ its connection $1$-form, $J$ is a compatible almost complex structure so that $\nabla J=0$ and $\mathfrak{P}_L\subset Sp(2n, \R)$ a maximal parabolic subgroup as above. Assume $P_L^J$ is a reduction of $P^J$ to $\hat {\mathfrak{P}}_L$, thus $P^J_L\subset P^J\subset P^J_{G,L}\subset P_G$ is the  chain of inclusions of principal fibre bundles wrt the chain of inclusions of structure groups $\hat U_L(n)\subset \hat U(n) \subset G_L\subset G$ as described above. Consider a reduction of the given symplectic connection $\hat Z:TP\rightarrow \mathfrak{sp}(2n,\mathbb{R})$ to $P^J$, $\hat Z^J$ resp. its further $\hat U_L(n)$-reduction $\hat Z_L^J$ to $P^J_L$ and the extension of $\hat Z_L^J$ to $P^J_{L,G}$, called $\hat Z_{G,L}^{J}$. Denote the extension of the given symplectic connection $\hat Z:P\rightarrow \mathfrak{sp}(2n,\mathbb{R})$ to $\hat P_G$ by $\hat Z^0_G$. With the corresponding inclusions of Lie algebras $\mathfrak{u}(2n,\mathbb{R})\cap \mathfrak{p}_L\subset \mathfrak{g}_{L,\C}\subset \mathfrak{g}_{\C}$ we have the commuting diagram:
\begin{equation}\label{dia1}
\begin{diagram}
T P^J_L   &\rTo^{\phi_{{Ad}}} &T P^J_{G,L}   &\rTo^{\phi_{{\rm Ad},i}}  &T \hat P_G\\
\dTo^{\hat Z_L^J} &     &\dTo^{\hat Z^J_{G,L}} &      &\dTo_{\hat Z^0_{G}}\\
\mathfrak{u}(2n,\mathbb{R})\cap \mathfrak{p}_L  &\rTo_{i_*}   &\mathfrak{g}_{L,\C}  &\rTo_{i_*} &\mathfrak{g}_{\C}
\end{diagram}
\end{equation}
Consider now the 'tautological' connection $\hat Z_{G,L}^{J,\omega}=\hat Z_{G,L}^{J}+w_0:T\hat P^J_{G, L} \rightarrow \mathfrak{g}_{L,\C}$ described above and its extension $\hat Z^\omega_G$ to $\hat P_G$. Further, let $(\Omega, \mathcal{L}, \nabla)$ be a semisimple standard irreducible Frobenius structure corresponding to a section $s$ of $\hat P_{G/G^0_L}$, defining a $G^0_L\simeq\hat U(n)_L$-reduction $\hat P^J_{L,s}$ of $\hat P_G$ so that the $\hat U_L(n)$-reduction of $P$ given by the composition of $s:\hat P \rightarrow G/G^0_L$ with the canonical projection on the subquotient $\pi_{Mp}:G/G^0_L\rightarrow Mp(2n, \R)/\hat U_L(n)$, that is the section $\tilde s=\pi_{Mp}\circ s:M\rightarrow P\times_{Mp(2n,\R)} Mp(2n, \R))/\hat U_L(n)$ corresponds to the above pair $(P^J_L, J)$. Then if $P^J_{G_L,s}$ is the reduction of $\hat P_G$ corresponding to the composition of $s$ with the quotient map $G/G^0_L\rightarrow G/G_L$ we have $P^J_{G_L,s}\simeq P^J_{G,L}$ and considering $\hat Z^{J,\omega}_{G_L,s}:TP_{G_L,s}\rightarrow \mathfrak{g}_{L, \C}$ as the reduction of $\hat Z^\omega_{G}:T\hat P_{G}\rightarrow \mathfrak{g}_{\C}$, we have wrt this isomorphism $\hat Z^{J,\omega}_{G_L,s}= \hat Z^{J,\omega}_{G,L}$. With $\tilde Z_{L,s}^J:TP^J_{L,s}\rightarrow \mathfrak{u}(2n,\mathbb{R})\cap \mathfrak{p}_L$ further reducing $\hat Z^{J,\omega}_{L,G}$ to $P^J_{L,s}$ the following diagram commutes: 
\begin{equation}\label{dia2}
\begin{CD}
T P^J_{L,s}   @>\phi_{Ad}>> T P^J_{G_L,s}  @>{\phi_{{\rm Ad},i_L}}>>  T \hat P_G\\
@VV{{\hat Z_{L,s}^J}}V      @VV{{\hat Z^{J,\omega}_{G,L}}}V          @VV{{\hat Z^\omega_{G}}}V\\
\mathfrak{u}(2n,\mathbb{R})\cap \mathfrak{p}_L  @>{i_*}>>  \mathfrak{g}_{L,\C}  @>{i_*}>> \mathfrak{g}_{\C}
\end{CD}
\end{equation}
Assume now that $L=L_0=\R^n\times \{0\}$. Then $\hat Z^0_G$ reduces to connections $\hat Z^0_{L,s}:T P^J_{L,s} \rightarrow \mathfrak{u}(2n,\mathbb{R})\cap \mathfrak{p}_L$ on $P^J_{L,s}$ resp. $\hat Z^0_{G,L}:T P^J_{L,G}\simeq TP^J_{G_L,s} \rightarrow \mathfrak{g}_{L,\C}$. Consider $s$ as a map $s:M\rightarrow \hat P_{G/G^0_L}$, then if $\hat s=\tilde {\rm pr}_1\circ s:M\rightarrow T^*M$ satisfies $d\hat s=0$ we have 
$\hat Z^0_{L,s}=\hat g^*\hat Z^J_{L}$ when considering $P^J_{L,s}$ and $P^J_{L}$ as subsets of $\hat P_G$ and with $\hat g:\hat P_G \rightarrow G/G^0_L$ associated to $s, s_0:\hat P_G \rightarrow G/G^0_L$ as above. Furthermore there is a tensorial $1$-form $w_L:T P^J_{G,L} \rightarrow \mathfrak{g}^0_{L,\C}\subset \mathfrak{g}_{L,\C}$ of type $Ad$ (namely $w_L:= \hat Z_{G,L}^{J,\omega} -\hat Z^0_{G,L}$) so that with the above notations the connection $\hat \nabla$ on $\Gamma(\mathcal{E}_i)$ that is associated to the connection $1$-form (see the below remark) 
\begin{equation}\label{duconn}
\tilde Z:= i_*\hat Z^0_{L,s} + w_L\circ \phi_{Ad} : T P^J_{L,s}\rightarrow \mathfrak{g}^0_{L, \C}\subset \mathfrak{g}_{L,\C}
\end{equation}
is identical (as a map $\hat \nabla: \Gamma(\mathcal{L}) \rightarrow \Gamma(T^{*}M \otimes \mathcal{L})$) to $\tilde \nabla$ as defined in (\ref{dubrovin1}) (for $z=1$). Note that here, we represent $\mathfrak{g}_{L,\C}$ on $\mathcal{S}(\R^n)$ by the assignment $\kappa_{T_0}: \mathfrak{g}_{L,\C}\rightarrow  {\rm End}(\mathcal{S}(\R^n))$. Also, we identify the associated bundles $\mathcal{Q}$ on $P^J_{L,s}$ and $P^J_{G, L}$ by the usual identification.
\end{prop}
{\it Remark.} Note that we consider the spinorbundle $\mathcal{Q}$ associated to $P_{L,s}^J$ resp. $P^J_{G_L,s}$ by the representation 
\[
\hat \mu: G\rightarrow {\rm End}(\mathcal{S}(\R^n)),\quad  ((h,t,g),f)\ \mapsto\ \pi((h,t))L(g)f,
\]
compare (\ref{semirep}), restricted to $G^0_L$ resp. $G_L$. The Frobenius structure (semisimple, irreducible, standard) associated to $s:M\rightarrow \hat P_{G/G^0_L}$ is by Theorem \ref{genclass} then given by associating the line $\mathcal{A}^0_2$ to $P^J_{L,s}$, in particularly the connection $\tilde Z: TP^J_{L,s}\rightarrow \mathfrak{g}^0_{L, \C}$ in (\ref{duconn}) gives a connection (the 'first structure connection') on $\mathcal{L}=\mathcal{E}_M=P^J_{L,s}\times_{\hat \mu\circ i}\mathcal{L}_0,\ \mathcal{L}_0:=\C\cdot f_{0,iI} \subset \mathcal{S}(\R^n)$, with the notation of Proposition \ref{class} by the following procedure: we have $\hat \mu|Mp(2n, \R)=L$, as is obvious. On the other hand we define the covariant derivative associated to $\tilde Z: T P^J_{L,s}\rightarrow \mathfrak{g}^0_{L, \C}$ by the formula
\begin{equation}\label{repduconn}
\hat \nabla_{X} \varphi = [\overline s_U,du(X) + \kappa_{T_0}(\tilde Z \circ (\overline s_U)_{*}(X))u],\quad X\in \Gamma(TM),
\end{equation}
where $T_0=iI \in \mathfrak{h}$, $s_U:U\subset M\rightarrow P^J_{L,s}$ is a local section and $\kappa_{T_0}:\mathfrak{g}_{L,\C}\rightarrow {\rm End}(\mathcal{S}(\R^n))$ is as defined in Lemma \ref{fock} and $[s_U,u], u: U\subset M\rightarrow \mathcal{S}(\R^n)$ represents a local section $\varphi:U \rightarrow \mathcal{L}\subset \mathcal{Q}$. Note that by Lemma \ref{fock} $\kappa_{T_0}|\mathfrak{sp}(2n,\mathbb{R})=\Phi_T\circ L_*$ while $\kappa_{T_0}|\mathfrak{h}_n=\Phi_{T_0}\circ \hat \mu_*$. As we will see in the proof below $\kappa_{T_0}|\mathfrak{u}(2n,\mathbb{R})=L_*$, since $\Phi_{T_0}|\mathfrak{u}(2n,\mathbb{R})=id_{\mathfrak{u}(2n,\mathbb{R})}$, thus $\kappa_{T_0}\circ{\rm pr}_{\mathfrak{u}\cap \mathfrak{p}_L} \circ \hat Z^0_{L,s}=L_*\circ \hat Z^0_{L,s}$, as required by the above.
\begin{proof}
Consider elements of $P_G$ as representatives $\left(((y,q), x), (p, x)\right), x \in M, y \in \R^{2n}, p,q \in P, \pi_P(p)=\pi_P(q)=x$ as above while representatives of $\hat P_G$ as $(p,x), (h,g), p \in P, x \in M, \pi_P(p)=x, h \in \R^{2n}, g \in Mp(2n, \R)$. We claim there is a well-defined map
\begin{equation}\label{psi}
\Psi: P_G\rightarrow \hat P_G,\quad \Psi(((y,q), x), (p, x))= \left((p,x), (y, g(p,q))\right),
\end{equation}
where $g(p,q) \in Mp(2n, \R)$ is the unique element so that $p.g(p,q)=q$. Thus we claim that $\Psi$ is equivariant wrt to the respective $Mp(2n, \R)$-actions on the sets of representatives of $P_G$ resp. $\hat P_G$, thus $\Psi[((y,q), x), (p, x)]=[(p,x), (y, g(p,q))]$ and that the resulting factor map $\Psi_G: P_G\rightarrow \hat P_G$ is smooth and equivariant wrt to the respective $G$-actions on $P_G$ and $\hat P_G$. To see the first claim, let $g_1 \in Mp(2n, \R)$ and note that by definition
\[
\Psi(((y,q.g_1), x), (p.g_1, x))= \left((p.g_1,x), (y, \tilde g(p.g_1,q.g_1))\right)
\]
where $p.g_1.\tilde g(p.g_1,q.g_1)=q.g_1$. Since $p.g(p,q)=q$, we see that $\tilde g(p.g_1,q.g_1)=Ad(g_1^{-1})g(p,q)$ which shows the first assertion. The smoothness of $\Psi$ follows by considering the defining formula (\ref{psi}) relative to a local section $s:U \rightarrow P$ while the equivariance wrt to the right $G$-actions on $P_G$ and $\hat P_G$ is now obvious and left to the reader (note that $G$ acts on the second factor in $\hat P_G$ by the usual $G$-action, not by the adjoint).\\
By the discussion above Theorem \ref{genclass} the isomorphy classes of $\tilde G\subset \hat U(n)\subset G$-reductions of $\hat P_G$ are given by homotopy classes of sections $s: M\rightarrow \tilde P_{G/\tilde G}$ of
\[
\tilde P_{G/\tilde G}=\tilde P_G\times_{H_n\times_\rho Mp(2n, \R)}(H_n\times_\rho Mp(2n, \R))/{\tilde G})\rightarrow M,
\]
where here $\tilde P_G=P\times_{Mp(2n, \R), i} G$ and $i:Mp(2n, \R)\rightarrow G$ is the (standard) inclusion, we denote a representatve of such a reduction by $P_{\tilde G,s}$. The associated bundle $\tilde P_{G/\tilde G}$ remains the same when replacing $i$ by $Ad:\tilde G\rightarrow G$ (note that we have a given a canonical reduction of $\hat P_G$ to $i(\tilde G)\subset G$ if $P$ is reduced to $\tilde G$), thus also fixing the equivalence class of $P_{\tilde G,s}$ as a $\tilde G$-reduction of $\hat P^G$ wrt to the homomorphism $Ad:\tilde G\rightarrow G$. Let $s$ be the section of $\tilde P_{G/\tilde G}$ correponding to a fixed $\tilde G=\hat U_L(n)\subset \hat U(n)\subset Mp(2n,\R)$-reduction of $\hat P_G$ and a fixed semisimple irreducible standard Frobenius structure associated to this reduction as discussed in Theorem \ref{genclass}, then $P^J_{L,s}$ is the corresponding $\hat U_L(n)$-bundle. We can also replace $\tilde G\subset \hat U(n)\subset G$ by the embedding of $G_L\subset G$ as defined in (\ref{glemb}) and thus consider sections $\hat s:M\rightarrow P_{G/G_L}$. Then any section $s: M\rightarrow \tilde P_{G/\tilde G}$ with $\tilde G\subset \hat U(n)$ as associated to a semisimple standard irreducible Frobenius structure as above, fixes in a canonical way a homotopy class of sections $\hat s:M\rightarrow P_{G/G_L}$ (by projecting to the quotient) and the bundle $P^J_{G_L,s}$ will be the corresponding reduction of $\hat P_G$ to $G_L$. On the other hand, considering for $\hat U_L(n)\subset G$ resp. $G_L\subset G$ we also have the (standard) reductions $P^J_L$ resp. $P^J_{G,L}$ of $\hat P^J$ resp. of $\hat P_G$ as introduced above (\ref{glemb}), the existence of the former was assumed in this Proposition. By the definition of $G_L$ and by the assumption that $s:\hat P \rightarrow G/G^0_L$, projected down to $Mp(2n, \R)/U(n)_L\simeq G/G_L$, defines $P^J_L$, it follows that $P^J_{G_L,s}$ is naturally isomorphic to  $P^J_{G,L}$.\\
What remains to show is on one hand that the given symplectic connection $\hat Z:TP\rightarrow \mathfrak{sp}(2n,\mathbb{R})$ reduces to $P^J$ resp. to $\hat Z^J_L$ on $P^J_L$ and furthermore, that the extensions of $\hat Z^J_L$ to $P^J_{G,L}$ and $\hat P_G$ (cf. \ref{dia1}) reduce to $P^J_{G_L,s}$ and $P^J_{L,s}$ as in (\ref{dia2}). Analogously, we have to show that the extension of $\hat Z_{G,L}^{J,\omega}$ in (\ref{dia1}) to $\hat P_G$ reduces to $P^J_{G_L,s}$ and $P^J_{L,s}$ in (\ref{dia2}).\\
Note that an extension $\hat Z^0_G:TP_G\rightarrow \mathfrak{g}_{\C}$ of $\hat Z:TP\rightarrow \mathfrak{sp}(2n,\mathbb{R})$ always exists and is unique (by $R_G^*$-invariance). That a $\hat U(n)$-reduction $\hat Z^J$ of $\hat Z$ exists follows (as is well-known) from the fact that $Ad(\hat U(n))(\mathfrak{m})\subset \mathfrak{m}$ where $\mathfrak{m}\subset \mathfrak{sp}(2n,\mathbb{R})$, that is
\[
\mathfrak{sp}(2n,\mathbb{R})=\mathfrak{u}(2n, \R) \oplus \mathfrak{m}, \quad  \mathfrak{m}=\{X \in \mathfrak{gl}(2n, \R): XJ=-JX, X^t=X\}.
\]
Consider now the Iwasawa decomposition of $\mathfrak{sp}(2n,\mathbb{R})$, so $\mathfrak{sp}(2n,\mathbb{R})=\mathfrak{k}\oplus\mathfrak{a}\oplus \mathfrak{n}$, where $\mathfrak{k}=\mathfrak{u}(2n, \R)$ corresponds to the fixed point set of the Cartan involution on $\mathfrak{sp}(2n,\mathbb{R})$, $\mathfrak{a}$ is maximally abelian and $\mathfrak{n}$ is a nilpotent subalgebra. Then $\mathfrak{a}\oplus \mathfrak{n}$ is contained in a Borel subalgebra of $\mathfrak{sp}(2n,\mathbb{R})$ (cf. \cite{cap}, 3.2.8). Because of the latter, we have $\mathfrak{a}\oplus \mathfrak{n} \subset \mathfrak{p}_L$, on the other hand since $\mathfrak{sp}(2n,\mathbb{R})= \mathfrak{m}_1\oplus \mathfrak{p}_L$ for some $Ad(P_L)$-invariant Lie-subalgebra $\mathfrak{m}_1$, we can define $\mathfrak{m}_P=\mathfrak{u}(2n, \R)\cap \mathfrak{m}_1$ and have $\mathfrak{u}(2n, \R)=\mathfrak{u}(2n,\mathbb{R})\cap \mathfrak{p}_L \oplus \mathfrak{m}_P$, thus the desired $Ad(P_L\cap \hat U(n))$-invariant complement of $\mathfrak{u}(2n,\mathbb{R})\cap \mathfrak{p}_L$ in $\mathfrak{u}(2n, \R)$. This proves that $\hat Z^J$ further reduces to $\hat Z^J_L$, thus to $P^J_L$. Analogously, given the connections $\hat Z^J_{G_L,s}$ or $\hat Z^0_{G,L}$ on $P^J_{G_L,s}\simeq P^J_{G,L}$, these reduce to $\hat U(n)_L$-connections $\hat Z^J_{L,s}$ resp. $\hat Z^0_{L,s}$ on $P^J_{L,s}$ since $\mathfrak{u}(2n, \R)\cap \mathfrak{m}_1\oplus \mathfrak{h}_n$ is an $Ad(\hat U(n)\cap P_L)$-invariant complement of $\mathfrak{u}(2n,\mathbb{R})\cap \mathfrak{p}_L$ in $\mathfrak{g}$ (given $Ad(\hat U(n)\cap P_L)$ of course also preserves $\mathfrak{h}_n)$. Again analogously, $\hat Z^\omega_G$ on $\hat P_G$ reduces to $\hat Z^{J,\omega}_{G_L,s}$ with values in $\mathfrak{g}_{L,\C}$ since $\mathfrak{u}(2n, \R)\cap \mathfrak{m}_1 \oplus \mathfrak{m}$ is an $Ad(G_L)$-invariant complement to $\mathfrak{u}(2n, \R)\cap \mathfrak{p}_L\oplus \mathfrak{h}_n$ in $\mathfrak{g}_{\C}$.\\
It remains to show that $\hat Z_{G_L,s}^{J,\omega}$ and $\hat Z_{G,L}^{J,\omega}$ coincide wrt the isomorphism $P^J_{G_L,s}\simeq P^J_{G,L}$ and that $\hat w_L:= \hat Z_{G_L,s}^{J,\omega} -\hat Z^0_{G,L}=w_L$ indeed takes values in $\mathfrak{u}(2n,\mathbb{R})\cap \mathfrak{p}_L \oplus \mathfrak{l}$. Note that $\hat Z^{J,\omega}_{G_L,s}$ is {\it defined} as the reduction of $\hat Z^\omega_G$ on the right-most vertical arrow of (\ref{dia2}) to $P^J_{G_L,s}$ and $\hat Z^\omega_G$ is the extension of $\hat Z_{G,L}^{J,\omega}$ to $\hat P_G$. But since the sections $\hat s=\pi_{G/G_L}\circ s: \hat P \rightarrow G/G_L\simeq Mp(2n, \R)/\hat U(n)_L$ and $s_0:\hat P\rightarrow G/G_L$ defining $P^J_{G_L,s}$ and $P^J_{G,L}$ coincide (modulo isomorphy) by assumption and of course extension and subsequent reduction lead to the same connection (modulo the isomorphy), the claim follows. Note finally that $w_L$ takes values in $\mathfrak{u}(2n,\mathbb{R})\cap \mathfrak{p}_L\oplus \mathfrak{l}_{\C}$ since by definition of $P_L$ $\mathfrak{l}_{\C}$ is an $Ad(P_L)$-invariant complement of $\mathfrak{u}(2n,\mathbb{R})\cap \mathfrak{p}_L$ in $\mathfrak{u}(2n,\mathbb{R})\cap \mathfrak{p}_L \oplus \mathfrak{l}_{\C}$. Note further that for the arguments above, we can ignore whether a given extension is defined via $Ad$ or inclusion since a principal bundle homomorphism $\phi_{Ad}:P\rightarrow P\times_{H, Ad} G$ ($P$ a $H$ bundle, $H\subset G$ subgroup) corresponding to $Ad$ preserves given horizontal distributions $\mathcal{H}\subset TP, \mathcal{H}_G \subset T(P\times_{H, Ad} G)$ if and only if the homomorphism $i:P\rightarrow P\times_{H, Ad} G$ given by $i(p)=(p,(0,Id))$ preserves the same. But $P\times_{H, Ad} G$ is equivalent to $P\times_{H, id} G$ as a $G$-extension of $P$ by the above remarks.\\
Considering $s$ as a map $s:M\rightarrow \hat P_{G/G^0_L}$ we now show that if $\hat s=\tilde {\rm pr}_1\circ s:M\rightarrow T^*M$ satisfies $d\hat s=0$ we have $\hat Z^0_{L,s}=\hat g^*\hat Z^J_{L}$ when considering $P^J_{L,s}$ and $P^J_{L}$ as subsets of $\hat P_G$ and with $\hat g:\hat P_G \rightarrow G/G^0_L$ associated to $s, s_0:\hat P_G \rightarrow G/G^0_L$ as indicated above. Fix a local trivialization  $U\subset \hat P_G\simeq M\times G$ and consider the $\hat U_L(n)$-bundle $P^J_L$ as a subset $P^J_L \subset \hat P_G$ given by the inclusion $P^J_L\hookrightarrow  \hat P_G$ of (\ref{reductionbla}) (considering $P^J_L\subset P^J_{G,L}$). Note that the horizontal distributions $\mathcal{H}^0\subset T\hat P_G$ corresponding to $\hat Z^0_G$, the extension of $Z^J_L$ to $\hat P_G$ are given at any point $(h, g) \in G=H_n\times_\rho Mp(2n, \R), h \in H_n, g \in \hat U_L(n)$ in $U$ by $\mathcal{H}_{(g^{-1}h,g)}=(R_h)_*(\mathcal{H}_{(0,g)})$, where $(R_h)_*$ is the differential of the right translation $R_h:P^J_L\subset \hat P_G \rightarrow \hat P_G, \  (0,g)\mapsto (g^{-1}h,g)$, restricted to $TU\cap TP^J_L$ and $\mathcal{H}_{(0,g)}\subset T_{(0,g)}(P^J_L\cap U)$ is the horizontal subspace associated to $Z^J_L$ at $(0,g)\in P^J_L\cap U$. Considering the horizontal distribution $\mathcal{H}_s\subset P^J_{L,s}$ accociated to $Z^0_{L,s}$ as a family of subspaces $\mathcal{H}_s\subset T\hat P_G|P^J_{L,s}$ we thus see again by right-invariance of $\mathcal{H}_s$ that $\hat Z^0_{L,s}=\hat g^*\hat Z^J_{L}$ if and only if $\hat g_*:T\hat P_G\rightarrow T \hat P_G$, restricted to $\mathcal{H}\subset T(P^J_L\cap U)\subset \hat T(\hat P_G\cap U)$ satisfies if $(h,g) \in {\rm im}(g)$ and $(0,g)= \hat g^{-1}((g^{-1}h,g))$
\[
(R_h)_*((0,g), (0,X,a))=\hat g_*((0,g), (0, X, a)),  \ (X ,a) \in \mathcal{H}\subset T_xM\oplus \mathfrak{u}(2n,\mathbb{R})\cap \mathfrak{p}_L,\ X\neq 0,
\]
and since $(R_h)_*(0,X,a)=(a^{-1}h, X, a)$ by a direct calculation involving (\ref{semid}) while by equivariance of $s, s_0$ we have $\hat g_*((0,g), (a^{-1}h+g^{-1}(d_1\hat g)_1, X, a))$, where $(d_1\hat g)_1$ is the first coordinate of the partial differential of $\hat g|U\cap P^J_{L}:U\cap P^J_{L}\subset \hat P_G \rightarrow G/G^0_L$ into the direction of the first ($H_n$-)coordinate. Note that by our assumption on $s$, namely that $\tilde s=\pi_{Mp}\circ s:M\rightarrow P\times_{Mp(2n,\R)} Mp(2n, \R))/\hat U_L(n)$ determines the pair $(P^J_L, J)$ we can assume that $d_2\hat g=0$. Since $d\hat s=0$ implies $(d_1\hat g)_1=0$, we arrive at the assertion.\\
Note that for $L=L_0$, the subgroup $\mathfrak{P}_L\cap U(n)\subset Sp(2n, \R)$ equals $O(n)$, so given an element of $p \in P^J_L$ we have for any $X \in T_xM$ so that $\pi_L(p)=x$, where $\pi_L:P_L^J\rightarrow M$ a unique splitting $T_xM=L_1\oplus L_2$ and an isomorphism $\Phi_p:T_xM\rightarrow \R^{2n}$ so that the Lagrangian splitting $T_x M=L_1\oplus L_2$ induced by the $O(n)$-reduction of $P^J$ to $P^J_L$ is mapped under $\Phi_p$ to $\R^n\oplus \R^n$, the standard Lagrangian splitting. We then set $\mathfrak{g}^0_{L,\C}= \mathfrak{u}(2n,\mathbb{R})\cap \mathfrak{p}_L \oplus  \mathfrak{l}_{\C}$, write $X=X_1+X_2, X \in T_xM$ wrt the splitting above and define $w_L \in \Omega^1(M,\underline {\mathfrak{g}^0_{L,\C}})$ for any $p \in T P^J_{L,s}$ by
\[
w_L:= \hat Z_{G,L}^{J,\omega} -\hat Z^0_{G,L}=((p, (g,h)), (0,\sum_{l=1}^{n}\left({\rm Ad}(g^{-1})(h)_l\Phi_p(X)_l+i {\rm Ad}(g^{-1})(h)_{l+n}\Phi_p(X)_{l+n}\right)a_l)
\] 
for $(g,h=(h_1,\dots, h_{2n})) \in G_L$, using (\ref{fouriermu}). It is then easy to verify that $w_L \in \Omega^1(M,\underline {\mathfrak{g}^0_{L,\C}})$ and using (\ref{repduconn}) together with the equality $\hat Z^0_{L,s}=\hat g^*\hat Z^J_{L}$ proven above we see that (\ref{duconn}) defines the Dubrovin connection as defined in (\ref{dubrovin1}) for $z=1$. The assertion $\kappa_{T_0}|\mathfrak{u}(2n,\mathbb{R})=L_*$ from the remark below the proposition follows by rewriting the spanning elements of $\mathfrak{mp}(2n,\mathbb{R})$ of Proposition \ref{diffmp} for the case $\mathfrak{u}(2n,\mathbb{R})$, this is for instance done in \cite{habermann}.
\end{proof}
{\it Remark.} Note that the connection (\ref{duconn}) can be interpreted in some sense as 'half' of a Cartan geometry (cf. Cap/Slovak \cite{cap}) of type $(G,U(n))$ over $M$, since $TM$ is pointwise isomorphic to $H_n$, we hope to pursue this viewpoint in a subsequent paper. The $1$-form $w_0 \in \Omega^1(M,\underline {\mathfrak{g}^0_{L,\C}})$ constructed in the proof above will be in the following referred to occasionally as the 'Higgs field' of the semisimple standard irreducible Frobenius structure $(\Omega, \mathcal{L}, \nabla)$ and the parabolic subgroup $\mathfrak{P}_L\subset Sp(2n, \R)$. Note further that the set of principal bundles $P^J_{L,s}$ and connections $\hat Z_{G_L,s}^J$ (first structure connection) resp. $\hat Z^0_{L,s}$ determining topology and geometry of a semisimple (irreducible, standard) Frobenius structure are essentially contained in the 'universal bundle' $\hat P_G$ resp. its tautological connection $\hat Z^\omega_{G}$, which is why these two objects should be regarded as 'classifying objects' for the respective structures in this situation. Note that a similar, but more complicated discussion as Proposition \ref{higgs} (and its Corollaries below) can be given in the case of indecomposable standard Frobenius structures in the sense of Theorem \ref{genclassN}, this will be done in \cite{kleinlag}. \\
Consider now a given connection $1$-form $Z:T P^J_{L,s}\rightarrow \mathfrak{g}_{L,\C}$ whose curvature $\Omega_{Z}\in \Omega^2(P^J_{L,s},\mathfrak{g}_{L, \C})$, defined in slight extension of the usual notion of curvature for connections $Z:P\rightarrow \mathfrak{g}$ on $G$-bundles $P$, is given by
\[
\Omega_{Z} =d Z + [Z, Z],
\]
where here, $[\cdot, \cdot] \in \Omega^2(P, V)$ for a given $G$-principal bundle and a given vector space $V$ is the usual bracket on $V$-valued $1$-forms on $P$ (cf. \cite{friedrich}), specified to tensorial $1$-forms on $P^J_{L,s}$ with values in $V=\mathfrak{g}_{L,\C}$. Since the curvature forms $\Omega_{\tilde Z}, \Omega_{\hat Z}$ associated to $\tilde Z, \hat Z:=i_*\hat Z^0_{L,s}:T P^J_{L,s}\rightarrow \mathfrak{g}_{L,\C}$, $\tilde Z$ as in Proposition \ref{higgs}, are related by 
\begin{equation}\label{maurercartan}
\Omega_{\tilde Z}- \Omega_{\hat Z}= D_{\hat Z}w_L+\frac{1}{2}[w_L, w_L],
\end{equation}
where $w_L \in \Omega^1(M,\underline {\mathfrak{g}^0_{L, \C}})$ is as defined in the proof of Proposition \ref{higgs} and considered as a tensorial $1$-form with values in $\mathfrak{g}^0_{L,\C}$, thus an element of $\Omega^1(P^J_{L,s}, \mathfrak{g}^0_{L,\C})$ using the isomorphism described above the Proposition, while $D_{\hat Z}w_0=dw_L+ [\hat Z, w_L]$, we have as an immediate Corollary the following. Note that for $P_G$ as in (\ref{tangent}), we denote by $pr_1:P_G\rightarrow T^*M$ the map $pr_1((y,q), x), (p, x))=((gy,q),x),\ x \in M,\ y \in \R^{2n},\ p,q \in P y \in \R^{2n},\ p,q \in P^J, \ q=p.g,\ g \in \hat U(n)$.
\begin{folg}\label{curvature}
Assume, given the assumptions and notations of of Proposition \ref{higgs}, that the curvature $\Omega_{\hat Z^J}$ of $\hat Z^J:P^J\rightarrow \mathfrak{u}(2n,\mathbb{R})$, that is the $\hat U(n)$-reduction of the given symplectic connection $\hat Z:P\rightarrow \mathfrak{sp}(2n,\mathbb{R})$ as defined in Proposition \ref{higgs}, vanishes. Assume furthermore that the section $s:M\rightarrow P_{G/G_L^0}$ defining the semisimple, irreducible, standard Frobenius structure is {\it closed} in the sense that ${\rm pr}_1\circ s:M\rightarrow T^*M$ is closed when using the description (\ref{tangent}) of $\hat P_G\simeq P_G$ and noting that $TM\simeq T^*M$ when considering the $G_L$-reduction $P^J_{L,G}$ of $\hat P_G$ as in (\ref{reductionbla}). Then the same vanishing of the curvature holds for $\Omega_{\tilde Z}$, that is
\[
D_{\hat Z}w_L+\frac{1}{2}[w_L, w_L]=0.
\]
In especially, the Dubrovin connection $\tilde  \nabla$ induced by $\tilde Z$ on the subbundle $\mathcal{L}\subset \mathcal{Q}$ associated to $P_{L,s}^J$ and the given section $s:M\rightarrow P_{G/G_L^0}$ as above (\ref{repduconn}), is flat, that is $(\tilde \nabla)^2 \in \Omega^2(M, End(\mathcal{L}))$ vanishes, so that locally on $M$, there are $\tilde \nabla$-parallel sections of $\mathcal{L}$.
\end{folg}
\begin{proof}
Given the above formulas, the assertion is immediate when considering that $w_L$ takes values in the maximally abelian subspace $\mathfrak{l}_{\C}\subset \mathfrak{g}_{L,\C}$ while $({\rm im} \hat Z)\cap \mathfrak{l}_{\C}=\{0\}$ and under the above assumptions, we have $dw_L=0$.
\end{proof}
{\it Examples.} We finally return to our example in the introduction, at least in its most simple form: given a closed section $l:N\rightarrow T^*N=M$ of a cotangent bundle over a $N$-dimensional manifold $N$, $T^*N$ carrying the canonical symplectic form $\omega$, that is $\mathfrak{l}={\rm im}(l)$ is a Lagrangian submanifold, that is $l^*\omega=0$, we can tautologically consider $l$ as a map $\tilde l:N\rightarrow T^*M|N$ ($T^*M|N$ here means $i^*(T^*M)$ if $i:N\hookrightarrow M$) by considering pointwise $l(x)=(x,p) \in T^*N$ and writing $\tilde l(x)=((x, 0), (p,0))=(\tilde x, \tilde p)\in T^*M$ and extend $\tilde l$ to an open neighbourhood $N \subset U\subset M$ in $M$ to give a closed smooth section $s_l: U\subset M\rightarrow T^*M|U$ so that $\tilde l_U|N=\tilde l$. Then, considering $T^*M$ with its standard symplectic form $\omega_0$, choosing a symplectic connection $\nabla$ on $T^*M|U$ and a compatible almost complex structure $J$ s.t. $\nabla J=0$ on $T^*M|U$, understanding $s_l$ as a closed section of $T^*M$ over $U$ and lifting it to a section $\hat s_l:U\subset M\rightarrow i^*P_{G/G_L^0}\simeq i^*(\pi_P^*(T^*M) \times_{Mp(2n, \R)} P/G_L^0)$ by setting $\hat s_l(x)= ((p,\tilde s_l), p).G_L^0$ for $x \in M$ if $s_l(x)=(p, \tilde s_l(x)),\ \tilde s_l(x) \in \R^{2n},\ p\in P$, where $\hat P_{G/G^0_L}\simeq P_G/G_L^0$. Here, $i:U\subset M$ denotes inclusion and $\hat P_G$ over $M$ is reduced to $G_L^0\simeq \hat U(n)_L$ as in (\ref{reductionbla}) and $L=\R^n\times \{0\}$, so that $\hat U(n)_L\simeq \hat O(n)\subset \hat U(n)$ and the $G_L$- resp. $G_L^0\simeq \hat O(n)$-reductions $P^J_{G,L}$ resp. $P^J_{L,s_l}$ of $i^*\hat P_G$ (notation as above) are fixed by the given almost complex structure on $T^*M$ and the union of the cotangent fibres $V^*M\subset T^*M$ over $U$. Thus we arrive (after possibly homotoping $J$ and $\nabla$ preserving the isomorphy class of $P^J_{G,L}$ as described in Proposition \ref{classi}) at an irreducible standard (in general singular) Frobenius structure
\[
\mathcal{L}=\mathcal{E}_U= P^J_{L,s_l}\times_{G_L^0,\tilde\mu_2\circ i}\mathcal{A}^0_2
\]
over $U\subset M$ by using Theorem \ref{genclass} (using notation from its proof) with first structure connection $\tilde Z:T P^J_{L,s_l}\rightarrow \mathfrak{g}^0_{L, \C}$ as given by Proposition (\ref{higgs}) whose curvature vanishes by Corollary \ref{curvature} if and only if the symplectic connection chosen on $(U\subset M,\omega)$ is flat. Denoting by $i_N:N\hookrightarrow U$ the inclusion, we can consider the pullback $i_N^*\mathcal{L}$ and by using the assignment (\ref{frobeniusm}) one gets a well-defined Frobenius multiplication $\Omega$ of elements of $TN$ resp. ${\rm Sym}^*(TN)$ on $i_N^*\mathcal{L}$. Note that alternatively in the sense of the discussion below Theorem \ref{genclass}, we can understand this Frobenius structure as the image of the section of the bundle $\mathcal{E}_{{\rm Gr}_1(\mathcal{W})}=P^J_L\times_{G_L^0, {\rm ev}_1\circ \tilde \mu_2\circ (i, i_{\mathcal{W}})}{\rm Gr}_1(\mathcal{W})$ given by $s_l$ as described in the proof of Theorem \ref{genclass}, where the implicit embedding is here $i_{\mathcal{W}}:{\rm Gr}_1(\mathcal{W})\rightarrow \mathcal{A}_1$. Using Lemma \ref{eigenvalues}, it then follows immediately that
\[
\mathfrak{l}\simeq {\rm Spec}\left(\frac{\mathcal{O}_{T^*N}}{I_\Omega}\right)
\]
where $I_\Omega$ is the ideal in $\mathcal{O}_{T^*N}$ generated by the characteristic polynomial of $\Omega$, acting on $i_N^*\mathcal{L}$. If the Lagrangian section of $T^*N$ given by $l$ is furthermore {\it exact}, that is if it is the time $1$-image of the zero section of $T^*N$ under a Hamiltonian flow, one can show that $l$ gives rise to a {\it rigid} (in general singular) and {\it self-dual} irreducible standard Frobenius structure, but this will be done in \cite{kleinlag}.\\
On the other hand, to give an example in the sense of Theorem \ref{genclassN}, assume we have a Lagrangian embedding $l:\tilde L\rightarrow T^*N=M$ where $N$ is an $n$-dimensional manifold and $T^*N$ carries the canonical symplectic form, that is $l^*\omega=0$, so that with $L=l(\tilde L)$ $\pi:L\subset T^*N\rightarrow N$ is singular on a connected closed $1$-codimensional submanifold $S\subset L$, that is $S=\{x \in L:{\rm dim }({\rm ker}((d\pi)_x)\cap dl_*(T_{l^{-1}(x)}\tilde L))= 1\}$ is a closed connected codimension-$1$ submanifold of $L$ and the set of all $x \in L$ where  ${\rm ker}((d\pi)_x)\cap dl_*(T_{l^{-1}(x)}\tilde L)>1$ is empty. Denote $S=S_1$ and denote by $S_{1,0}$ the set of points of $S_1$ so that $\pi|S_1$ is non-singular, that is $S_{1,0}=\{x \in S_1: {\rm ker}(d(\pi|S)_x)=0, d(\pi|S):TS\rightarrow TN\}$ and denote by $S_{1,1}\subset S_{1}$ the set of points $x \in S_1$ where ${\rm dim}({\rm ker}(d(\pi|S)_x))=1$. Continuing like this (defining $S_{1,1,0}$ as the set of points in $S_{1,1}$ where $\pi|S_{1,1}$ is non-singular and $S_{1,1,1}$ as those points of $S_{1,1}$ where the rank of its differential drops $1$ etc.) we arrive at the set of Thom-Boardman singularities of $(\pi,L)$ of type $S_{1,1,1,0,\dots}$ which is generically a smooth submanifold of $L$ of codimension $k$ in $L$ (closed in the case of $S_{1,\dots,1}$), where $k$ denotes the number of $1$'s in the notation $S_{1,1,\dots,1,1,0,\dots}\subset L$. As above, we can tautologically consider $l$ as a map $\tilde l:\tilde L\rightarrow T^*M|N$ by considering pointwise $l(z)=(x,p) \in T^*N,\ z \in \tilde L$ and writing $\tilde l(z)=((x, 0), (p,0))=(\tilde x, \tilde p)\in T^*M$ and extend $\tilde l$ to a map $\tilde l^e=\tilde L\times [0,1]^n \rightarrow T^*M$ so that ${\rm im}(\tilde l)=L^e$ is a Lagrangian submanifold of $T^*M$ (with the standard symplectic form $\omega^e$ induced by $\omega$ and $J$ on $M$) and $\tilde l^e|\tilde L\times \{0\}=\tilde l$. In any case we factor $\tilde l$ through a section $s_l: U \subset M \rightarrow T^*M|U$, $ds_l=0$ so that $L\subset U\subset M$ and $\tilde l^e= s_l\circ i_U$, where $i_U:\tilde L\times [0,1]^n\rightarrow M$ is the embedding of a tubular (Darboux-)neighbourhood $U$ of $L$ into $M$, that is $i_U|\tilde L\times \{0\}=l$ and ${\rm im}(i_U)=U$. We can then assume that there are smooth codimension $k$ submanifolds $S_k$ of $L^e$ for any $k \in \{1,\dots, n\}$ so that $S_k\cap {\rm im}(\tilde l)=S_k\cap s_L(L)=s_L(S_{1,1,..0})$ with $k$ $1$'s appearing in $S_{1,1,..0}\subset L$. Furthermore, assuming that the normal bundle $N_k$ of any $S_{1,1,..0} \subset L$ in $TL$ is trivial, we can assume that $i_U$ is chosen so that the normal bundle $N_k^e$ of $s_l^{-1}(S_k)\subset {\rm im}(i_U)$ in $TU$ is also trivial. \\
We are thus in the situation of Theorem \ref{genclassN} and with the notations above this theorem, let $\tilde G= G_{L}^0\subset G$ (with $G_L^0$ as defined above (\ref{reductionbla2})), then we can argue that the closed section $s_l: U \subset M \rightarrow T^*M|U$ defines for any $k \in \{1,\dots,n\}$ (note that $S_k$ may be empty for $k\geq k_0$) a section $\hat s_l:U\cap S_k\subset M\rightarrow i_k^*P_{G/\tilde G_{k}}\simeq i_k^*(\pi_P^*(T^*M) \times_{Mp(2n, \R)} P/\tilde G_{k})$ (where $i_k:S_k\hookrightarrow U\subset M$ and $\tilde G_k\subset \tilde G$ is as defined in the proof of Theorem \ref{genclassN} for any $k \in \{1, \dots, n\}$) by setting $\hat s_l(x)= ((p,\tilde s_l), p).\tilde G_k$ for $x \in M$ if $s_l(x)=(p, \tilde s_l(x)),\ \tilde s_l(x) \in \R^{2n},\ p\in P$, where we use $\hat P_{G/G^0_{L,k}}\simeq P_G/G_{L,k}^0$ with $G_{L,k}^0=\tilde G_k$. We thus associate to $\hat s_l$ the line bundle $\mathcal{E}_U \rightarrow U$ which is given over $S_k\subset U, k \in \{1,\dots,n\}$ as
\[
\mathcal{L}|S_k=\mathcal{E}_U|S_k=P_{\tilde G_k,s}\times_{\tilde G_k, \mu^N_2\circ i}(\mathcal{A}^0_2)^N.
\]
Denoting by $i_L:L\hookrightarrow U\subset M=T^*N$ the inclusion, we can consider the pullback bundle $i_L^*\mathcal{L}\rightarrow L$ and by using the assignment (\ref{frobeniusm}) one gets (again, after possibly homotoping $J$ and $\nabla$ preserving the isomorphy class of $P_{\tilde G_k}$ as indicated in Proposition \ref{classi}) an (in general, singular) standard indecomposable Frobenius structure in the sense of Definition \ref{frobenius}, i.e. a well-defined Frobenius multiplication $\Omega$ of elements of $TL$ resp. ${\rm Sym}^*(TL)$ on $i_L^*\mathcal{L}$. Note that alternatively in the sense of the discussion in the proof of Theorem \ref{genclassN}, we can understand this Frobenius structure as the image of the section of the bundle $\mathcal{E}_{{\rm Gr}_1(\mathcal{W})}|S_k=P_{\tilde G_k}\times_{\tilde G_k, {\rm ev}_1\circ \tilde \mu_2^N\circ (i, i_{\mathcal{W}})}{\rm Gr}_N(\mathcal{W})\rightarrow S_k\subset U$ given by $s_l$ as described in the proof of Theorem \ref{genclassN}, where the implicit embedding is here $i_{\mathcal{W}}:{\rm Gr}_N(\mathcal{W})\rightarrow \mathcal{A}_1^N$. Using Lemma \ref{eigenvalues}, it then follows in the same sense as above that
\[
L \simeq {\rm Spec}\left(\frac{\mathcal{O}_{T^*L}}{I_{\Omega,min}}\right)
\]
where $I_{\Omega,min}$ is the ideal in $\mathcal{O}_{T^*L}$ generated by the minimal polynomial of $\Omega$, acting on $i_L^*\mathcal{L}$, since $L\simeq {\rm Supp}(i_L^*\mathcal{L})$. If the Lagrangian submanifold of $T^*N$ given by $l:\tilde L\hookrightarrow T^*N$ is furthermore {\it exact}, i.e. if it is the time $1$-image of the zero section of $T^*N$ under a Hamiltonian flow, one can show as above that $l$ gives rise to a (weakly, in general singular) {\it rigid} and {\it self-dual} indecomposable standard Frobenius structure, but this will be done in \cite{kleinlag}.

\end{document}